\documentclass[12pt,a4paper,dvipdfmx]{amsart}
\numberwithin{equation}{section}

\usepackage[all]{xy}
\usepackage{amsmath}
\usepackage{amsfonts}
\usepackage{amssymb}
\usepackage{amsthm}
\usepackage{amscd}
\usepackage{tikz}
\usepackage{comment}
\usepackage{xcolor}
\usepackage{fancybox}
\usepackage{ascmac}

\usepackage{MnSymbol}
\usepackage{bm}

\theoremstyle{plain}
\newtheorem{thm}{Theorem}[section]
\newtheorem{cor}[thm]{Corollary}
\newtheorem{lem}[thm]{Lemma}
\newtheorem{prop}[thm]{Proposition}
\newtheorem{conj}[thm]{Conjecture}
\newtheorem*{maintheorem*}{Main Theorem}
\newtheorem*{thm*}{Theorem}

\theoremstyle{definition}

\newtheorem{claim}{Claim}[thm]
\newtheorem*{claim*}{Claim}
\newtheorem*{nc}{Notation and Convention}
\newtheorem{notation}[thm]{Notation}
\newtheorem*{ack}{Acknowledgements}
\newtheorem*{fact*}{Fact}

\theoremstyle{remark}
\newtheorem*{case*}{Case}
\newtheorem*{subcase*}{Subcase}

\newcommand{\mbC}{\mathbb{C}}

\newcommand{\mbG}{\mathbb{G}}
\newcommand{\mbH}{\mathbb{H}}

\newcommand{\mbK}{\mathbb{K}}
\newcommand{\mbL}{\mathbb{L}}

\newcommand{\mbN}{\mathbb{N}}
\newcommand{\mbO}{\mathbb{O}}
\newcommand{\mbP}{\mathbb{P}}

\newcommand{\mbR}{\mathbb{R}}

\newcommand{\mcA}{\mathcal{A}}

\newcommand{\mcD}{\mathcal{D}}
\newcommand{\mcE}{\mathcal{E}}
\newcommand{\mcF}{\mathcal{F}}

\newcommand{\mcH}{\mathcal{H}}
\newcommand{\mcI}{\mathcal{I}}

\newcommand{\mcK}{\mathcal{K}}
\newcommand{\mcL}{\mathcal{L}}
\newcommand{\mcM}{\mathcal{M}}
\newcommand{\mcN}{\mathcal{N}}
\newcommand{\mcO}{\mathcal{O}}
\newcommand{\mcP}{\mathcal{P}}
\newcommand{\mcQ}{\mathcal{Q}}

\newcommand{\mcS}{\mathcal{S}}
\newcommand{\mcT}{\mathcal{T}}
\newcommand{\mcU}{\mathcal{U}}
\newcommand{\mcV}{\mathcal{V}}
\newcommand{\mcW}{\mathcal{W}}
\newcommand{\mcX}{\mathcal{X}}
\newcommand{\mcY}{\mathcal{Y}}
\newcommand{\mcZ}{\mathcal{Z}}

\newcommand{\mrH}{\mathrm{H}}

\newcommand{\Ker}{\mathrm{Ker}}
\newcommand{\Coker}{\mathrm{Coker}}

\newcommand{\id}{\mathrm{id}}

\newcommand{\rank}{\mathrm{rank}}

\newcommand{\mbfS}{\mathbf{S}}
\newcommand{\mbfD}{\mathbf{D}}

\newcommand{\Pn}{\mathbb{P}^{n}}
\newcommand{\GLV}{\mathrm{GL}(V)}

\newcommand{\Supp}{\mathrm{Supp}}
\newcommand{\End}{\mathrm{End}}
\newcommand{\Aut}{\mathrm{Aut}}
\newcommand{\mbKi}{\mathbb{K}^{(i)}}
\newcommand{\wtPd}{\widetilde{\mcP}^{\dagger}}
\newcommand{\tilf}{\tilde{f}}
\newcommand{\tilmcN}{\tilde{\mathcal{N}}}

\newcommand{\barR}{\overline{R}}

\newcommand{\barS}{\overline{\mathbf{S}}}

\newcommand{\Hilb}{\mathrm{Hilb}}
\begin{document}

\title[Lang-Vojta conjecture]{Lang-Vojta conjecture over function fields for very general log projective spaces}
\author[Takeshi ABE]{Takeshi ABE}
\address{Faculty of Science and Engineering, 
Doshisha University, 1-3 Tatara Miyakodani, Kyotanabe-shi 610-0394, Japan}
\email{takabe@mail.doshisha.ac.jp}
\maketitle
\begin{abstract}
We consider  a bound of  degrees of curves by their genera in the log pair $(\Pn\times C, D)$,
where $C$ is a smooth complex projective curve and $D$ is a very general 
smooth divisor on $\Pn\times C$.
Our study is motivated by  a function field analogue of the Lang-Vojta conjecture.
\end{abstract}

%

\section{Introduction and Main result}
In this paper we consider  a bound of  degrees of curves in the log pair $(\Pn\times C, D)$,
where $C$ is a smooth projective curve and $D$ is a smooth divisor on $\Pn\times C$.
Our result is motivated by  a function field analogue of the Lang-Vojta conjecture.

\subsection{Lang-Vojta Conjecture}\label{subsection-LV}

Lang-Vojta conjecture (\cite[Conjecture F.5.3.6]{HS})
predicts that if $X$ is a smooth projective variety over a number field $k$,
$S$ is a finite set of places of $k$ containing all infinite places,
$D$ is a normal crossing divisor on $X$ such that the divisor $K_{X}+D$ is big,
then no set of $S$-integral $k$-rational points on $X\setminus D$ is Zariski dense.

Pursuing an analogy between diophantine problems and Nevanlinna Theory,
Vojta proposed the following inequality : 
\begin{conj}\label{conj-Vojta1} $($\cite[Conjecture 15.6]{Vojta}$)$
Let $X$, $k$ and $S$ be as above.
Let $D$ be a normal crossing divisor on $X$, and let $\mcA$ be an ample line bundle on $X$.
Then for any $\epsilon >0$, there is a proper Zariski-closed subset $Z$ of $X$
such that for all $c\in \mbR$, the inequality
\begin{equation}
N_{S}(D,x)\geq h_{K_{X}(D),k}(x)-h_{\mcA, k}(x)-c
\end{equation}
holds for almost $x\in (X\setminus Z)(k)$.
\end{conj}
Here $N_{S}(D, x)$ is the counting function (see \cite[\S 11]{Vojta}),
$h_{K_{X}(D),k}(x)$ and $h_{\mcA, k}(x)$ are heights of $x$ with respect to $K_{X}(D)$ and $\mcA$ 
respectively,
and ``almost all'' means ``all but finitely many.''

Conjecture \ref{conj-Vojta1} implies Lang-Vojta conjecture (cf. \cite[Proposition 15.9]{Vojta}).
Vojta also proposes a stronger inequality :

\begin{conj}[ \text{[Vojta, Conjecture 25.3 (b)]} ]\label{conj-Vojta2}
Let $X$, $k$, $S$, $D$, and $\mcA$ be as in Conjecture \ref{conj-Vojta1}.
Let $r$ be a positive integer.

For any $\epsilon >0$,
there is a proper Zariski-closed subset $Z$ of $X$
such that for all $c\in \mbR$,
the inequality
\begin{equation*}
N_{S}^{(1)}(D,x)
+d_{k}(x)
\geq
h_{\mcK(D), k}(x)-\epsilon h_{\mcA,k}(x)-c
\end{equation*}
holds for almost all $x\in(X\setminus Z)(\bar{k})$
with $[\kappa(x):k]\leq r$.
\end{conj}
Here $N_{S}^{(1)}(D,x)$ is the $1$-truncated counting function with respect to $D$
(see \cite[\S 23]{Vojta}),
and $d_{k}(x)$ is the logarithmic discriminant (see \cite[\S 24]{Vojta}).

\subsection{Function field analogue of Lang-Vojta conjecture}\label{subsection-FunctLV}
We consider a function field analogue of Lang-Vojta conjecture.
We here give a formulation using a model over a curve.

Let $F$ be an algebraically closed field of characteristic zero.
Let $C$ be a smooth projective curve over $F$,
let $\mcX$ be a smooth projective variety over $F$,
and let $\pi:\mcX \to C$ be a surjective morphism with connected fibers.
Let $D$ be a normal crossing divisor on $\mcX$,
and let $S$ be a finite set of closed points of $C$.
\begin{conj}\label{conj-FunctLV}
Suppose that $K_{\mcX}+D$ is $\pi$-big.
Then there is a proper Zariski-closed subset $Z$ of $\mcX$
such that the family of curves  $\tau(C)$ on $\mcX$,
where $\tau:C\to \mcX$ is a section of $\pi$ such that
$\tau^{-1}(D)\subset S$ and $\tau(C)\not \subset Z$,
is bounded.
\end{conj}
Here a family of curves on $\mcX$ is bounded if
the degrees (with respect to an ample line bundle on $\mcX$)  of the curves in the family are bounded.

We can also consider a function field analogue of Conjecture \ref{conj-Vojta2}.
For an irreducible curve $Y$ on $\mcX$ mapping onto $C$ by $\pi$,
and not contained in $D$,
define $N^{(1)}_{S}(D,Y)$ and $d_{C}(Y)$ as follows.
Let $\nu_{Y}:\widetilde{Y}\to Y$ be the normalization.
The set-theoretic inverse $\nu_{Y}^{-1}(D\cap Y)$ is a finite set of points of $\widetilde{Y}$,
and define
\[
N^{(1)}_{S}(D,Y)
:=\frac{1}{\deg (\pi|_{Y})\circ \nu_{Y}}
\left|
\nu_{Y}^{-1}(D\cap Y) 
\cap
\left( \widetilde{Y}\setminus ((\pi|_{Y})\circ \nu_{Y})^{-1}(S) \right)
\right|.
\]
So $N^{(1)}_{S}(D,Y)$ counts the number of points of $\nu_{Y}^{-1}(D\cap Y) $
contained in the open subset $\widetilde{Y}\setminus ((\pi|_{Y})\circ \nu_{Y})^{-1}(S)$,
divided by the degree of $(\pi|_{Y})\circ \nu_{Y}$.
The logarithmic discriminant term in the function field case is defined as  (cf. \cite[\S 28]{Vojta})
\[
d_{C}(Y):=\frac{1}{\deg (\pi|_{Y})\circ \nu_{Y}}
(2g(\widetilde{Y})-2) -(2g(C)-2),
\]
where $g(\widetilde{Y})$ and $g(C)$ are the genera of $\widetilde{Y}$ and $C$ respectively.
For a line bundle $\mcM$ on $\mcX$, define
\[
h_{\mcM}(Y):=
\frac{1}{\deg (\pi|_{Y})\circ \nu_{Y}}
\deg\nu_{Y}^{*}(\mcM|_{Y}).
\]

The following ia a  function field analogue of Conjecture \ref{conj-Vojta2}.
\begin{conj}\label{conj-ffVojta2}
Let $\mcA$ be an ample line bundle on $\mcX$,
and let $r$ be a positive integer.
For any $\epsilon>0$,
there is a proper Zariski closed subset $Z$ of $\mcX$ such that
for all $c\in \mbR$,
the inequality
\begin{equation}\label{eq-ffVojta}
N_{S}^{(1)}(D,Y)+d_{C}(Y)
\geq
h_{\mcK(D)}(Y)-\epsilon h_{\mcA}(Y)-c
\end{equation}
holds for all irreducible curves $Y$ on $\mcX$, 
not contained in $Z\cup D$, such that $\deg (\pi|_{Y})\circ\nu_{Y} \leq r$,
except a bounded family of curves  on $\mcX$.
\end{conj}
Let $\eta\in C$ be the generic point of the curve $C$,
and let $\mcX_{\eta}$ and $D_{\eta}$ be the fiber of $\pi : \mcX\to C$ and
$\pi|_{D}:D\to C$ over $\eta$ respectively.

The function field analogue of Lang-Vojta conjectures
was studied by several people : 
Corvaja and Zannier (\cite{CZ08}) and Turchet (\cite{T})
considered the case
when $\mcX_{\eta}=\mbP^{2}_{\eta}$ and $D_{\eta}$ is a quartic 
consisting of two lines and a conic.
Corvaja and Zannier (\cite{CZ13}) and Capuano and Turchet (\cite{CT})
treats the case when $\mcX_{\eta}\setminus D_{\eta}$ is a ramified cover of $\mbG_{m,\eta}^{2}$.
Guo, Nguyen, Sun and Wang studied the case when 
$\mcX_{\eta}\setminus D_{\eta}$ is a ramified cover of $\mbG_{m,\eta}^{n}$.
They also considered the inequality (\ref{eq-ffVojta}) when $\mcX_{\eta}$ is a toric variety
(see \cite[Theorem 1]{GNSW} for the precise statement).

We also note that 
Yamanoi (\cite[Theorem 5]{Y}) proves
Conjecture \ref{conj-ffVojta2}  when $\mcX_{\eta}$ is a curve.

\subsection{Main result}\label{subsection-MT}

From now on we let $F=\mbC$, the field of complex numbers.
In this paper, we consider the case $\mcX=\Pn\times C$ and
the divisor $D$ is ``very general'' in the sense clarified below.

Let $V:=\mbC^{n+1}$ and $\Pn:=\mbP(V)$.
Let $C$ be a smooth projective curve over $\mbC$,
and $\mcL$ a line bundle on $C$ generated by global sections.
Let $W \subset\mrH^{0}(C,\mcL)$ be a subspace generating $\mcL$,
that is, the natural map $W\otimes\mcO_{C}\to \mcL$ is surjective.
Put $\Sigma:=\mbP_{*}\left( \mrH^{0}(\Pn, \mcO(d)) \otimes W\right)$
(see Notation and Convention below for the notation $\mbP_{*}$).
Since $\mrH^{0}(\Pn, \mcO(d)) \otimes W
\subset \mrH^{0}\left(
 \Pn\times C,  \mcO(d) \boxtimes \mcL
\right)$,
each point $\sigma\in \Sigma$ determines a non-zero global section, up to scalar,  of the line bundle
$\mcO(d)\boxtimes \mcL$ on $\Pn\times C$, 
and we denote  its divisor of zeros
 by $D_{\sigma}\subset  \Pn\times C$.
The fibers of the projection $pr_{C}: \Pn \times C \to C$ are isomorphic to $\Pn$,
and curves in the fibers are said to be  vertical.
For $\sigma\in \Sigma$, 
we denote by $\mathbf{E}_{\sigma}\subset \Pn\times C$
the Zariski closure of the union of vertical lines intersecting
$D_{\sigma}$ in at most two points.
You can see that if $d>n+1$ and 
the point $\sigma\in \Sigma$ is general, then $\mathbf{E}_{\sigma}\subsetneq \Pn\times C$.
Let $S$ be a finite set  of (closed) points of $C$.

\begin{thm}\label{thm-main}
Assume that $d>\frac{3n+3}{2}$.
Let $\sigma\in \Sigma$ be a very general point,
i.e. $\sigma \in \Sigma\setminus (\text{countable union of proper Zariski-closed subsets of $\Sigma$})$.
Then for each non-vertical curve $Y$, not contained in $D_{\sigma}$,  on $\Pn\times C$
with its normalization $\nu: \widetilde{Y}\to Y \subset  \Pn\times C$,
either (i) or (ii) of the following holds:
\begin{enumerate}
 \item[(i)] The inequality
\begin{equation*}
\begin{split}
\frac{1}{e}\deg(pr_{\Pn}\circ\nu)^{*}\mcO_{\Pn}(1) & 
\leq \frac{1}{e}
\left\{
2g(\widetilde{Y})-2
+\left|
\nu^{-1}(D_{\sigma})\setminus
(pr_{C}\circ \nu)^{-1}(S)
\right|
\right\} \\
&-(2g(C)-2)+c_{0}
\end{split}
\end{equation*}
holds, where $e$ is the degree of the finite map $Y\to C$, 
and $c_{0}=l(n-1)+\max\{0,2g(C)-2+|S|-l\}$;
 \item[(ii)] The curve $Y$ lies in $\mathbf{E}_{\sigma}$.
\end{enumerate}
\end{thm}
From this theorem, we can obtain an inequality which is
similar to but weaker than (\ref{eq-ffVojta}).

\begin{cor}\label{cor-ofmain}
We retain the notation described before Theorem \ref{thm-main}.
We also use the notation described before Conjecture \ref{conj-ffVojta2}.
Let $\mcA$ be an ample line bundle on $\Pn\times C$,
and let $\mcK$ be the canonical line bundle of $\Pn\times C$.
Let $r$ be a positive integer.
Assume that $d>\frac{3n+3}{2}$.
Let $\sigma\in \Sigma$ be a very general point.
\begin{enumerate}
\item[(1)]
For any $\epsilon >0$ and $c\in \mbR$,
the inequality
\begin{equation}\label{eq-corofmain}
N_{S}^{(1)}(D_{\sigma},Y)+d_{C}(Y)
\geq
\frac{1}{d-n-1}h_{\mcK(D_{\sigma})}(Y)-\epsilon h_{\mcA}(Y)-c
\end{equation}
holds for all irreducible non-vertical curves $Y$ on $\Pn\times C$, 
not contained in $\mathbf{E}_{\sigma}\cup D_{\sigma}$, such that $\deg (\pi|_{Y})\circ\nu_{Y} \leq r$,
except a bounded familiy of curves on $\Pn \times C$.
\item[(2)]
The family of curves $\tau(C)\subset \Pn\times C$,
where $\tau:C\to \Pn\times C$ is a section of $pr_{C}:\Pn\times C \to C$ such that
$\tau^{-1}(D_{\sigma})\subset S$ and $\tau(C)\not\subset\mathbf{E}_{\sigma}$,
is bounded.
\end{enumerate}
\end{cor}

\subsection{Outline of Proof of Theorem \ref{thm-main}}
In \cite{A24}, the author studied subvarieties of very general log projective spaces
employing a method, which uses the positivity of normal bundles,
invented by Ein (\cite{E88}, \cite{E91}), and then developed by Voisin (\cite{V96}, \cite{V98}),
Pacienza (\cite{P03}, \cite{P04}), Clemens and Ran (\cite{CR}) and Pacienza and Rousseau (\cite{PR}).

The proof of Theorem \ref{thm-main} is similar to that of \cite[Theorem 8.1]{A24};
we only repeat the argument in \cite{A24} relatively over the curve $C$.
Specifically,
if a curve $Y\subset \Pn\times C$ not contained in $D_{\sigma}$ does not satisfy the inequality
in Theorem \ref{thm-main} (i), 
then we show that for a general point $y\in Y$,
there exists a vertical line $l_{y}$ passing through $y$ that intersects $D_{\sigma}$
in at most two points.
To this end, we first show that to a general points $y\in Y$,
we can associate a vertical line $l_{y}$ passing through $y$ in such a way that
for general two points $y$ and $y'\in Y$,
the divisors $\left. D_{\sigma}\right|_{l_{y}}$ on the pointed line $(l_{y}, y)$ and
$\left. D_{\sigma}\right|_{l_{y'}}$ on the pointed line $(l_{y'}, y')$
are equivalent,
that is,
there exists an isomorphism $l_{y}\simeq l_{y'}$ of lines
mapping $y$ to $y'$ such that $\left. D_{\sigma}\right|_{l_{y}}$ and
$\left. D_{\sigma}\right|_{l_{y'}}$ correspond under the isomorphism.
Then using a moduli space of lines intersecting $D_{\sigma}$ 
``in a fixed divisor,''
we show that if $l_{y}$ intersects $D_{\sigma}$ in more than two points,
then $Y$ satisfies the inequality in Theorem \ref{thm-main} (i).

\subsection{Organization of the paper}
In Section \ref{section-preliminaries}, we recall some facts about positivity
of vector bundles and log tangent bundles.
In Section \ref{section-setting}, we state Theorem \ref{thm-core}, the core theorem of this paper,
and Proposition \ref{prop-half}.
In Section \ref{section-proofhalf},
we prove Proposition \ref{prop-half}.
Section \ref{sect-lines}
is almost a reproduction of \cite[\S 5]{A24},
in which we introduce moduli of lines intersecting a hypersurface in a fixed divisor.
In Section \ref{section-proofcore},
we prove Theorem \ref{thm-core}.
In Section \ref{section-proofofmain},
we show that Theorem \ref{thm-core}
implies Theorem \ref{thm-main},
and also prove Corollary \ref{cor-ofmain}.
In Section \ref{section-Remarks},
we give two remarks;
the first one explains why vertical lines intersecting the divisor $D_{\sigma}$ in two points appear as the exceptional locus.
The second one concerns the hypersurface version of the core theorem.

\begin{ack}
This work was supported by JSPS KAKENHI Grant Number JP22K03232 and JP26K06752.
\end{ack}

\begin{nc}
\begin{itemize}
\item The geometric genus of a projective curve $Y$ is denoted by $g(Y)$.
\item For a finite dimensional vector space $W$, we put $\mbP(W):=\mathrm{Proj}\left(\bigoplus_{d=0}^{\infty}S^{d}W\right)$
and $\mbP_{*}(W):=\mathrm{Proj}\left(\bigoplus_{d=0}^{\infty}S^{d}W^{*}\right).$
So $\mbP(W)$ parametrizes codimension one subspaces of $W$, and $\mbP_{*}(W)$ parametrizes one dimensional subspaces of $W$.
Likewise, for a locally free sheaf $\mcW$ on a scheme $Z$, 
we put $\mbP(\mcW):=\mathrm{Proj}\left(\bigoplus_{d=0}^{\infty}S^{d}\mcW\right)$
and $\mbP_{*}(\mcW):=\mathrm{Proj}\left(\bigoplus_{d=0}^{\infty}S^{d}\mcW^{*}\right).$
\item We use Weil divisors and Cartier divisors on a smooth scheme interchangeably.
An effective divisor $D=\sum_{i} m_{i}D_{i}$, $m_{i}\geq 1$ and $D_{i}$ distinct prime divisors,
is called \textit{reduced} if $m_{i}=1$ for all $i$;
and for the above $D$, we put $D_{\mathrm{red}}:=\sum_{i} D_{i}$.
For an effective divisor $D$, the subscheme defined by the ideal sheaf $\mcO(-D)$ is also denoted by the same letter $D$.
In particular, we will not distinguish a reduced effective divisor $D$ and the closed subset $\mathrm{Supp}\, D$, the support of $D$.

If $f:Z'\to Z$ is a morphism of smooth varieties, and $D$ is an effective divisor on $Z$ with $f(Z')\not\subset \mathrm{Supp}\, D$,
the pull-back, as a divisor, of $D$ by $f$ is denoted by $f^{*}D$.
When $D$ is a reduced divisor, we use the notation $f^{-1}(D)$ to mean $(f^{*}D)_{\mathrm{red}}$,
in other words, $f^{-1}(D)$ is the set-theoretic inverse of $D$ by $f$.

For reduced divisors $D$ and $E$, we write $D\cup E$ for $(D+E)_{\mathrm{red}}$,
that is, their set-theoretic union. 

\item Throughout the paper, we fix an $(n+1)$-dimensional vector space $V:=\mbC^{n+1}$,
and a projective space $\Pn=\mbP(V)$.

\item For a point $z\in \Pn$ and an integer $d\geq 0$,
we write $M_{z}^{d}$ for $\mrH^{0}(\Pn, \mcI_{z}(d))$.
For a line $l\subset\Pn$, we write $M_{l}^{d}$ for $\mrH^{0}(\Pn, \mcI_{l}(d))$.

\end{itemize}
\end{nc}

\section{Preliminaries}\label{section-preliminaries}

\subsection{Positivity}\label{subsect-posi}

We collect some facts about positivity of vector bundles.
Our reference is \cite[\S2]{CR}.

Let $Z$ be a scheme over $\mbC$, and let $\mcA$ be a vector bundle generated by global sections on $Z$.
Let $R\subset \mrH^{0}(Z,\mcA)$ be a subspace such that the natural morphism of sheaves
\[
R\otimes\mcO_{Z} \to \mcA.
\]
is surjective.
Denote by $\mcM_{Z}^{(\mcA;R)}$ the kernel of this morphism

\begin{prop}\label{prop-posi}
For an integer $i\geq 0$, 
the vector bundle $\wedge^{i}\mcM_{Z}^{(\mcA;R)}\otimes \det\mcA$ is generated by global sections.
\end{prop}
\proof
This follows from the existence of 
a surjective morphism
\[
\wedge^{i+\rank\;\mcA}\left(
R\otimes\mcO_{Z}
\right)
\to \wedge^{i}\mcM_{Z}^{(\mcA;R)}
\otimes\det \mcA.
\]
\endproof

Let $\mbG$ be the Grassmannian variety parametrizing lines in $\Pn=\mbP(V)$.
Equivalently, $\mbG$ parametrizes $2$-dimensional quotients of $V$.
Let 
\begin{equation}\label{eq-univG}
0\to \mcK \to V\otimes \mcO_{\mbG}\to \mcQ \to 0
\end{equation}
be the natural sequence with $\mcQ$ the universal rank $2$ quotient.
Put $\mcO_{\mbG}(1):=\det \mcQ$, which is the ample generator of $\mathrm{Pic}\,\mbG$.

\begin{notation}\label{notation-M}
For an integer $j\geq 1$,
the vector bundles $\mcM_{\Pn}^{\left(\mcO_{\Pn}(j); S^{j}V\right)}$ 
and $\mcM_{\mbG}^{\left(S^{j}\mcQ; S^{j}V\right)}$ 
are denoted by $\mcM_{\Pn}^{j}$ and $\mcM_{\mbG}^{j}$ respectively.
\end{notation}

\begin{prop}\label{prop-Mposi}
For an integer $j\geq 1$, the vector bundles $\mcM^{j}_{\Pn}\otimes \mcO_{\Pn}(1)$ on $\Pn$
and $\mcM^{j}_{\mbG}\otimes \mcO_{\mbG}(1)$ on $\mbG$ are generated by global sections.
\end{prop}
\proof
When $j=1$, the result follows from Proposition \ref{prop-posi}.
When $j\geq 2$,
there are surjective maps $S^{j-1}V\otimes \mcM_{\Pn}^{1}\to\mcM^{j}_{\Pn}$
and $S^{j-1}V\otimes \mcM_{\mbG}^{1}\to\mcM^{j}_{\mbG}$,
so the result follows from  the case $j=1$.
\endproof

\subsection{Log tangent bundle}\label{subsect-logtan}

Let $D$ be a SNC (=simple normal crossing) divisor on a smooth variety $Z$.
The log tangent bundle $T_{Z}(-\log D)$ of the log pair $(Z,D)$ is defined by the exact sequence
\[
0\to T_{Z}(-\log D) \to T_{Z} \to \mcH om_{\mcO_{Z}}(\mcO(-D), \mcO_{D}). 
\]
If $D$ is defined locally  by $z_{1}\cdots z_{r}=0$  with analytic coordinates $z_{1}, \dots, z_{n}$,
then $T_{Z}(-\log D)$ is generated by $z_{i}\partial/\partial z_{i}$ ($1\leq i \leq r$)
and $\partial/\partial z_{j}$ ($r+1\leq j \leq n$)
as an $\mcO_{Z}$-module.

If $Y$ is another smooth variety with a SNC divisor $E$,
and $f$ is a morphism of $Z$ to $Y$ with $D=f^{-1}(E)$,
then the natural map 
$T_{Z}\to f^{*}T_{Y}$ of tangent bundles
induces a map
\begin{equation}\label{eq-logtan}
T_{Z}(-\log D)\to f^{*}T_{Y}(-\log E)
\end{equation}
 of log tangent bundles.
When the morphism $f$ is generically finite,
the map (\ref{eq-logtan}) is an injective morphism of $\mcO_{Z}$-modules,
and we define the log normal sheaf $\mcN^{\log}_{f}$ of $f$ to be the cokernel of (\ref{eq-logtan}).

Let $\pi:\mcZ\to U$ be a smooth morphism of smooth varieties.
A SNC divisor $\mcD$ on $\mcZ$ is said to be \textit{relatively SNC over $U$}
if any intersection of  irreducible components of $\mcD$ is smooth over $U$.
In this case, the restriction $D_{u}$ of $\mcD$ to the fiber $Z_{u}:=\pi^{-1}(u)$ over $u\in U$ is also
a SNC divisor.
We define the relative log tangent bundle $T_{\mcZ/U}(-\log \mcD)$
to be
\[
T_{\mcZ}(-\log \mcD) \cap T_{\mcZ/U},
\]
where the intersection is taken in $T_{\mcZ}$.
Then we have 
\begin{equation}
\left. T_{\mcZ/U}(-\log \mcD)\right|_{Z_{u}}=T_{Z_{u}}(-\log D_{u}).
\end{equation}
The short exact sequence
\[
0 \to T_{\mcZ/U} \to T_{\mcZ} \to \pi^{*}T_{U} \to 0
\]
induces a short exact sequence
\begin{equation}\label{eq-reltan}
0 \to T_{\mcZ/U}(-\log \mcD) \to T_{\mcZ}(-\log \mcD) \to \pi^{*}T_{U} \to 0
\end{equation}
because $\mcD$ is relatively SNC over $U$.
\begin{prop}\label{prop-reltan}
Let $U$ be a smooth variety, $\mcX$ and $\mcY$  varieties smooth over $U$,
$f:\mcX\to \mcY$ a morphism over $U$
and $\mcE$ a SNC divisor on $\mcY$  relatively SNC over $U$
such that $\mcD:=f^{-1}(\mcE)$ is a SNC divisor on $X$.

\noindent (1) If $\mcE$ is smooth, then there exists a short exact sequence
\[
0 \to T_{\mcY/U}(-\log \mcE) \to T_{\mcY/U} \to \left. \mcO_{\mcY}(\mcE)\right|_{\mcE} \to 0.
\]

\noindent (2) If the morphism $f$ is smooth, then $\mcD$ is relatively SNC over $U$,
and there is a short exact sequence
\[
0 \to T_{\mcX/\mcY} \to T_{\mcX/U}(-\log \mcD)
\to f^{*}T_{\mcY/U}(-\log \mcE) \to 0.
\]
If, moreover, $\mcF$ is a SNC divisor on $\mcX$ relatively SNC over $\mcY$
such that $\mcF+\mcD$ is SNC and relatively SNC over $U$,
then there is a short exact sequence
\[
0 \to T_{\mcX/\mcY}(-\log \mcF) \to T_{\mcX/U}(-\log (\mcD+\mcF))
\to f^{*}T_{\mcY/U}(-\log \mcE) \to 0.
\]

\noindent (3)
If the morphism $f$ is generically finite and $\mcD$ is relatively SNC over $S$,
then there is a short exact sequence
\[
0\to T_{\mcX/U}(-\log \mcD) \to f^{*}T_{\mcY/U}(-\log \mcE) \to \mcN_{f}^{\log} \to 0.
\]
\end{prop}
\proof
The proof of (1) and (2) are immediate,
and (3) follows from the definition of $\mcN_{f}^{\log}$ and  the short exact sequence (\ref{eq-reltan}).
\endproof
\begin{prop}\label{prop-TPU}
Let $\mcU$ be a vector bundle on a smooth variety $Z$,
and let $\mcM\subset \mcU$ be a vector subbundle with $\mcQ:=\mcU/\mcM$ a line bundle.
So we have an inclusion  $\mbP_{*}(\mcM)\hookrightarrow\mbP_{*}(\mcU)$ of projective bundles over $Z$,
and $\mbP_{*}(\mcM)$ is a divisor of $\mbP_{*}(\mcU)$.
Denote by $\pi$ the projection $\mbP_{*}(\mcU)\to Z$.
Then the relative log tangent bundle $T_{\mbP_{*}(\mcU)/Z}(-\log \mbP_{*}(\mcM))$
is isomorphic to $\mcO_{\mbP_{*}(\mcU)}(1)\otimes \pi^{*}\mcM$.
\end{prop}
\proof
From the Euler sequences on $\mbP_{*}(\mcU)$ and $\mbP_{*}(\mcM)$,
we obtain the following commutative diagram of sheaves on $\mbP_{*}(\mcM)$ : 
\begin{equation*}
\xymatrix{
0 \ar[r] & \mcO_{\mbP_{*}(\mcM)}(-1) \ar[r] \ar@{=}[d] 
& \left.\pi^{*}\mcM\right|_{\mbP_{*}(\mcM)} \ar[r] \ar[d] 
&T_{\mbP_{*}(\mcM)/Z}(-1) \ar[d] \ar[r] & 0 \\
0 \ar[r] & \mcO_{\mbP_{*}(\mcM)}(-1) \ar[r] 
& \left.\pi^{*}\mcU\right|_{\mbP_{*}(\mcM)} \ar[r] 
&\left.T_{\mbP_{*}(\mcU)/Z}(-1)\right|_{\mbP_{*}(\mcM)} \ar[r] & 0. 
}
\end{equation*}
By the snake lemma, we have an isomorphism
\begin{equation}\label{eq-QN}
\left.\pi^{*}\mcQ\right|_{\mbP_{*}(\mcM)} \stackrel{\sim}{\to} N_{\mbP_{*}(\mcM)/\mbP_{*}(\mcU)}(-1).
\end{equation}
Let $\mcK:=\Ker\left(
\pi^{*}\mcU\to \left.\pi^{*}Q\right|_{\mbP_{*}(\mcM)}
\right).$
Then $\pi^{*}\mcM \subset \mcK$, and there is an exact sequence
\begin{equation}\label{eq-MKQ}
0 \to \pi^{*}\mcM \to \mcK
\to \pi^{*}Q\left(-\mbP_{*}(\mcM)\right) \to 0.
\end{equation}
The subsheaf $\mcO_{\mbP_{*}(\mcU)}(-1) \subset \pi^{*}\mcU$ 
is a subsheaf of $\mcK$, and the composite
$\mcO_{\mbP_{*}(\mcU)}(-1) \to\mcK\to \pi^{*}Q\left(-\mbP_{*}(\mcM)\right)$
is an isomorphism. 
So the exact sequence (\ref{eq-MKQ}) splits and we have 
\begin{equation}\label{eq-KMO-1}
\mcK=\pi^{*}\mcM\oplus \mcO_{\mbP_{*}(\mcU)}(-1).
\end{equation}
By the definition of $\mcK$, there is a commutative diagram
\[
\xymatrix{
0 \ar[r] & \mcK \ar[r] \ar[d] & \pi^{*}\mcU \ar[r] \ar[d] 
& \left.\pi^{*}\mcQ\right|_{\mbP_{*}(\mcM)} \ar[r] \ar[d] & 0 \\
0 \ar[r] &T_{\mbP_{*}(\mcU)/Z}\left(-\log \mbP_{*}(\mcM)\right)(-1) \ar[r] & T_{\mbP_{*}(\mcU)/Z} (-1) \ar[r]
& N_{\mbP_{*}(\mcM)/\mbP_{*}(\mcU)}(-1) \ar[r] & 0.
}
\]
Since the right vertical arrow is an isomorphism by (\ref{eq-QN}),
we have, by the snake lemma,
an exact sequence
\begin{equation}\label{eq-OKT}
0\to \mcO_{\mbP_{*}(\mcU)}(-1) \to \mcK
\to T_{\mbP_{*}(\mcU)/Z}\left(-\log \mbP_{*}(\mcM)\right)(-1) \to 0. 
\end{equation}
The desired isomorphism follows from (\ref{eq-KMO-1}) and (\ref{eq-OKT}).
\endproof

\section{Setting}\label{section-setting}

We retain the notation used in Section \ref{subsection-MT}.
But we  put $X:=\Pn\times C$ (instead of $\mcX$).

As in Theorem \ref{thm-main},
we fix a finite subset $S$ of $C$.
Put $\mcS:=\Pn\times S$, which is a reduced divisor on $X$.

Let $\Sigma^{sm}$ be the Zariski open subscheme of $\Sigma$
consisting of points $\sigma\in \Sigma$ such that the corresponding divisors
$D_{\sigma}\subset X$ are smooth and
the divisors $\mcS + D_{\sigma}\subset X$ are SNC.
Let $\mcD\subset X \times\Sigma^{sm}$ be the universal family of divisors $D_{\sigma}$
parametrized by $\Sigma^{sm}$.
We denote by $p$ the projection $X \times\Sigma^{sm}\to \Sigma^{sm}$,
and by $h$ the projection $X \times\Sigma^{sm}\to X$.

The general linear group $\GLV$ naturally acts on $\Pn=\mbP(V)$,
and by letting $\GLV$ act on trivially on $C$, we obtain a $\GLV$-action on $X(=\Pn\times C)$.
We have a natural induced action of $\GLV$ on  $\Sigma^{sm}$.
We also obtain induced $\GLV$-actions on $X\times\Sigma^{sm}$ and $\mcD$.
With these actions,   the projections $p$ and $h$ are $\GLV$-equivariant.

Let $\Sigma'$ be a smooth variety with a $\GLV$-action, and $\beta:\Sigma'\to \Sigma^{sm}$ a $\GLV$-equivariant smooth morphism.
Put $\mcD':=\mcD\times_{\Sigma^{sm}}\Sigma'\subset X\times\Sigma'$,
the base change of $\mcD$ by the morphism $\beta$.
Put $\beta':=\id_{X}\times \beta:X \times\Sigma'\to X\times\Sigma^{sm}$.
Let $\mcY$ be a smooth variety with a $\GLV$-action,  
and $f:\mcY\to X\times \Sigma'$ be a $\GLV$-equivariant finite morphism
such that the composite $p_{1}:=p'\circ f:\mcY\to \Sigma'$ is a smooth projective morphism
with connected $1$-dimensional fibers,
where $p':X \times \Sigma'\to \Sigma'$ is a projection.
We assume that the image of $f$ is not contained in $\mcD'$.
Then set-theoretic inverse image $\mcD_{\mcY}:=f^{-1}(\mcD')\subset \mcY$ 
of $\mcD'$ to $\mcY$  is a divisor on $\mcY$, and we assume it is \'{e}tale over $\Sigma'$.
The situation is summarized in the diagram (\ref{eq-setting}).
\begin{equation}\label{eq-setting}
\xymatrix@R=15pt@C=33pt{
\mcD_{\mcY}  \ar@{}[d]|{\bigcap} &  \mcD'  \ar@{}[d]|{\bigcap}     &  \mcD  \ar@{}[d]|{\bigcap}   &   &      \\
\mcY  \ar[r]^-{f}  \ar[rd]_-{p_{1}}  &  X\times\Sigma'  \ar[r]^-{\beta'} \ar[d]^-{p'}
& X\times\Sigma^{sm} \ar[d]^-{p} \ar[r]^-{h} &  X=\Pn\times C \ar[r]^-{pr_{\Pn}} \ar[rd]_-{pr_{C}} & \Pn \\
                    & \Sigma'  \ar[r]^-{\beta}                  & \Sigma^{sm}                       &   & C
}
\end{equation}
For a point $\sigma'\in \Sigma'$, we denote by $Y_{\sigma'}$ the fiber $p_{1}^{-1}(\sigma')$ of $p_{1}$ over $\sigma'$,
and put $D_{Y_{\sigma'}}:=\mcD_{\mcY}\mid_{Y_{\sigma'}}$
and $D'_{\sigma'}:=\mcD'|_{X\times\{\sigma'\} }$.
The restriction of the morphism $f$ to the fiber over $\sigma'$
is denoted by $f_{\sigma'}:Y_{\sigma'}\to X(=X\times\{\sigma'\} )$.

We assume that for all $\sigma'\in \Sigma'$, the composite $pr_{C}\circ f_{\sigma'}:Y_{\sigma'}\to C$ is a finite morphism,
and we denote its degree by $\deg(Y_{\sigma'}/C)$.
The  line bundle $(pr_{\Pn}\circ f_{\sigma'})^{*}\mcO_{\Pn}(1)$ on $Y_{\sigma'}$
is denoted shortly by $\mcO_{\Pn}(1)|_{Y_{\sigma'}}$.

Put 
$\mcS_{\Sigma^{sm}}:=\mcS\times\Sigma^{sm}$,
$\mcS_{\Sigma'}:=\mcS\times\Sigma'$
and $\mcS_{\mcY}:=f^{-1}(\mcS_{\Sigma'})$.
The divisor $\mcD'+ \mcS_{\Sigma'}$ on $X\times\Sigma'$ is SNC and relatively SNC over $\Sigma'$.
We assume that the diviosr $\mcD_{\mcY}\cup \mcS_{\mcY}(=(\mcD_{\mcY}+ \mcS_{\mcY})_{\mathrm{red}}$) on $\mcY$ is SNC
and relatively SNC over $\Sigma'$.
We write $\mcS_{Y_{\sigma'}}$ for $\mcS_{\mcY}|_{Y_{\sigma'}}$.

With this setting, we have the following theorem.
\begin{thm}\label{thm-core}
Assume that $d>\frac{3n+3}{2}$.
Then either (i) or (ii) of the following holds:
\begin{enumerate}
 \item[(i)] For general $\sigma'\in\Sigma'$, the inequality
\begin{equation}\label{eq-ineqcore}
\frac{1}{e}\deg\mcO_{\Pn}(1)|_{Y_{\sigma'}} 
 \leq \frac{1}{e}
\left\{
2g(Y_{\sigma'})-2
+\left|
f_{\sigma'}^{-1}(D'_{\sigma'}\setminus \mcS)
\right|
\right\} 
 -(2g(C)-2)+c_{0}
\end{equation}
holds, where $e=\deg(Y_{\sigma'}/C)$ and $c_{0}=l(n-1)+\max\{0,2g(C)-2+|S|-l\}$.
 \item[(ii)] For a general $y\in \mcY$, 
denoting the point $f(y)\in X\times\Sigma'$ by $(x,\sigma')$,
we can find a canonically determined vertical line $l_{y}$ in $X(=\Pn\times C)$
such that $l_{y}$ passes through the point $x$ and intersects $D'_{\sigma'}$ in at most two points.
\end{enumerate}
\end{thm}

To prove Theorem \ref{thm-core},
we first prove Proposition \ref{prop-half} below.
To state the proposition, we define equivalence of divisors on pointed lines.
Let $(l_{i}, z_{i})$, $i=1,2$, be pointed lines,
and $D_{i}$ be a divisor on $l_{i}$.
We say that the divisors $D_{1}$ on the pointed line $(l_{1},z_{1})$
and $D_{2}$ on $(l_{2},z_{2})$ are equivalent
if there is an isomorphism $\varphi:l_{1}\to l_{2}$ of lines
such that $\varphi(z_{1})=z_{2}$ and $\varphi(D_{1})=D_{2}$.

\begin{prop}\label{prop-half}
In the situation described before Theorem \ref{thm-core},
assume that $d>\frac{3n+3}{2}$.
Then either (i) or (ii) of the following holds:
\begin{enumerate}
 \item[(i)] For general $\sigma'\in\Sigma'$, the inequality
\begin{equation}\label{eq-half1}
\frac{1}{e}\deg\mcO_{\Pn}(1)|_{Y_{\sigma'}} 
 \leq \frac{1}{e}
\left\{
2g(Y_{\sigma'})-2
+\left|
f_{\sigma'}^{-1}(D'_{\sigma'}\setminus \mcS)
\right|
\right\} 
 -(2g(C)-2)+c_{0}
\end{equation}
holds, where $e=\deg(Y_{\sigma'}/C)$ and $c_{0}=l(n-1)+\max\{0,2g(C)-2+|S|-l\}$.
 \item[(ii)] To a general $y\in \mcY$, 
denoting the point $f(y)\in X\times\Sigma'$ by $(x,\sigma')$,
we can associate a canonically determined vertical line $l_{y}$ in $X(=\Pn\times C)$
which passes through the point $x$ in such a way that
the divisors $D'_{\sigma'}|_{l_{y}}$ on the pointed lines $(l_{y}, x)$ are equivalent as a general point $y$ varies in $\mcY$.
\end{enumerate}
\end{prop}

\section{Proof of Proposition \ref{prop-half}} \label{section-proofhalf}
We divide the proof of Proposition \ref{prop-half} into several steps.

\subsection{Terminology}\label{subsection-term}

In the discussion, we use the following terminology.
Suppose we have an exact sequence 
\begin{equation}\label{eq-exseq}
0\to \mcT'\xrightarrow{\delta} \mcT \to \mcT'' \to 0
\end{equation}
 and a surjective morphism 
\begin{equation}\label{eq-surj}
\epsilon :\mcT\twoheadrightarrow \mcN
\end{equation}
of sheaves on a scheme.
If we put $\mcN'':=\Coker(\epsilon\circ\delta)$,
then we obtain an induced surjective morphism $\epsilon'':\mcT''\to \mcN''$.
If we put $\mcN':=\Ker(\mcN\to \mcN'')$, then we obtain an induced surjective morphism
$\epsilon':\mcT'\to \mcN'$.
The sheaves $\mcN'$ and $\mcN''$, together with the surjective morphisms $\epsilon'$ and $\epsilon''$,
are called \textit{induced quotients} of $\mcT'$ and $\mcT''$ respectively
obtained from (\ref{eq-exseq}) and (\ref{eq-surj}),
and the exact sequence 
\[
0\to \mcN'\to \mcN\to \mcN''\to 0
\]
is called \textit{an induced exact sequence}.

\subsection{Preparation}\label{subsection-prepa}
In this subsection, we study some tangent sheaves of varieties appearing in the diagram (\ref{eq-setting}).

Recall that $\Sigma=\mbP_{*}(S^{d}(V) \otimes W)$.
Projectivizing the trivial vector bundle $S^{d}(V)\otimes W\otimes \mcO_{C}$ on $C$,
we obtain a projective bundle 
$\mbP_{*}\left(
S^{d}(V)\otimes W\otimes \mcO_{C}
\right)$
over $C$,
which is nothing but $C\times \Sigma$.
We have a surjective map
$$
S^{d}(V)\otimes W \otimes \mcO_{C}
\to S^{d}(V) \otimes \mcL
$$
of vector bundles on $C$,
which gives rise to a dominant rational map
\begin{equation}\label{eq-CSig}
C\times\Sigma
=\mbP_{*}\left(
S^{d}(V)\otimes W\otimes \mcO_{C}
\right)
\cdots\to
\mbP_{*}\left(
S^{d}(V)\otimes\mcL
\right)
\end{equation}
of projective bundles over $C$.
Using the natural isomorphism
$$
\mbP_{*}\left(
S^{d}(V)\otimes\mcL
\right)
\simeq
\mbP_{*}\left(
S^{d}(V)\otimes\mcO_{C}
\right)
=C\times\mbP_{*}\left(
S^{d}(V)
\right),
$$
the rational map (\ref{eq-CSig}) can be written as
$C\times\Sigma \cdots\to C\times
\mbP_{*}\left(
S^{d}(V)
\right)$.
\begin{notation}\label{notation-barS}
For simplicity, we put $\barS:=\mbP_{*}\left(S^{d}(V)\right)$, the parameter space of degree $d$ hypersurfaces in $\Pn(=\mbP(V))$,
and write $\mcO_{\barS}(1)$ for $\mcO_{\mbP_{*}\left(S^{d}(V)\right)}(1)$.
\end{notation}
This rational map sends a point $(c, \sigma)$ to $(c, D_{\sigma}|_{\Pn\times\{c\}})$.
We denote by 
\begin{equation}\label{eq-alpha}
\alpha:C\times\Sigma^{sm}\to C\times \barS
\end{equation}
the restriction of the rational map to the open subset $C\times\Sigma^{sm}$,
which is a smooth morphism.
By construction, we have an isomorphism of sheaves on $C\times \Sigma^{sm}$
\begin{equation}\label{eq-alOS}
\alpha^{*}\left(\mcO_{C}\boxtimes\mcO_{\barS}(1) \right)
\simeq \mcL\boxtimes \mcO_{\Sigma^{sm}}(1),
\end{equation}
where $\mcO_{\Sigma^{sm}}(1)$ is the restriction to the open subset $\Sigma^{sm}\subset\Sigma$
of the line bundle $\mcO_{\Sigma}(1)$ on $\Sigma=\mbP_{*}(S^{d}(V) \otimes W)$.

Recall from \S \ref{subsect-posi} that the vector bundle $\mcM_{C}^{(\mcL;W)}$ is defined to be the kernel of the surjective map
$W\otimes\mcO_{C}\to \mcL$.
Then 
we have an isomorphism
\begin{equation}\label{eq-TCSigsm}
T_{C\times\Sigma^{sm}/C\times\barS}
\simeq \left(
\mcM_{C}^{(\mcL;W)}\otimes S^{d}(V)
\right)
\boxtimes\mcO_{\Sigma^{sm}}(1).
\end{equation}

The morphism $pr_{\Pn}\circ h:X\times\Sigma^{sm}\to\Pn$ in the diagram (\ref{eq-setting})
is a composition
\begin{equation}\label{eq-SXPCS}
\xymatrix@R=5pt{
X\times\Sigma^{sm} \ar@{=}[d] & X\times \barS \ar@{=}[d] &  &  \\
\Pn\times C\times \Sigma^{sm} \ar[r]^-{\mathrm{id}_{\Pn}\times \alpha}
&
\Pn\times C \times \barS \ar[r] &
\Pn
}
\end{equation}
where the second morphism is a projection to the first component.

On $\Pn \times\barS$,
we have a divisor $\mathbf{D}$ defined by
\begin{equation}\label{eq-mbfD}
\Pn \times \barS
\supset
\mathbf{D}:=\left\{
(z,D) \mid z\in D
\right\}.
\end{equation}
Let
$q_{12}:\Pn\times C\times \barS
\to \Pn\times C$
and
$q_{13}:\Pn\times C\times \barS
\to \Pn\times \barS$
be the projection to the $(1,2)-$ and $(1,3)$-component respectively.
The divisor $(q_{13}\circ (\mathrm{id}_{\Pn}\times \alpha))^{*}\mathbf{D}$ is equal to the divisor $\mcD$
in the diagram (\ref{eq-setting}).
By Proposition \ref{prop-TPU},
we have an isomorphism
\begin{equation}\label{eq-PPSd}
T_{\Pn\times \barS/\Pn}(-\log \mathbf{D})
\simeq
\mcM^{d}_{\Pn}\boxtimes\mcO_{\barS}(1),
\end{equation}
where see Notation \ref{notation-M} for $\mcM^{d}_{\Pn}$.

The divisor $q_{12}^{*}\mcS+q_{13}^{*}\mathbf{D}$ on
$\Pn\times C\times \barS$ is SNC and relatively SNC over $\Pn$.
From (\ref{eq-SXPCS}), we obtain  an exact sequence
\begin{equation}\label{eq-TTTSD}
\begin{split}
0\to
T_{\Pn\times C \times \Sigma^{sm}/ \Pn\times C\times \barS}
\to
T_{\Pn\times C \times \Sigma^{sm}/\Pn}\left(
-\log(\mcD+\mcS_{\Sigma^{sm}})
\right) \\
\to
(\mathrm{id}_{\Pn}\times \alpha)^{*}
T_{\Pn\times C\times \barS/\Pn}
\left(
-\log(
q_{12}^{*}\mcS+q_{13}^{*}\mathbf{D}
)
\right)
\to 0.
\end{split}
\end{equation}
By (\ref{eq-TCSigsm}),
we have
\begin{equation}\label{eq-Tcps}
T_{\Pn\times C \times \Sigma^{sm}/ \Pn\times C\times \barS}
\simeq 
\mcO_{\Pn}\boxtimes \left( \mcM_{C}^{(\mcL;W)}\otimes S^{d}(V) \right)
\boxtimes\mcO_{\Sigma^{sm}}(1).
\end{equation}
By (\ref{eq-PPSd}),
we have
\begin{equation}\label{eq-Tqq}
T_{\Pn\times C\times \barS/\Pn}
\left(
-\log(
q_{12}^{*}\mcS+q_{13}^{*}\mathbf{D}
)
\right)
\simeq q_{2}^{*}T_{C}(-S)
\oplus
q_{13}^{*}
\left(
\mcM_{\Pn}^{d}\boxtimes \mcO_{\barS}(1)
\right),
\end{equation}
where $q_{2}:\Pn\times C\times \barS \to C$ is the projection to the second factor.

\subsection{Proof of  the case (i) in Proposition \ref{prop-half}}
In this subsection, we will define a non-negative integer $s_{0}$ and show that if the inequality
$d-n-s_{0}-1 \geq 1$ holds, then (i) of Proposition \ref{prop-half} holds.
In the next subsection, we will show that (ii) of Proposition \ref{prop-half} holds if $d-n-s_{0}-1 \leq 0.$

Recall that we are in the situation described in Section \ref{section-setting}.
Recall also from \S \ref{subsect-logtan} that the log normal bundle $\mcN_{f}^{\log}$ of the morphism $f$ in the diagram (\ref{eq-setting})
is defined  by the exact sequence
\[
0\to 
T_{\mcY}\left(
-\log(\mcD_{\mcY}\cup \mcS_{\mcY})
\right)
\to
f^{*}T_{X\times\Sigma'}\left(-\log(
\mcD'\cup\mcS_{\Sigma'})
\right)
\to\mcN_{f}^{\log}
\to 0.
\]
For simplicity of notation, we write $\mcN_{f}$ for  $\mcN_{f}^{\log}$.
Since we are assuming that the divisor $\mcD_{\mcY}\cup \mcS_{\mcY}$ on $\mcY$ and
$\mcD'\cup \mcS_{\Sigma'}$ on $X\times \Sigma'$ are relatively SNC over $\Sigma'$,
by Proposition \ref{prop-reltan}, 
the log normal bundle $\mcN_{f}$ also fits in an exact sequence
 \[
0\to 
T_{\mcY/\Sigma'}\left(
-\log(\mcD_{\mcY}\cup \mcS_{\mcY})
\right)
\to
f^{*}T_{X\times\Sigma'/\Sigma'}\left(-\log(
\mcD'\cup\mcS_{\Sigma'})
\right)
\to\mcN_{f}
\to 0.
\]
By restricting this exact sequence to the fiber $Y_{\sigma'}$, $\sigma'\in \Sigma'$,
we obtain an exact sequence
\begin{equation}\label{eq-NfY}
0\to 
T_{Y_{\sigma'}}\left(
-\log(D_{Y_{\sigma'}}\cup \mcS_{Y_{\sigma'}})
\right)
\to
f_{\sigma'}^{*}T_{X}\left(
-\log(
D'_{\sigma'}\cup\mcS)
\right)
\to\left.\mcN_{f}\right|_{Y_{\sigma'}}
\to 0.
\end{equation}
We have
\begin{equation}\label{eq-TY>}
\begin{split}
\deg
T_{Y_{\sigma'}}\left(
-\log(D_{Y_{\sigma'}}\cup \mcS_{Y_{\sigma'}})
\right)
& = 2-2g(Y_{\sigma'})-\left|
D_{Y_{\sigma'}}\setminus \mcS_{Y_{\sigma'}}
\right|
-\left|
\mcS_{Y_{\sigma'}}
\right| \\
&
\geq 
2-2g(Y_{\sigma'})-\left|
D_{Y_{\sigma'}}\setminus \mcS_{Y_{\sigma'}}
\right|
-e|S|,
\end{split}
\end{equation}
where $e=\deg(Y_{\sigma'}/C)$.
Since
\[
\det
T_{X}\left(
-\log(
D'_{\sigma'}\cup\mcS)
\right)\simeq
\mcO_{\Pn}(n+1-d)\boxtimes
(K_{C}^{-1}\otimes\mcL^{-1}(-S)),
\]
we have
\begin{equation}\label{eq-fDS}
\deg
f_{\sigma'}^{*}T_{X}\left(
-\log(
D'_{\sigma'}\cup\mcS)
\right)
=(n+1-d)\deg\left.\mcO_{\Pn}(1)\right|_{Y_{\sigma'}}+e(2-2g(C)-l-|S|).
\end{equation}
By (\ref{eq-NfY}),
(\ref{eq-TY>}) and (\ref{eq-fDS}),
we have
\begin{equation}\label{eq-Nupper}
\begin{split}
\deg
\mcN_{f}|_{Y_{\sigma'}} & = \deg
f_{\sigma'}^{*}T_{X}\left(
-\log(
D'_{\sigma'}\cup\mcS)
\right)
-
\deg
T_{Y_{\sigma'}}\left(
-\log(D_{Y_{\sigma'}}\cup \mcS_{Y_{\sigma'}})
\right) \\
& \leq
(n+1-d)\deg\left.\mcO_{\Pn}(1)\right|_{Y_{\sigma'}}+e(2-2g(C)-l-|S|) \\
& \quad +
2g(Y_{\sigma'})-2+\left|
D_{Y_{\sigma'}}\setminus \mcS_{Y_{\sigma'}}
\right|
+e|S| \\
&
=2g(Y_{\sigma'})-2+\left|
D_{Y_{\sigma'}}\setminus \mcS_{Y_{\sigma'}}
\right|
+(n+1-d)\deg\left.\mcO_{\Pn}(1)\right|_{Y_{\sigma'}}+e(2-2g(C)-l)
\end{split}
\end{equation}

Next we will obtain a lower bound of $\deg
\mcN_{f}|_{Y_{\sigma'}}$.

Since $GL(V)$ acts on $\Pn$ transitively, 
the $GL(V)$-equivariant morphism 
$pr_{\Pn}\circ h \circ \beta' \circ f : \mcY \to \Pn$
is smooth, and the divisors
$\mcD_{\mcY}\cup \mcS_{\mcY}$ on $\mcY$ and
$\mcD'\cup \mcS_{\Sigma'}$ on $X\times \Sigma'$ are relatively SNC over $\Pn$.
Therefore the normal bundle $\mcN_{f}$ also fits in an exact sequence
\begin{equation}\label{eq-toNf}
0 \to
T_{\mcY/\Pn}\left(
-\log (\mcD_{\mcY}\cup \mcS_{\mcY})
\right)
\to
f^{*} T_{X\times \Sigma'/\Pn}\left(
-\log(\mcD'\cup\mcS_{\Sigma'})
\right)
\to
\mcN_{f}
\to 0.
\end{equation}
Pulling back the exact sequence of sheaves on $X\times \Sigma'$
\begin{equation*}
0\to
p'^{*}T_{\Sigma'/\Sigma^{sm}}
\to
T_{X\times \Sigma'/\Pn}\left(
-\log(\mcD'\cup\mcS_{\Sigma'})
\right)
\to
\beta'^{*}T_{X\times \Sigma^{sm}/\Pn}\left(
-\log(\mcD\cup \mcS_{\Sigma^{sm}})
\right)
\to 0
\end{equation*}
by the morphism $f$, 
we obtain an exact sequence of sheaves on $\mcY$
\begin{equation}
0\to
p_{1}^{*}T_{\Sigma'/\Sigma^{sm}}
\to
f^{*}T_{X\times \Sigma'/\Pn}\left(
-\log(\mcD'\cup\mcS_{\Sigma'})
\right)
\to
(\beta'\circ f)^{*}T_{X\times \Sigma^{sm}/\Pn}\left(
-\log(\mcD\cup \mcS_{\Sigma^{sm}})
\right)
\to 0.
\end{equation}
From this exact sequence and the surjective map 
$f^{*} T_{X\times \Sigma'/\Pn}\left(
-\log(\mcD'\cup\mcS_{\Sigma'})
\right)
\to
\mcN_{f}$ in (\ref{eq-toNf}),
we obtain the induced quotients
\begin{equation}\label{eq-N1}
(\beta'\circ f)^{*}T_{X\times \Sigma^{sm}/\Pn}\left(
-\log(\mcD\cup \mcS_{\Sigma^{sm}})
\right)\twoheadrightarrow \mcN_{1}
\end{equation}
$p_{1}^{*}T_{\Sigma'/\Sigma^{sm}} \twoheadrightarrow\mcN_{2}$,
and the induced exact sequence
\[
0\to \mcN_{2} \to \mcN_{f} \to \mcN_{1} \to 0
\]
(see \S \ref{subsection-term} for induced quotients and induced exact sequences).
Since the restriction of $p_{1}^{*}T_{\Sigma'/\Sigma^{sm}}$ to the fiber $Y_{\sigma'}$ is a trivial bundle,
we have
\begin{equation}\label{eq-degN2}
\deg\mcN_{2}|_{Y_{\sigma'}} \geq 0.
\end{equation}

In order to obtain a lower bound of $\deg \mcN_{1}|_{Y_{\sigma'}}$,
we need to show some positivity of $(\beta'\circ f)^{*}T_{X\times \Sigma^{sm}/\Pn}\left(
-\log(\mcD\cup \mcS_{\Sigma^{sm}})
\right)$.
From (the pull-back by $\beta'\circ f$ of) the exact sequence (\ref{eq-TTTSD}) and the surjective map (\ref{eq-N1}),
we obtain the induced quotients
\begin{align}
(\beta'\circ f)^{*} (\mathrm{id}_{\Pn}\times \alpha)^{*}
T_{\Pn\times C\times \barS/\Pn}
\left(
-\log(
q_{12}^{*}\mcS+q_{23}^{*}\mbfD
)
\right) \twoheadrightarrow \mcN_{3}, \label{eq-N3} \\
(\beta'\circ f)^{*}
T_{\Pn\times C \times \Sigma^{sm}/ \Pn\times C\times \barS }
\twoheadrightarrow \mcN_{4}, \notag
\end{align}
and the induced exact sequence
\[
0\to \mcN_{4}\to \mcN_{1} \to \mcN_{3} \to 0.
\]
By (\ref{eq-Tcps}), 
we see that there is a surjective map
\[
(pr_{C}\circ f_{\sigma'})^{*}\left(\mcM_{C}^{(\mcL;W)}\right)^{\oplus \dim S^{d}(V)}\twoheadrightarrow \mcN_{4}|_{Y_{\sigma'}}.
\]
Thus the sheaf $\left(\mcN_{4}|_{Y_{\sigma'}}\right)\otimes(pr_{C}\circ f_{\sigma'})^{*}\mcL$ is globally generated
by Proposition \ref{prop-posi},
and we have
\begin{equation}\label{eq-degN4}
\deg\mcN_{4}|_{Y_{\sigma'}} \geq -e (\rank\, \mcN_{4})\cdot \deg\mcL.
\end{equation}

By (\ref{eq-Tqq}),
there is a short exact sequence
\begin{align*}
0\to  (\beta'\circ f)^{*} & (\mathrm{id}_{\Pn}\times \alpha)^{*}
q_{2}^{*}
T_{C}(-S) \\
 \to
(\beta '\circ& f)^{*}   (\mathrm{id}_{\Pn}\times \alpha)^{*}
T_{\Pn\times C\times \barS /\Pn}
\left(
-\log(
q_{12}^{*}\mcS+q_{23}^{*}\mbfD
)
\right) \\
&  \to 
(\beta'\circ f)^{*} (\mathrm{id}_{\Pn}\times \alpha)^{*}
q_{13}^{*}
\left(
\mcM_{\Pn}^{d}\boxtimes \mcO_{\barS}(1)
\right)
\to 0.
\end{align*}
From this and the surjective map (\ref{eq-N3}),
we obtain the induced quotients
\begin{align}
(\beta'\circ f)^{*}\circ (\mathrm{id}_{\Pn}\times \alpha)^{*}
q_{13}^{*}
\left(
\mcM_{\Pn}^{d}\boxtimes \mcO_{\barS}(1)
\right)
\twoheadrightarrow \mcN_{5}, \label{align-N5} \\
 (\beta'\circ f)^{*}\circ (\mathrm{id}_{\Pn}\times \alpha)^{*}
q_{2}^{*}
T_{C}(-S) \twoheadrightarrow \mcN_{6},\label{align-N6}
\end{align}
and the induced exact sequence
\[
0\to \mcN_{6} \to \mcN_{3} \to \mcN_{5} \to 0.
\]
By (\ref{align-N6}), we have
\begin{equation}\label{eq-degN6}
\deg\mcN_{6}|_{Y_{\sigma'}} \geq
e(\rank \mcN_{6})(2-2g(C) -|S|).
\end{equation}
Finally we will study a lower bound of $\deg\mcN_{5}|_{Y_{\sigma'}}$.
Note that by definition of $\mcN_{5}$, there is an exact sequence
\begin{equation}\label{eq-exN5}
T_{\mcY/\Pn}\left(
-\log(\mcD_{\mcY}\cup \mcS_{\mcY})
\right)
\to
(\beta'\circ f)^{*} (\mathrm{id}_{\Pn}\times \alpha)^{*}
q_{13}^{*}
\left(
\mcM_{\Pn}^{d}\boxtimes \mcO_{\barS}(1)
\right)
\xrightarrow{\eta} \mcN_{5} \to 0,
\end{equation}
where we denoted by $\eta$ the surjective map in (\ref{align-N5}).
Take general elements $P_{1}, P_{2}, \dots$ of $S^{d-1(}V)$.
For a point $z\in\Pn$,
we can consider a multiplication map
\[
M_{z}^{1} \xrightarrow{P_{i}\cdot} M^{d}_{z},
\]
where see Notation and Convention for the notation $M_{z}^{d}$.
Since the fiber of $\mcM_{\Pn}^{d}$ at a point $z\in \Pn$ is $M_{z}^{d}$,
we can also consider a multiplication map
\[
\mcM_{\Pn}^{1} \xrightarrow{P_{i}\cdot} \mcM^{d}_{\Pn}.
\]
Define
\[
\gamma(i)  := \rank\, \eta\left(
\sum_{j=1}^{i}P_{j}\cdot (\beta'\circ f)^{*} (\mathrm{id}_{\Pn}\times \alpha)^{*}q_{13}^{*}\left(
\mcM_{\Pn}^{1}\boxtimes \mcO_{\mathbf{P}_{*}\left(
S^{d}(V)
\right)}(1)
\right)
\right).
\]
Then we have $\gamma(1)\leq \gamma(2)\leq \cdots$ and the sequence $\left( \gamma(i+1)-\gamma(i)\right)_{i\geq 1}$ is 
non-increasing.
Put
\[
s_{j} :=\min\{
i \mid \gamma(i+1)-\gamma(i) \leq j
\}.
\]
Then we have a generically surjective map
\[
(\beta'\circ f)^{*} (\mathrm{id}_{\Pn}\times \alpha)^{*}
q_{13}^{*}
\left(
\mcM_{\Pn}^{1}\boxtimes \mcO_{\barS}(1)
\right)^{\oplus s_{0}}
\to\mcN_{5},
\]
thus there is a generically surjective map to $\wedge^{n_{5}}\mcN_{5}\otimes(\beta'\circ f)^{*} (\mathrm{id}_{\Pn}\times \alpha)^{*}
q_{3}^{*}\mcO_{\barS}(-n_{5})$
from a direct sum of sheaves of the form
\begin{equation}\label{eq-wedge}
\bigotimes_{k=1}^{s_{0}}
\bigwedge^{a_{k}}
(pr_{\Pn}\circ h \circ \beta' \circ f)^{*}
\mcM_{\Pn}^{1},
\end{equation}
where $n_{i}:=\rank\,\mcN_{i}$, $\sum_{k=1}^{s_{0}}a_{k}=n_{5}$ 
and $q_{3}:\Pn\times C\times \barS \to \barS$
is the  projection to the third factor.
By Proposition \ref{prop-posi},
the sheaf $\left( \bigotimes_{k=1}^{s_{0}}
\bigwedge^{a_{k}}
\mcM_{\Pn}^{1}\right)
\otimes \mcO_{\Pn}(s_{0})$
is generated by global sections.

So we have
\begin{align*}
\left.
\deg\left(
\wedge^{n_{5}}\mcN_{5}
\right)
\otimes
(\beta'\circ f)^{*} (\mathrm{id}_{\Pn}\times \alpha)^{*}
q_{3}^{*}\mcO_{\barS}(-n_{5}) 
\otimes
(pr_{\Pn}\circ h \circ \beta' \circ f)^{*}\mcO_{\Pn}(s_{0})
\right|_{Y_{\sigma'}} \geq 0.
\end{align*}
Hence we obtain
\begin{equation}\label{eq-degN5}
\deg\mcN_{5}|_{Y_{\sigma'}} \geq
n_{5}\cdot e\cdot l -s_{0} \deg\mcO_{\Pn}(1)|_{Y_{\sigma'}} .
\end{equation}
By (\ref{eq-degN2}), (\ref{eq-degN4}), (\ref{eq-degN6}) and (\ref{eq-degN5}),
we have
\begin{equation}
\deg\mcN_{f}|_{Y_{\sigma'}}
\geq el(n_{5}-n_{4})+en_{6}(2-2g(C)-|S|)-s_{0} \deg\mcO_{\Pn}(1)|_{Y_{\sigma'}} .
\end{equation}
Note that
\[
n=\rank\,\mcN_{f}=n_{2}+n_{4}+n_{5}+n_{6} \quad
\text{and}
\quad n_{6}\leq 1.
\]
So we have 
\begin{equation}\label{eq-Nlower}
\deg\mcN_{f}|_{Y_{\sigma'}}
\geq 
-el(n-1)-e\cdot\max\{
l,2g(C)-2+|S|
\}
-s_{0} \deg\mcO_{\Pn}(1)|_{Y_{\sigma'}} .
\end{equation}
Combining (\ref{eq-Nupper}) and (\ref{eq-Nlower}), 
we obtain
\begin{align*}
2g(Y_{\sigma'})-2+\left|
D_{Y_{\sigma'}}\setminus \mcS_{Y_{\sigma'}}
\right|
+(n+1-d)\deg\left.\mcO_{\Pn}(1)\right|_{Y_{\sigma'}}+e(2-2g(C)-l) \\
\geq
-el(n-1)-e\cdot\max\{
l,2g(C)-2+|S|
\}
-s_{0} \deg\mcO_{\Pn}(1)|_{Y_{\sigma'}} .
\end{align*}
Hence we have
\begin{align*}
(d-n-1-s_{0})\deg\left.\mcO_{\Pn}(1)\right|_{Y_{\sigma'}}
&\leq  2g(Y_{\sigma'})-2+ \left|
D_{Y_{\sigma'}}\setminus \mcS_{Y_{\sigma'}}
\right| \\
& +
e\left(
2-2g(C)-l+l(n-1)+ \max\{l, 2g(C)-2+|S| \}
\right).
\end{align*}
If $d-n-1-s_{0} \geq 1$, then
\begin{align*}
\frac{1}{e}\deg\left.\mcO_{\Pn}(1)\right|_{Y_{\sigma'}}
\leq & \frac{1}{e} \left\{ 2g(Y_{\sigma'})-2+ \left|
D_{Y_{\sigma'}}\setminus \mcS_{Y_{\sigma'}}
\right| 
\right\}
\\
& +
2-2g(C)+l(n-1)+\max\{0, 2g(C)-2+|S| -l\}.
\end{align*}

\subsection{Proof of  the case (ii) in Proposition \ref{prop-half}}
In this subsection, continuing the arguments in the previous subsection,
we show that (ii) of Proposition \ref{prop-half} holds if $d-n-s_{0}-1 \leq 0.$
In the rest of this section, we assume $d-n-s_{0}-1 \leq 0.$

Denote by $k$ the composite of morphisms
\begin{equation*}
\mcY\xrightarrow{f}
X\times\Sigma'\xrightarrow{\beta}
X\times\Sigma^{sm}=\Pn\times C\times\Sigma^{sm}
\xrightarrow{\mathrm{id}_{\Pn}\times \alpha}
\Pn\times C\times\barS
\xrightarrow{q_{13}}\Pn\times\barS,
\end{equation*}
where recall that $\barS=\mbP_{*}(S^{d}V)$ and the morphism $\alpha$ is defined in (\ref{eq-alpha}).
Let $y\in \mcY$ be a general point.
This in particular implies that $y$ does not lie in the divisor $\mcD_{\mcY}\cup \mcS_{\mcY}$,
and that the point $k(y)$ does not line in the divisor $\mathbf{D}$ defined in (\ref{eq-mbfD}).
Let $f(y)\in X\times\Sigma'=\Pn\times C \times \Sigma'$ be expressed as 
$(z,c,\sigma')$ with $z\in\Pn$, $c\in C$ and $\sigma'\in \Sigma'$.
Put $T:=\left.\mcT_{\mcY/\Pn}(-\log(\mcD_{\mcY}\cup \mcS_{\mcY}))\right|_{y}
=\left.\mcT_{\mcY/\Pn}\right|_{y}.$
Restricting the map
\[
T_{\mcY/\Pn}\to k^{*}T_{\Pn\times\barS/\Pn}
\]
of tangent bundles induced by the morphism $k$ to the fiber over $y$,
we obtain a linear map
\begin{equation}
k_{*}:
T\to \left.T_{\Pn\times\barS/\Pn}\right|_{k(y)}
=\frac{S^{d}V}{\left.\mbC\cdot D'_{\sigma'}\right|_{\Pn\times\{c\}}}
\otimes \left(
\left.\mbC\cdot D'_{\sigma'}\right|_{\Pn\times\{c\}}
\right)^{*},
\end{equation}
where, by abuse of notation,
we denoted by $\left.\mbC\cdot D'_{\sigma'}\right|_{\Pn\times\{c\}}$
the $1$-dimensional subspace of $S^{d}V$
spanned by an element defining the divisor $\left.D'_{\sigma'}\right|_{\Pn\times\{c\}}$.
Since $z\notin \left.D'_{\sigma'}\right|_{\Pn\times\{c\}}$,
we have an isomorphism
\[
M^{d}_{z}\xrightarrow{\sim}
\frac{S^{d}V}{\left.\mbC\cdot D'_{\sigma'}\right|_{\Pn\times\{c\}}},
\]
and we identify the two vector spaces via this isomorphism (see Notation and convention for the notation $M_{z}^{d}$).
Then the map $k_{*}$ is equal to the restriction of the left arrow in (\ref{eq-exN5})
to the fiber over $y\in \mcY$.
\begin{claim}\label{claim-ss2}
We have $s_{0}-s_{1}\geq 2$.
\end{claim}
\proof[Proof of Claim \ref{claim-ss2}]
Suppose that $s_{0}-s_{1} \leq 1$.
Then we have
\[
n\geq n_{5}\geq s_{0}+s_{1} \geq 2s_{0}-1.
\]
So we have 
\[
\frac{n+1}{2} \geq s_{0}\geq d-n-1.
\]
Then we obtain that $\frac{3n+3}{2}\geq d$,
which contradicts our assumption that $d>\frac{3n+3}{2}$.
\endproof
Recall that $P_{1},P_{2},\dots$ are general elements of $S^{d-1}(V)$.
The subspace
\[
\Ker\left(
M^{1}_{z}
\otimes 
\left(
\left.\mbC\cdot D'_{\sigma'}\right|_{\Pn\times\{c\}}
\right)^{*}
\xrightarrow{P_{s_{1}+1}\cdot}
\frac{M^{d}_{z}
\otimes
\left(
\left.\mbC\cdot D'_{\sigma'}\right|_{\Pn\times\{c\}}
\right)^{*}}
{
(\mathrm{Im}\, k_{*})
+\sum_{i=1}^{s_{1}}
P_{i}\cdot
\left(
M^{1}_{z}\otimes 
\left(
\left.\mbC\cdot D'_{\sigma'}\right|_{\Pn\times\{c\}}
\right)^{*}
\right)
}
\right)
\]
of $M^{1}_{z}
\otimes 
\left(
\left.\mbC\cdot D'_{\sigma'}\right|_{\Pn\times\{c\}}
\right)^{*}$
is $1$-dimensional,
so it is of the form
\[
M^{1}_{l_{y}}
\otimes 
\left(
\left.\mbC\cdot D'_{\sigma'}\right|_{\Pn\times\{c\}}
\right)^{*}
\]
for a line $l_{y}\subset\Pn$ passing through $z$
(see Notation and Convention for the notation $M^{1}_{l_{y}}$).
By Claim \ref{claim-ss2},
carrying out a similar argument as in the proof of \cite[Lemma 3.2]{CR},
we see that the line $l_{y}$ does not depend on the choice
of general elements
$P_{1},P_{2},\dots \in S^{d-1}(V)$,
and 
\begin{equation}\label{eq-MImk}
M^{d}_{l_{y}}
\otimes 
\left(
\left.\mbC\cdot D'_{\sigma'}\right|_{\Pn\times\{c\}}
\right)^{*}
\subset
(\mathrm{Im}\,k_{*})
+\sum_{i=1}^{s_{1}}
P_{i}\cdot
\left(
M^{1}_{z}\otimes 
\left(
\left.\mbC\cdot D'_{\sigma'}\right|_{\Pn\times\{c\}}
\right)^{*}
\right).
\end{equation}
Define 
\[
T':=
k_{*}^{-1}
\left(
M^{d}_{l_{y}}
\otimes 
\left(
\left.\mbC\cdot D'_{\sigma'}\right|_{\Pn\times\{c\}}
\right)^{*}
\right).
\]
By (\ref{eq-MImk}),
we have $\dim T/T'=d-s_{0}$.
So we have
\[
\dim T'=\dim \mcY-n-d+s_{0}.
\]
As a general point $y$ varies in $\mcY$,
the subspaces $T'\subset T$ form a subbundle
$\mcT'\subset \mcT_{\mcY/\Pn}$ over a Zariski open subset $\mcY^{\circ}\subset\mcY$.
Recall that $\mbG$ is the Grassmannian variety parametrizing lines in $\Pn=\mbP(V)$.
Let $\mbL$ be the moduli space parametrizing pairs $(l,w)$,
where $l$ is a line in $\Pn$, and $w$ is a point of $l$.
Let $\pi_{1}:\mbL\to \mbG$,
$\pi_{2}:\mbL\to\Pn$
be morphisms defined by $(l,w)\mapsto l$ and $(l,w)\mapsto w$ respectively.
By associating $(l_{y},z)$ to a general point $y$,
we obtain a morphism $g:\mcY^{\circ}\to \mbL$
such that 
$\pi_{2}\circ g
=(pr_{\Pn}\circ h \circ \beta' \circ f)|_{\mcY^{\circ}}$.
By a similar argument as in \cite[Appendix]{A24} (see also
\cite[Lemma 3]{V98}, \cite[\S 6]{C03}, \cite[\S 3]{CR}),
we can show that 
\begin{align}
\text{
the subbundle $\mcT'\subset \mcT_{\mcY/\Pn}$ is integrable;
}\label{align-T'T} 
\\
\text{
$\mcT'$ is in the kernel of 
$g_{*}:\mcT_{\mcY^{\circ}/\Pn}
\to g^{*}\mcT_{\mbL/\Pn}$.
}\label{align-TgT}
\end{align}

Fix a point $z_{0}\in \Pn$ and a line $l_{0}\subset \Pn$ passing through $z_{0}$.
Define closed subschemes $\mcY^{\circ}_{z_{0}}$ and $\mcY^{\circ}_{(l_{0},z_{0})}$
of $\mcY^{\circ}$ by
\begin{align*}
\mcY^{\circ}_{z_{0}} &=
\left.(pr_{\Pn}\circ h\circ \beta' \circ f)\right|_{\mcY^{\circ}}^{-1}(z_{0}), \\
\mcY^{\circ}_{(l_{0},z_{0})} &=
\left\{
y\in \mcY^{\circ}
\mid g(y)=(l_{0}, z_{0})
\right\}.
\end{align*}
Since $\GLV$ acts transitively on $\mbL$, and the morphism $g$ is $\GLV$-equivariant,
these subschemes are smooth, and
\begin{align*}
\dim\mcY^{\circ}_{z_{0}}&=
\dim\mcY^{\circ}-n, \\
\dim\mcY^{\circ}_{(l_{0},z_{0})}&=
\dim\mcY^{\circ}-2n+1.
\end{align*}
Here the subscheme
$\mcY^{\circ}_{(l_{0},z_{0})}$ is irreducible.
In fact, if $Y_{1}\subset \mcY^{\circ}_{(l_{0},z_{0})}$ is a connected component, and  $\mathfrak{S}\subset \GLV$ is the stabilizer
of the point $(l_{0},z_{0})\in \mbL$, then, 
since $\mathfrak{S}$ is connected, the fibers of $\GLV\cdot Y_{1}(\subset \mcY)$ over $\mbL$ are
isomorphic to $Y_{1}$.
Since $\mcY$ is irreducible by assumption, we have $\GLV\cdot Y_{1}=\mcY$.
So we have $Y_{1}=\mcY^{\circ}_{(l_{0},z_{0})}$.

Define a morphism $\rho:\mcY^{\circ}_{(l_{0},z_{0})}\to \mbP_{*}\mrH^{0}\left(l_{0}, \mcO_{l_{0}}(d)\right)$
by $y \mapsto \left.D_{\sigma'}\right|_{l_{0}\times\{c\}}$,
where $f(y)=(z_{0}, c,\sigma')$.
For each point $y\in \mcY^{\circ}_{(l_{0},z_{0})}$, 
there is a submanifold $L\subset \mcY^{\circ}_{z_{0}}$
passing through $y$ such that $T_{L}=\left.\mcT'\right|_{L}\subset \left.\mcT_{\mcY^{\circ}/\Pn}\right|_{L}$.
By (\ref{align-TgT}), the morphism $g$ is constant on $L$, so we have $L\subset \mcY^{\circ}_{(l_{0},z_{0})}$.
Moreover, by the definition of $T'$, the morphism $\rho$ is constant on $L$.
So we have
\begin{align*}
\dim \mathrm{Im}\, \rho & \leq
\dim \mcY^{\circ}_{(l_{0},z_{0})} -\dim L \\
&=\dim \mcY-2n+1-(\dim\mcY-n-d+s_{0}) \\
& = d-n-s_{0}+1 \leq 2.
\end{align*}
We see that
if $\dim \mathrm{Im}\, \rho=1$, then for general $y\in \mcY^{\circ}$ with $f(y)=(z,c,\sigma')$,
the line $l_{y}$ intersects $D_{\sigma'}$ in one point;
and if $\dim \mathrm{Im}\, \rho=2$,
then for general points $y_{1},y_{2} \in \mcY$ with 
$f(y_{i})=(z_{i},c_{i},\sigma'_{i})$, $i=1,2$,
the divisor 
$\left.
D_{\sigma'_{1}}
\right|_{l_{y_{1}}\times\{c_{1}\}}$
on the pointed line $(l_{y_{1}},z_{1})$ is equivalent to the divisor
$\left.
D_{\sigma'_{2}}
\right|_{l_{y_{2}}\times\{c_{2}\}}$
on the pointed line $(l_{y_{2}},z_{2})$.
This shows that (ii) of Proposition \ref{prop-half} holds.
\hfill \qed

%
\section{Lines intersecting hypersurfaces in a fixed divisor}\label{sect-lines}
Fixing a degree $d$ effective divisor $E_{0}$ on a pointed line $(l_{0}, p_{0})$, 
we construct a moduli space of quadruples $(l,p,,E,B)$ where $l\subset \Pn$ is a line, $p\in l$,
$E$ is a degree $d$ effective divisor on $l$, and $B\subset \Pn$ is a degree $d$ hyperplane such that
$B|_{l}=E$ and the divisor $E$ on the pointed line $(l,p)$ and the divisor $E_{0}$ on $(l_{0},p_{0})$ are equivalent.
We also construct  a compactification of the moduli.
We carried out such construction in \cite[\S 5]{A24},
where moduli were constructed relatively over the affine space $\mrH^{0}(\Pn, \mcO(d))$.
In this paper we want the moduli space to be constructed relatively over the projective space
$\mbP_{*}\left(
\mrH^{0}(\Pn, \mcO(d))
\right)
=\barS$.
To obtain such a moduli space,
we decided to reproduce most of the argument in \cite[\S 5]{A24}
rather than constructing our needed moduli space by dividing by the $\mbC^{\times}$-action
the moduli space constructed in \cite[\S 5]{A24}.
One reason for this is that the reproduction of the argument in \cite[\S 5]{A24} makes this paper self-contained;
another reason is that we need not only the moduli space but also the description of line bundles on it,
and 
it seems easier to reproduce the argument than to look at how $\mbC^{\times}$ acts on the fibers of  line bundles appearing  in \cite[\S 5]{A24}.

\subsection{Construction}\label{subsect-const}

Fix a pointed line $(l_{0},p_{0})$ in $\Pn$ and a degree $d$ effective divisor $E_{0}$ on $l_{0}$
such that $p_{0}\notin\Supp\,E_{0}$ and 
\begin{equation}\label{eq-assumption}
\sharp\Supp\,E_{0}\geq 2.
\end{equation}

In this section we construct a smooth variety $\overline{R}$ 
proper over $\barS=\mbP_{*}\left(
\mrH^{0}(\Pn,\mcO_{\Pn}(d))
\right)
=\mbP_{*}
\left(
S^{d}V
\right)$
whose open subset parametrizes quadruples $(l,p.E,B)$,
where  $l\subset \Pn$ is a line, $p$ is a point of $l$,
$E$ is a degree $d$ effective divisor on $l$ and $B\in \overline{\mbfS}$,
a hypersurface of degree $d$,
such that the divisor $E_{0}$ on $(l_{0},p_{0})$ and
$E$ on $(l,p)$ are equivalent, and that the equality $l\cap B=E$ holds as divisors.
For this, most of this section is devoted to the construction
of a smooth compactification of a variety $\mbO$ parametrizing
triples $(l,p,E)$ which constitute the first three components of the quadruples $(l,p.E,B)$
mentioned above.
(The precise definition of $\mbO$ will be given later.)
Now we begin the construction.
Decompose the divisor $E_{0}$ 
according to the multiplicities at each point as
\[
E_{0}=m^{(1)}E_{0}^{(1)}+\dots +m^{(c)}E_{0}^{(c)}, \quad c\geq 1,
\]
where $E_{0}^{(i)}=q_{1}^{(i)}+\dots +q_{e^{(i)}}^{(i)}$
with $q_{j}^{(i)}$, $1\leq j \leq e^{(i)}$, $1\leq i \leq c$,
distinct points,
and $m^{(i)}$, $1\leq i \leq c$,  distinct positive integers.
(Here we changed the notation from \cite{A24}, where
the notation $E_{0}^{(i)}$ was used for $m^{(i)}(q_{1}^{(i)}+\dots +q_{e^{(i)}}^{(i)})$.)

Note that we have
\[
d=\sum_{i=1}^{c}m^{(i)}e^{(i)}.
\]
For $1\leq i \leq c$,
let $\mbK^{(i)}$ be the space parametrizing triples
$(l,p,E^{(i)})$,
where $l\subset \Pn$ is a line, $p$ is a point of $l$
and $E^{(i)}$ is a degree $e^{(i)}$ effective divisor on $l$.
We have a natural morphism $\mbK^{(i)}\to \mbL$ given by
$(l,p,E^{(i)})\mapsto (l,p)$.
More precisely,
the variety $\mbK^{(i)}$ is constructed as follows.
Recall the natural exact sequence (\ref{eq-univG}) on $\mbG$
\[
0\to \mcK \to V\otimes \mcO_{\mbG}\to \mcQ \to 0.
\]
Then $\mbK^{(i)}=\mbP_{*}\left(
S^{e^{(i)}}\pi_{1}^{*}\mcQ \right)$.
Here see Notation and Convention at the end of the introduction
for the use of $\mbP_{*}$,
and recall that $\pi_{1}:\mbL\to \mbG$ is the morphism given by $(l,p)\mapsto l$.
The tautological line bundle $\mcO_{
\mbP_{*}\left(
S^{e^{(i)}}\pi_{1}^{*}\mcQ \right)
}(1)$
is denoted by $\mcO_{\mbK^{(i)}}(1)$.

Put $\mbK :=\mbK^{(1)}\times_{\mbL}\times \dots \times_{\mbL}\mbK^{(c)}$.
Thus $\mbK$ parametrizes tuples
$$(l,p,E^{(i)}, 1\leq i\leq c),$$
where $(l,p,E^{(i)})\in \mbK^{(i)}$ for each $1\leq i\leq c$.
Let 
$$\tau^{(i)}_{1}:\mbK\to \mbK^{(i)}$$
be the $i$-th projection and
\[
\tau^{(i)}_{2}:\mbK^{(i)}\to \mbL \quad
\text{and}
\quad 
\tau:\mbK \to \mbL
\]
be natural morphisms to $\mbL$.
We denote by $k_{0}$ the point $(l_{0}, p_{0}, E_{0}^{(i)}, 1\leq i \leq c)$ of $\mbK$.
The group $\GLV$ naturally acts on $\mbK$, and the orbit $\GLV \cdot k_{0}$
of $k_{0}$ by this action is denoted by $\mbO$.
We can consider $\mbO$ as a parameter space parametrizing triples $(l,p,E)$,
where $l\subset \Pn$ is a line, $p$ is a point of $l$
and $E$ is a degree $d$ effective divisor on $l$
such that the divisor $E_{0}$ on $(l_{0},p_{0})$ and
$E$ on $(l,p)$ are equivalent.
The closure $\overline{\mbO}\subset \mbK$ might have singularities.
We want to construct its smooth model.

Let
\begin{equation}\label{eq-frakS}
\GLV \supset
\mathfrak{S}:=\left\{
g\in \GLV \mid
g(l_{0},p_{0})=(l_{0},p_{0})
\right\}
\end{equation}
be the stabilizer group of $(l_{0},p_{0})$,
so we have an isomorphism $\mbL\simeq\GLV/\mathfrak{S}$
because $\GLV$ acts transitively on $\mbL$.
Denote by $\overline{\mbO}_{(l_{0}, p_{0})}$ (resp.$\mbO_{(l_{0}, p_{0})}$) the fiber of the map $\tau|_{\overline{\mbO}}:\overline{\mbO}\to \mbL$ 
(resp. $\tau|_{\mbO}:\mbO\to \mbL$) 
over $(l_{0},p_{0})$.
We aim at constructing a smooth model of $\overline{\mbO}_{(l_{0}, p_{0})}$,
then the homogeneity of $\mbL$ will allow us to obtain a smooth model of $\overline{\mbO}$.
Denote by $\mathsf{p}_{0} \subset \mrH^{0}(l_{0}, \mcO_{l_{0}}(1) )$ 
the codimension $1$ subspace corresponding to the point $p_{0}$,
and fix an isomorphism
\begin{equation}\label{eq-C2}
\mrH^{0}(l_{0}, \mcO_{l_{0}}(1) ) \simeq \mbC^{2}
\end{equation}
such that the subspace $\mathsf{p}_{0}$
maps to the subspace $\langle (1,0) \rangle$.
The isomorphism (\ref{eq-C2}) determines an isomorphism
\begin{equation}\label{eq-l0P}
l_{0} \simeq \mbP^{1}
\end{equation}
such that the point $p_{0}\in l_{0}$ corresponds to $[1:0]\in \mbP^{1}$.

Define a subalgebra $H_{0}\subset \End\, \mrH^{0}(l_{0}, \mcO_{l_{0}}(1) )$ by
\[
H_{0}=\left\{
\sigma \in \End\,\mrH^{0}(l_{0}, \mcO_{l_{0}}(1) ) \mid
\sigma(\mathsf{p}_{0})\subset \mathsf{p}_{0}
\right\},
\]
and a subgroup $H_{0}^{\times} \subset\Aut\, \mrH^{0}(l_{0}, \mcO_{l_{0}}(1) )$
by
$H_{0}^{\times}:=H_{0}\cap \Aut\, \mrH^{0}(l_{0}, \mcO_{l_{0}}(1) ) $.
Under the isomorphism (\ref{eq-C2}),
we can consider $H_{0}$ as the set of $2\times 2$ matrices of the form
$\begin{pmatrix}
* & * \\
0 & * 
\end{pmatrix}$.
Put
\[
\mbH_{0}:=\mbP_{*}(H_{0})=\left.\left\{\left.
\begin{pmatrix}
a & b \\
0 & d 
\end{pmatrix}
\right|
(a,b,d)\in \mbC^{3} \setminus\{(0,0,0)\} 
\right\}\right/
\mbC^{\times}
\simeq \mbP^{2}=\left\{
[a:b:d]
\right\}
\]
and $\mbH_{0}^{\times}:=H_{0}^{\times}/\mbC^{\times}$,
which is an open subset of $\mbH_{0}$
and nothing but $\Aut\, (l_{0}, p_{0})$.
The tautological line bundle $\mcO_{\mbP_{*}(H_{0})}(1)$ is denoted by $\mcO_{\mbH_{0}}(1)$.
The complement $\mbH_{0}\setminus \mbH_{0}^{\times}$ is a union of two lines \begin{equation}\label{eq-ADelta}
A:=(a=0)\quad \text{and}\quad \Delta:=(d=0).
\end{equation}

In the following argument, it is more convenient to consider that
$\mbK^{(i)}$ parametrizes triples $(l,p,\mathsf{E})$,
where $(l,p)\in \mbL$ and $\mathsf{E}$ is a $1$-dimensional subspace of 
$\mrH^{0}\left(
l_{0},
\mcO_{l_{0}}(e^{(i)})
\right)$.
Let $\mathsf{E}_{0}^{(i)}\subset 
\mrH^{0}\left(
l_{0},
\mcO_{l_{0}}(e^{(i)})
\right)$
be the $1$-dimensional subspace corresponding to the divisor $E_{0}^{(i)}$ on $l_{0}$.
We denote by $\mbK_{(l_{0}, p_{0})}$ the fiber of the map $\tau:\mbK\to \mbL$ over $(l_{0},p_{0})$.
We have an isomorphism
\[
\mbK_{(l_{0}, p_{0})}
\simeq \mbP^{e^{(1)}} \times \dots \times\mbP^{e^{(c)}}.
\]
Define a rational map 
$
\mu_{0}:\mbH_{0}\dashedrightarrow \mbK_{(l_{0}, p_{0})}
$
by
\[
\mu_{0}([\sigma])
=
\left(
l_{0}, p_{0},
(S^{e^{(i)}}\sigma)(\mathsf{E}_{0}^{(i)}),
1\leq i \leq c
\right)
\]
for $[\sigma]\in \mbH_{0}$, where $0\neq\sigma \in H_{0}$.
Let us consider the indeterminacy locus of this rational map.

The isomorphism (\ref{eq-l0P}) and the isomorphism $\mbP^{1}\setminus \{[1:0]\} \simeq \mbC$
by $[\alpha:1]\mapsto \alpha$ gives an isomorphism
\begin{equation}\label{eq-lC}
l_{0} \setminus\{ p_{0} \}
\simeq \mbC.
\end{equation}
Identify $\mrH^{0}(l_{0},\mcO(1))$ with $\mbC^{2}$ by the isomorphism (\ref{eq-C2}),
so we have
$\mrH^{0}(l_{0},\mcO(e^{(i)})) \simeq S^{e^{(i)}}(\mbC^{2})$.
If the point $q_{j}^{(i)}$ corresponds to $a_{j}^{(i)}\in \mbC$ via the isomorphism (\ref{eq-lC}),
then $\mathsf{E}_{0}^{(i)}=
\left\langle
\prod_{j=1}^{e^{(i)}}(a_{j}^{(i)}\mathbf{e}_{0}+\mathbf{e}_{1})
\right\rangle
\subset S^{e^{(i)}}(\mbC^{2})$,
where $\mathbf{e}_{0}=(1,0)$ and $\mathbf{e}_{1}=(0,1)$ are the canonical basis of $\mbC^{2}$.
The rational map $\mu_{0}$ is undefined at a point $[\sigma]\in \mbH_{0}$,
$\sigma=\begin{pmatrix}
a & b \\
0 & d 
\end{pmatrix}$,
precisely when $\mathsf{E}_{0}^{(i)}\subset \Ker\,S^{e^{(i)}}\sigma$ for some $1\leq i\leq c$.
Therefore the indeterminacy locus of $\mu_{0}$ consists of points
$
\left[
\sigma^{(i)}_{j}
\right]$ $ (1\leq i\leq c, \; 1\leq j\leq e^{(i)}),
$
where $\sigma^{(i)}_{j}=
\begin{pmatrix}
1 & -a^{(i)}_{j} \\
0 & 0
\end{pmatrix}.$
Let $\beta:\widetilde{\mbH_{0}}\to \mbH_{0}$ be the blowing-up of $\mbH_{0}$ at all these points.
Denote by $C^{(i)}_{j}$ be the exceptional curve of $\widetilde{\mbH_{0}}$ mapping to the point $
\left[
\sigma^{(i)}_{j}
\right]$.
We claim that this blowing-up resolves the indeterminacy locus of the rational map $\mu_{0}$,
or equivalently, that there is a morphism $\widetilde{\mu_{0}}:\widetilde{\mbH_{0}}\to \mbK_{(l_{0},p_{0})}$
such that $\widetilde{\mu_{0}}=\mu_{0}\circ \beta$.
Indeed, if we choose local coordinates $(u,v)$ of $\mbH_{0}$ centered at, say, $[\sigma_{1}^{(1)}]$
as $\left[
\begin{pmatrix}
1 & u-a_{1}^{(1)} \\
0 & v
\end{pmatrix}
\right]$,
then, in a neighborhood of $[\sigma_{1}^{(1)}]$,
the rational map $\mu_{0}$ is expressed as
\begin{equation}\label{eq-local}
\begin{split}
(u,v) \mapsto
\left(
\left\langle
(u\mathbf{e}_{0}+v\mathbf{e}_{1})\prod_{j=2}^{e^{(1)}}
\left(
(a_{j}^{(1)}-a_{1}^{(1)}+u)\mathbf{e}_{0}+v\mathbf{e}_{1}\right)
\right\rangle, \right. \\
\left.
\left\langle
\prod_{j=1}^{e^{(i)}}
\left(
(a_{j}^{(i)}-a_{1}^{(1)}+u)\mathbf{e}_{0}+v\mathbf{e}_{1}
\right)
\right\rangle,
2\leq i \leq c
\right).
\end{split}
\end{equation}
The exceptional curve $C^{(1)}_{1}$ is covered by the two affine charts with coordinates
$(u, v/u)$ and $(u/v,v)$ respectively.
In the affine chart with coordinates $(u,v/u)$, we have
\begin{equation}
\begin{split}
(u\mathbf{e}_{0}+v\mathbf{e}_{1})\prod_{j=2}^{e^{(1)}}
\left(
(a_{j}^{(1)}-a_{1}^{(1)}+u)\mathbf{e}_{0}+v\mathbf{e}_{1}\right) \\
=
u\times(\mathbf{e}_{0}+\frac{v}{u}\mathbf{e}_{1})\prod_{j=2}^{e^{(1)}}
\left(
(a_{j}^{(1)}-a_{1}^{(1)}+u)\mathbf{e}_{0}+u\cdot\frac{v}{u}\mathbf{e}_{1}\right),
\end{split}
\end{equation}
and the term after ``$u\times$\rq\rq{} does not vanich on $C^{(1)}_{1}=(u=0)$.
For $i\geq 2$, the terms appearing (\ref{eq-local}) do not vanish at $[\sigma_{1}^{(1)}]$.
It follows from these that $\widetilde{\mu_{0}}$ is defined on this affine chart.
In the other affine chart with coordinates $(u/v,v)$,
we can similarly verify that $\widetilde{\mu_{0}}$ is defined.
Since we can argue similarly in neighborhoods of other $C_{j}^{(i)}$\rq{}s,
we see that $\widetilde{\mu_{0}}$ is a morphism.
Moreover, from the above computation, we see that
\begin{equation}\label{eq-muback}
\begin{split}
\widetilde{\mu_{0}}^{*}
\left(
\mcO_{\mbP^{e^{(1)}}} \boxtimes
\dots
\boxtimes \mcO_{\mbP^{e^{(i)}}}(-1)
\boxtimes\dots\boxtimes
\mcO_{\mbP^{e^{(c)}}}  
\right) \\
\simeq
\beta^{*}\mcO_{\mbH_{0}}(-e^{(i)})
\otimes
\mcO_{\widetilde{\mbH_{0}}}
\left(
\sum_{j=1}^{e^{(i)}}
C^{(i)}_{j}
\right).
\end{split}
\end{equation}
The image of $\widetilde{\mu_{0}}:\widetilde{\mbH_{0}}\to \mbK_{(l_{0},p_{0})}$ is
$\overline{\mbO}_{(l_{0},p_{0})}$,
but the morphism $\widetilde{\mu_{0}}:\widetilde{\mbH_{0}}\to \overline{\mbO}_{(l_{0},p_{0})}$
is not necessarily birational.
In order to make it birational, we need to divide $\widetilde{\mbH_{0}}$
by a finite group.
Put
\[
\mathfrak{S}_{1}:=\left\{
\sigma
\left|
\sigma E_{0}=E_{0}
\right.
\right\}
\subset \Aut (l_{0}, p_{0})=\mbH_{0}^{\times}.
\]
By the assumption $\sharp\Supp\,E_{0} \geq 2$,
the group $\mathfrak{S}_{1}$ is a finite cyclic group.
By replacing the isomorphism (\ref{eq-C2}) if necessary,
we may assume that
\[
\mathfrak{S}_{1}=
\left\langle
\left[
\begin{pmatrix}
1 & 0 \\
0 & \zeta 
\end{pmatrix}
\right]
\right\rangle,
\]
where $\zeta=\exp\left( \frac{2\pi i}{r} \right)$
and $r$ is the order of $\mathfrak{S}_{1}$.
Since $E_{0}$ is fixed by the action of $\mathfrak{S}_{1}$,
we have
\begin{equation}\label{eq-eineq}
r\leq e^{(i)}
\end{equation}
unless $E^{(i)}_{0}$ consists of a single point $[0:1]$.
It follows from this and (\ref{eq-assumption})
that
\begin{equation}\label{eq-rdineq}
r\leq \sum_{i=1}^{c}e^{(i)}\leq d.
\end{equation}

Note that the group 
$\mbH_{0}^{\times}=\Aut\,(l_{0},p_{0})$ has a natural right action on $\mbH_{0}$.
This action does not necessarily lift to $\widetilde{\mbH_{0}}$,
but the subgroup $\mathfrak{S}_{1}$ acts on  $\widetilde{\mbH_{0}}$
from the right, and the morphism $\widetilde{\mu_{0}}$ factors as
\[
\widetilde{\mbH_{0}}\to \widetilde{\mbH_{0}}/\mathfrak{S}_{1} \to\overline{\mbO}_{(l_{0},p_{0})},
\]
where the right arrow is birational.
However the variety $\widetilde{\mbH_{0}}/\mathfrak{S}_{1}$ might have quotient singularities. 
Let us construct its smooth model.
If the point $[0:1]\in l_{0}$ is in the support of the divisor $E_{0}$,
then permuting points $q^{(i)}_{j}$\rq{}s, we may assume that $q_{1}^{(1)}=[0:1]$.
(Here recall that we identify $l_{0}$ with $\mbP^{1}$ via (\ref{eq-l0P}).)
We define $\widetilde{\mbH_{0}}^{\dagger}$ to be $\widetilde{\mbH_{0}}$ if $[0:1]\in \Supp\, E_{0}$,
an the blowing-up of $\widetilde{\mbH_{0}}$ at the point $\left[
\begin{pmatrix}
1 & 0 \\
0 & 0 
\end{pmatrix}
\right]$
if $[0:1]\notin \Supp\, E_{0}$.
(Here, by abuse of notation, we are writing $\left[
\begin{pmatrix}
1 & 0 \\
0 & 0 
\end{pmatrix}
\right]$
for the pre-image of the point $\left[
\begin{pmatrix}
1 & 0 \\
0 & 0 
\end{pmatrix}
\right]\in\mbH_{0}$
by the map $\beta:\widetilde{\mbH_{0}}\to \mbH_{0}$.
Note that in the case $[0:1]\notin \Supp\, E_{0}$, the point 
$\left[
\begin{pmatrix}
1 & 0 \\
0 & 0 
\end{pmatrix}
\right]\in\mbH_{0}$ is not equal to any $[\sigma_{j}^{(i)}]$ 
($1\leq i \leq c,\, 1\leq j\leq e^{(i)}$).)
In the latter case, we denote by $C_{0}$ the exceptional curve of $\widetilde{\mbH_{0}}^{\dagger}\to
\widetilde{\mbH_{0}}$.
The strict transforms of the curves $C_{j}^{(i)}$ on $\widetilde{\mbH_{0}}$ to $\widetilde{\mbH_{0}}^{\dagger}$
are denoted by the same letters.
Denote $\widetilde{\mu_{0}}^{\dagger}$ the composed morphism
$\widetilde{\mbH_{0}}^{\dagger}
\to \widetilde{\mbH_{0}}\xrightarrow{\widetilde{\mu_{0}}} \overline{\mbO}_{(l_{0},p_{0})}. $
The action of $\mathfrak{S}_{1}$ on $\widetilde{\mbH_{0}}$ lifts to $\widetilde{\mbH_{0}}^{\dagger}$,
and we have a diagram
\[
\xymatrix@R=15pt@C=40pt{
\widetilde{\mbH_{0}}^{\dagger}  \ar[r] \ar[d] \ar@/_30pt/[dd]_-{\widetilde{\mu_{0}}^{\dagger}}
& \widetilde{\mbH_{0}}  \ar[r]^-{\beta} \ar[d] 
& \mbH_{0} \\
{ \widetilde{\mbH_{0}}^{\dagger}/\mathfrak{S}_{1}  } \ar[r] \ar[d] & {\widetilde{\mbH_{0}}/\mathfrak{S}_{1}  }  \ar[dl]& \\
\overline{\mbO}_{(l_{0},p_{0})}.  & & \\
}
\]
The variety $\widetilde{\mbH_{0}}^{\dagger}/\mathfrak{S}_{1}$ is smooth and the induced morphism 
$\widetilde{\mbH_{0}}^{\dagger}/\mathfrak{S}_{1} \to \overline{\mbO}_{(l_{0},p_{0})}$
gives an isomorphism of $\mbO_{(l_{0},p_{0})}$ with its inverse image,
so $\widetilde{\mbH_{0}}^{\dagger}/\mathfrak{S}_{1}$ is a smooth model of   $\overline{\mbO}_{(l_{0},p_{0})}$.
Finally, using the homogeneity of $\mbL$, we construct a smooth model of $\overline{\mbO}$.
The group $\mbH_{0}^{\times}=\Aut\,(l_{0},p_{0})$ acts naturally on $\mbH_{0}$
from the left.
Since the points of the form
$
\left[
\begin{pmatrix}
1 & * \\
0 & 0
\end{pmatrix}
\right]
$
of $\mbH_{0}$
are fixed points of this action, the blowing-ups $\widetilde{\mbH_{0}}$ 
and $\widetilde{\mbH_{0}}^{\dagger}$ 
have  induced left action of $\mbH_{0}^{\times}$.
Recall that we have an isomorphism
$\GLV/\mathfrak{S}\simeq \mbL$ (see (\ref{eq-frakS}) for the definition of the group $\mathfrak{S}$).
By associating to $g\in \mathfrak{S}$ the automorphism of $\mrH^{0}(l_{0},\mcO_{l_{0}}(1))$ induced from the morphism
$g:V\to V(=\mrH^{0}(\Pn,\mcO(1)))$,
we have a group homomorphism
\begin{equation}\label{eq-Sto}
\mathfrak{S}\to H_{0}^{\times}\subset \Aut\,\mrH^{0}(l_{0},\mcO_{l_{0}}(1)).
\end{equation}
Via this homomorphism (and the natural map $H_{0}^{\times}\to \mbH_{0}^{\times}$),
the group $\mathfrak{S}$ also acts on $\widetilde{\mbH_{0}}^{\dagger}$ from the left.
We define a right action of $\mathfrak{S}$ on $\GLV\times \widetilde{\mbH_{0}}^{\dagger}$ by
$(g,h)g\rq{}:=(g\cdot g\rq{},g\rq{}^{-1}\cdot h)$,
where $g\in \GLV$, $h\in \widetilde{\mbH_{0}}^{\dagger}$ and $g\rq{}\in \mathfrak{S}$.
The quotient space $\GLV\times_{\mathfrak{S}} \widetilde{\mbH_{0}}^{\dagger}$ of $\GLV\times \widetilde{\mbH_{0}}^{\dagger}$
by this action is a fiber space over $\mbL$ with fibers isomorphic to 
$\widetilde{\mbH_{0}}^{\dagger}$.
The morphism
\[
\GLV \times \widetilde{\mbH_{0}}^{\dagger}\to \overline{\mbO} \subset \mbK
\]
defined by$(g,h)\mapsto g\widetilde{\mu_{0}}^{\dagger}(h)$ factors as
\begin{equation}\label{eq-Obar}
\GLV \times \widetilde{\mbH_{0}}^{\dagger} \to \GLV \times_{\mathfrak{S}} 
\widetilde{\mbH_{0}}^{\dagger} \to \overline{\mbO}.
\end{equation}
By letting $\mathfrak{S}_{1}$ act trivially on the first factor,
we have a right $\mathfrak{S}_{1}$-action on $\GLV \times \widetilde{\mbH_{0}}^{\dagger}$.
This action induces a right action of $\mathfrak{S}_{1}$ on $\GLV \times_{\mathfrak{S}} \widetilde{\mbH_{0}}^{\dagger}$.
The right arrow in (\ref{eq-Obar}) factors as 
\begin{equation}\label{eq-Obar2}
\GLV \times_{\mathfrak{S}} \widetilde{\mbH_{0}}^{\dagger}
\to
\left.
\left(
\GLV \times_{\mathfrak{S}} \widetilde{\mbH_{0}}^{\dagger}
\right)\right/
\mathfrak{S}_{1}
\to \overline{\mbO}.
\end{equation}
The right arrow in (\ref{eq-Obar2}) is birational, and the variety $\left.\left(
\GLV \times_{\mathfrak{S}} \widetilde{\mbH_{0}}^{\dagger}
\right)\right/
\mathfrak{S}_{1}$ is smooth.

We have the following diagram :
\[
\xymatrix@R=15pt@C=40pt{
\GLV\times_{\mathfrak{S}}\widetilde{\mbH_{0}}^{\dagger}  \ar[r]^-{\eta_{0}} \ar[d]^-{\varphi}&
\GLV \times_{\mathfrak{S}}\mbH_{0} \ar@/^24pt/[lddddd]^-{\theta} \\
{(\GLV\times_{\mathfrak{S}}\widetilde{\mbH_{0}}^{\dagger})/\mathfrak{S}_{1}  } \ar[d]^-{\psi}&
\\
\overline{\mbO} \ar@{}[d]|{\bigcap} &  \\
\mbK \ar[d]^-{\tau^{(i)}_{1}} \ar@/_24pt/[dd]_-{\tau} & \\
\mbK^{(i)} \ar[d]^-{\tau^{(i)}_{2}} &  \\
\mbL \ar[r]^-{\pi_{1}} \ar[d]^-{\pi_{2}}& \mbG  \\
\mbP^{n} &  
}
\]
where $\varphi$, $\psi$, $\eta_{0}$,  $\theta$ are natural morphisms.
For simplicity of notation, put
\[
\widetilde{\mcP}^{\dagger}:=\GLV\times_{\mathfrak{S}}\widetilde{\mbH_{0}}^{\dagger}
\quad
\text{and}
\quad
\mcP:=\GLV \times_{\mathfrak{S}}\mbH_{0}.
\]
The variety $\widetilde{\mcP}^{\dagger}/\mathfrak{S}_{1}$ is smooth and proper,
and the birational morphism $\psi$ induces an isomorphism $\psi^{-1}(\mbO)\xrightarrow{\sim} \mbO$.


Now that we have constructed the smooth compactification $\widetilde{\mcP}^{\dagger}/\mathfrak{S}_{1}$
of $\mbO$, we next construct  a parameter space and its compactifications 
over $\overline{\mbfS}$
of quadruples $(l,p.E,B)$,
where $(l,p,E)\in \mbO$  and 
$B\in \overline{\mbfS}=\mbP_{*}\left(\mrH^{0}(\Pn,\mcO_{\Pn}(d)) \right)$
such that $l\cap B=E$.

Recall that $\mbK^{(i)}=\mbP_{*}\left( S^{e^{(i)}}\pi_{1}^{*}\mcQ\right)$ and 
we write $\mcO_{\mbKi}(1)$ for $\mcO_{ \mbP_{*}\left( S^{e^{(i)}}\pi_{1}^{*}\mcQ\right) }(1)$.
From the natural injective map
$\mcO_{\mbKi}(-1)\to (\pi_{1}\circ \tau_{2}^{(i)})^{*} S^{e^{(i)}}\mcQ$,
we obtain, by taking a tensor product for $1\leq i \leq c$,
an injective map
\begin{equation*}
\bigotimes_{i=1}^{c}
\tau_{1}^{(i)*}
\mcO_{\mbKi}(-m^{(i)})
\to
(\pi_{1}\circ \tau)^{*}S^{d}\mcQ.
\end{equation*}
We define a vector bundle $\mcU$ on $\mbK$ to be the kernel of the composite of morphisms
\begin{equation}\label{eq-defU}
S^{d}V\otimes\mcO_{\mbK}
\to (\pi_{1}\circ \tau)^{*}S^{d}\mcQ
\to
\frac{
 (\pi_{1}\circ \tau)^{*}S^{d}\mcQ
}
{
\bigotimes_{i=1}^{c}
\tau_{1}^{(i)*}
\mcO_{\mbKi}(-m^{(i)})
}.
\end{equation}
We have
\begin{equation}
\rank \; \mcU=\dim S^{d}V-d,
\end{equation}
and
we have an exact sequence (cf. Notation \ref{notation-M})
\begin{equation}\label{eq-MUt}
0 \to (\pi_{1}\circ \tau)^{*}\mcM^{d}_{\mbG}
\to \mcU \to
\bigotimes_{i=1}^{c}
\tau_{1}^{(i)*}
\mcO_{\mbKi}(-m^{(i)}) \to 0.
\end{equation}
Define a variety $\overline{R}_{\mbK}$ over $\mbK$ to be $\mbP_{*}(\mcU)$.
If $?=\overline{\mbO},\, \widetilde{\mcP}^{\dagger},\,
\widetilde{\mcP}^{\dagger}/\mathfrak{S}_{1}$,
then we put $\overline{R}_{?}:=(\overline{R}_{\mbK})\times_{\mbK}?$, 
the fiber product of $\overline{R}_{\mbK}$ and $?$ over $\mbK$.
But for simplicity of notation we write $\overline{R}$ for 
$\overline{R}_{ \widetilde{\mcP}^{\dagger}/\mathfrak{S}_{1} }$.
Since the vector bundle $\mcU$ is a subbundle of $S^{d}V\otimes\mcO_{\mbK}$, the variety $\overline{R}_{\mbK}$ is a subvariety
of $\mbK\times \overline{\mbfS}$.
We have the following commutative diagram:
\begin{equation}\label{eq-bigdiag}
\xymatrix@R=15pt@C=40pt{
&\overline{R}_{ \widetilde{\mcP}^{\dagger} } \ar[r]^-{\eta_{3}} \ar[d]^-{\rho} 
& \widetilde{\mcP}^{\dagger} \ar[d]^-{\varphi} \ar[r]^-{\eta_{0}}
& \mcP \ar@/^24pt/[lddddd]^-{\theta} 
 \\
&\overline{R}\ar[r]^-{\bar{\eta}_{3}} \ar[d]
& {\widetilde{\mcP}^{\dagger} /\mathfrak{S}_{1}  } \ar[d]^-{\psi}
&
 \\
&\overline{R}_{ \overline{\mbO} }  \ar[r] \ar@{}[d]|{\bigcap}  
& \overline{\mbO} \ar@{}[d]|{\bigcap} & \\
&\overline{R}_{\mbK}  \ar[dd]^{\eta}    \ar[r]^-{\lambda}                  
& \mbK \ar[d]_-{\tau^{(i)}_{1}} \ar@/^15pt/[dd]^-{\tau} &  \\
&{}& \mbK^{(i)} \ar[d]_-{\tau^{(i)}_{2}} &  \\
&\mbL\times \overline{\mbfS} \ar[r] \ar[d]^{\xi} 
& \mbL \ar[r]^-{\pi_{1}} \ar[d]^-{\pi_{2}}& \mbG & \\
\mathbf{D} \ar@{}[r]|*{\subset}&\Pn\times \overline{\mbfS} \ar[r] \ar[d]^-{pr_{\overline{\mbfS}}} & \mbP^{n} &  \\
&\overline{\mbfS} & &,
}
\end{equation}
where $\eta$ is the composite
$\overline{R}_{\mbK}\hookrightarrow \mbK\times \overline{\mbfS}
 \xrightarrow{\tau\times \id_{\overline{\mbfS}}} \mbL\times \overline{\mbfS}$,
$\lambda$ is the projection $\overline{R}_{\mbK}=\mbP_{*}(\mcU)\to\mbK$,
$\eta_{3}$, $\bar{\eta}_{3}$ are base-changes of $\lambda$,
$\rho:\overline{R}_{\widetilde{\mcP}^{\dagger}}\to \overline{R}$ is the base-change of $\varphi$,
and $\xi=\pi_{2} \times\id_{\overline{\mbfS}}$.

It is the variety $\overline{R}$ that we wanted to construct.
By construction, we have
\begin{equation}\label{eq-dimRS}
\dim \overline{R}=\dim \overline{\mbfS} +2n-d+1.
\end{equation}

\subsection{Calculation of Canonical bundles}\label{subsect-calcano}
In this section, we carry out the computation of the canonical bundle of the variety 
$\overline{R}$.

We introduce some notation for line bundles.
If $\mcV\to \mbH_{0}$ is an $\mathfrak{S}$-equivariant vector bundle,
we obtain the induced vector bundle $\GLV\times_{\mathfrak{S}}\mcV$ on $\mcP(:=\GLV \times_{\mathfrak{S}}\mbH_{0}).$
The induced line bundle on $\mcP$ obtained from the $\mathfrak{S}$-equivariant line bundle
$\mcO_{\mbH_{0}}(1)$ on $\mbH_{0}$ is denoted by $\mcO_{\mcP}(1)$.

Composing the group homomorphism (\ref{eq-Sto}) and the characters
$H_{0}^{\times}\to\mbC^{\times}$ given by
$\begin{pmatrix}
a & b \\
0 & d
\end{pmatrix}
\mapsto a$
and by
$\begin{pmatrix}
a & b \\
0 & d
\end{pmatrix}
\mapsto d$
respectively, we obtain characters $\alpha$ and $\delta :\mathfrak{S}\to \mbC^{\times}$ of 
$\mathfrak{S}$ respectively.
The character of $\mathfrak{S}$ given by $\mathfrak{S}\subset\GLV\xrightarrow{\det}\mbC^{\times}$
is denoted by $\mathbf{det}$.
For a character $\chi$ of $\mathfrak{S}$,
we write $(\mbC_{\mbH_{0}})_{\chi}$ for the trivial line bundle on $\mbH_{0}$ with the action given by $\chi$,
and $\mcL_{\chi}$ for the line bundle on $\mcP$ induced from $(\mbC_{\mbH_{0}})_{\chi}$.
In other words, the line bundle $\mcL_{\chi}$ is the quotient space of 
$\GLV \times (\mbC \times \mbH_{0})$
by the relation 
\begin{equation}\label{eq-equiv}
(g g\rq{}, a, z) \sim (g, \chi(g\rq{})a, g\rq{}z),
\end{equation}
where $g\rq{}\in \mathfrak{S}$.
For characters $\chi_{1}$ and $\chi_{2}$ of $\mathfrak{S}$, we write $\chi_{1}+\chi_{2}$
additively for their product.
\begin{lem}\label{lem-lbs}
(1) The line bundle $\mcL_{\mathbf{det}}$  on $\mcP$ is a trivial line bundle.

\noindent (2) We have $(\pi_{1}\circ \theta)^{*}\mcO_{\mbG}(1) \simeq \mcL_{\alpha+\delta}$,
where $\mcO_{\mbG}(1)=\det\mcQ$.

\noindent (3) We have $(\pi_{2}\circ \theta)^{*}\mcO_{\Pn}(1) \simeq \mcL_{\delta}$.

\noindent (4) We have 
\[
K_{\mcP/\mbL}  \simeq \mcL_{-2\alpha-\delta}\otimes\mcO_{\mcP}(-3) 
 \simeq (\pi_{1}\circ\theta)^{*}\mcO_{\mbG}(-2)\otimes (\pi_{2}\circ\theta)^{*}\mcO_{\Pn}(1) \otimes \mcO_{\mcP}(-3).
\]

\noindent (5) We have
$K_{\mcP}\simeq (\pi_{1}\circ\theta)^{*}\mcO_{\mbG}(-n-2) \otimes
(\pi_{2}\circ\theta)^{*}\mcO_{\Pn}(-1) \otimes \mcO_{\mcP}(-3)$.
\end{lem}
\proof
(1)
The isomorphism
$ \GLV\times \mbC \times \mbH_{0}
\to
\GLV\times \mbC \times \mbH_{0}$
given by $(g,a,z)\mapsto (g, (\det g)a,z)$ transforms the equivalence relation
(\ref{eq-equiv}),
with $\chi$ replaced by $\mathbf{det}$, on the source to the equivalence relation
\begin{equation*}
(gg',a,z)\sim (g,a, g'z)
\end{equation*}
on the target.
From this, (1) follows.

(2)
The group $\mathfrak{S}$ acts on $\mrH^{0}(l_{0}, \mcO(1))$.
The rank $2$ vector bundle $\pi_{1}^{*}\mcQ$ on $\mbL\simeq \GLV/\mathfrak{S}$ is 
isomorphic to $\GLV\times_{\mathfrak{S}}\mrH^{0}(l_{0},\mcO(1))$.
The $\mathfrak{S}$-action on $\wedge^{2}\mrH^{0}(l_{0},\mcO(1))$
is given by the character $\alpha+\delta$.
This proves (2).

(3)
This can be proved similarly as (2).

(4)
We have the Euler sequence
\[
0\to \Omega_{\mbH_{0}}(1)\to
H_{0}^{\vee} \otimes \mcO_{\mbH_{0}}\to \mcO_{\mbH_{0}}(1)
\to 0.
\]
From this, we have an isomorphism
\[
K_{\mbH_{0}}\simeq
(\det H_{0})^{\vee}\otimes \mcO_{\mbH_{0}}(-3).
\]

This implies that $K_{\mcP/\mbL} \simeq \mcL_{-2\alpha-\delta}
\otimes \mcO_{\mcP}(-3)$,
which is isomorphic to
$(\pi_{1}\circ\theta)^{*}\mcO_{\mbG}(-2)\otimes (\pi_{2}\circ\theta)^{*}\mcO_{\Pn}(1) \otimes \mcO_{\mcP}(-3)$
by (2) and (3).

(5)
This follows from (4) and the isomorphism
$K_{\mbL} \simeq\pi_{1}^{*}\mcO_{\mbG}(-n)
\otimes \pi_{2}^{*}\mcO_{\Pn}(-2)$.
\endproof
If $C$ is an $\mathfrak{S}$-invariant curve on $\mbH_{0}$ 
or $\widetilde{\mbH_{0}}^{\dagger}$, then $\GLV\times_{\mathfrak{S}}C$
is a divisor on $\mcP$  or $\widetilde{\mcP}^{\dagger}$.
By abuse of notation, we write $C$ for $\GLV\times_{\mathfrak{S}}C$.
For example, recall the lines $A$ and $\Delta$ on $\mbH_{0}$ defined in
(\ref{eq-ADelta}).
The divisors $\GLV\times_{\mathfrak{S}}A$ and $\GLV\times_{\mathfrak{S}}\Delta$
on $\mcP$ are also written as $A$ and $\Delta$ respectively.
With this notation, we have the following:
\begin{lem}\label{lem-lbs2}
(1) We have $\mcO_{\mcP}(-1)\otimes\mcO_{\mcP}(A) \simeq \mcL_{\alpha}$.

\noindent (2) We have $\mcO_{\mcP}(-1)\otimes\mcO_{\mcP}(\Delta) \simeq \mcL_{\delta}$.
\end{lem}
\proof
Since $A, \Delta\subset \mbH_{0}$ are defined by $a=0$ and $d=0$ respectively,
we have
\begin{equation*}
\left(\mbC_{\mbH_{0}}\right)_{-\alpha}\simeq
\mcO_{\mbH_{0}}(1)\otimes
\mcO_{\mbH_{0}}(-A)
\quad
\text{and}
\quad
\left(\mbC_{\mbH_{0}}\right)_{-\delta}\simeq
\mcO_{\mbH_{0}}(1)\otimes
\mcO_{\mbH_{0}}(-\Delta).
\end{equation*}
The lemma follows from these.
\endproof
The strict transforms to  $\widetilde{\mbH_{0}}^{\dagger}$ 
of the lines $A$ and $\Delta$ on $\mbH_{0}$ are denoted by the same letters.
The strict transforms to $\widetilde{\mbH_{0}}^{\dagger}$ of the curves $C_{j}^{(i)}$
on $\widetilde{\mbH_{0}}$ are denoted by the same letters.
By construction, we have isomorphisms of sheaves on $\wtPd$
\begin{align}
\eta_{0}^{*}\mcO_{\mcP}(A)\simeq \mcO_{\wtPd}(A),  \label{align-etaAA}\\
\eta_{0}^{*}\mcO_{\mcP}(\Delta)
\simeq
\begin{cases}
\mcO_{\wtPd}\left(\Delta
+\sum_{
\substack{1\leq i \leq c \\ 1\leq j \leq e^{(i)}}
}C^{(i)}_{j}
\right) & \text{if $[0:1]\in\Supp E_{0}$} \\
\mcO_{\wtPd}\left(\Delta
+\sum_{
\substack{1\leq i \leq c \\ 1\leq j \leq e^{(i)}}
}C^{(i)}_{j}
+C_{0}
\right)
& \text{if $[0:1]\notin\Supp E_{0}$}, 
\end{cases}
\\
K_{\widetilde{\mcP}^{\dagger}}
\simeq
\begin{cases}
\eta_{0}^{*}K_{\mcP}
\otimes
\mcO_{\widetilde{\mcP}^{\dagger}}\bigl(
\sum_{
\substack{1\leq i \leq c \\ 1\leq j \leq e^{(i)}}
}C^{(i)}_{j}
\bigr) & \text{if $[0:1]\in\Supp E_{0}$} \\
\eta_{0}^{*}K_{\mcP}
\otimes
\mcO_{\widetilde{\mcP}^{\dagger}}\bigl(
\sum_{
\substack{1\leq i \leq c \\ 1\leq j \leq e^{(i)}}
}C^{(i)}_{j}
\bigr) \otimes \mcO_{\widetilde{\mcP}^{\dagger}}(C_{0}) & \text{if $[0:1]\notin\Supp E_{0}$}. 
\end{cases}
\end{align}
By Lemma \ref{lem-lbs}(5), 
the line bundle 
$\eta_{0}^{*}K_{\mcP}
\otimes
\mcO_{\widetilde{\mcP}^{\dagger}}\bigl(
\sum_{
\substack{1\leq i \leq c \\ 1\leq j \leq e^{(i)}}
}C^{(i)}_{j}
\bigr)$ in the above expression 
is isomorphic to
\begin{equation}\label{align-lb}
 (\pi_{1}\circ \theta\circ \eta_{0})^{*}
\mcO_{\mbG}(-n-2)
\otimes
(\pi_{2}\circ\theta\circ\eta_{0})^{*}
\mcO_{\Pn}(-1)
\otimes \eta_{0}^{*}\mcO_{\mcP}(-3)
\otimes
\mcO_{\widetilde{\mcP}^{\dagger}}\bigl(
\sum_{
\substack{1\leq i \leq c \\ 1\leq j \leq e^{(i)}}
}C^{(i)}_{j}
\bigr).
\end{equation}
In order to make the notation simple,
from now on we omit the pull-back notation of line bundles.
For example, instead of the rigorous expression (\ref{align-lb}),
we will write 
\begin{equation}\label{eq-KtildeP}
\mcO_{\mbG}(-n-2)
\otimes
\mcO_{\Pn}(-1)
\otimes
\mcO_{\mcP}(-3)
\otimes
\mcO_{\widetilde{\mcP}^{\dagger}}\bigl(
\sum_{
\substack{1\leq i \leq c \\ 1\leq j \leq e^{(i)}}
}C^{(i)}_{j}
\bigr).
\end{equation}
After all, we have
\begin{equation}\label{eq-KPdagger}
K_{\widetilde{\mcP}^{\dagger} }\simeq
\begin{cases}
(\text{the line bundle (\ref{eq-KtildeP})} )& \text{if $[0:1]\in\Supp\,E_{0}$} \\
(\text{the line bundle (\ref{eq-KtildeP})} )\otimes 
\mcO_{\widetilde{\mcP}^{\dagger}}(C_{0}) & \text{if $[0:1]\not\in\Supp\,E_{0}$.} 
\end{cases}
\end{equation}
\begin{prop}\label{prop-KR}
The canonical line bundle $K_{\overline{R}_{\widetilde{\mcP}^{\dagger}}}$ 
of $\overline{R}_{ \widetilde{\mcP}^{\dagger} }$
is isomorphic to
\begin{equation}\label{eq-KPRdagger}
\begin{split}
\mcO_{\mbG}\left(
\frac{d(d+1)}{2}-n-2
\right)\otimes\mcO_{\Pn}(-1)\otimes\mcO_{\mcP}(-3) 
\otimes\mcO_{\widetilde{\mcP}^{\dagger}}\bigl(
\sum_{\substack{1\leq i \leq c \\
1\leq j \leq e^{(i)}}
}C_{j}^{(i)}
\bigr)
\\
\otimes\bigotimes_{i=1}^{c}\mcO_{\mbK^{(i)}}(m^{(i)})
\otimes
\mcO_{\overline{\mbfS}}(-\rank\;\mcU)
\end{split}
\end{equation}
if $[0:1]\in \Supp \,E_{0}$, and to
$(\ref{eq-KPRdagger})\otimes\mcO_{ \widetilde{\mcP}^{\dagger} }(C_{0})$
if $[0:1]\notin \Supp \,E_{0}$.
We have
\begin{equation}\label{eq-rhoKR}
\rho^{*}K_{\overline{R}}\simeq
\begin{cases}
K_{\overline{R}_{ \widetilde{\mcP}^{\dagger} }}\otimes
\mcO_{\widetilde{\mcP}^{\dagger}}\left(
-(r-1)(A+C_{1}^{(1)})
\right) & \text{if $[0:1]\in \Supp \,E_{0}$} \\
K_{\overline{R}_{ \widetilde{\mcP}^{\dagger} }}\otimes
\mcO_{\widetilde{\mcP}^{\dagger}}\left(
-(r-1)(A+C_{0} )
\right)& \text{if $[0:1]\notin \Supp \,E_{0}$}.
\end{cases}
\end{equation}
\end{prop}
\proof
See the proof of \cite[Proposition 5.3]{A24}.
\endproof
From (\ref{eq-muback}), we see that the pull-back of $\mcO_{\mbK^{(i)}}(-1)$
to $\widetilde{\mcP}^{\dagger}$ is isomorphic to
\begin{equation}\label{eq-K-1}
\eta_{0}^{*}\mcO_{\mcP}(-e^{(i)})\otimes
\mcO_{\widetilde{\mcP}^{\dagger}}\left(
\sum_{j=1}^{e^{(i)}}C_{j}^{(i)}
\right).
\end{equation}

Recall that $\mbfD\subset \Pn\times \overline{\mbfS}$ (cf. (\ref{eq-mbfD})) is the universal family of degree $d$ hypersurfaces.
We want to study the pull-back of $\mbfD$ to $\overline{R}_{\wtPd}$.
First we study its pull-back to $\overline{R}_{\mbK}$.
Define divisors $I$ and $J^{(i_{0})}$, $1\leq i_{0} \leq c$, of $\overline{R}_{\mbK}$ by
\begin{align*}
I&:=\left\{ \left.
(l,p,E^{(i)}, 1\leq i \leq c; B)\right| l\subset B
\right\} \\
\intertext{and} 
J^{(i_{0})} &:=\left\{ \left.
(l,p,E^{(i)}, 1\leq i \leq c; B)\right| p\in \Supp\, E^{(i_{0})}
\right\}. 
\end{align*}
On $\overline{R}_{\mbK}$, there are natural maps (cf. \ref{eq-MUt})
\[
(pr_{\overline{\mbfS}}\circ \xi \circ \eta)^{*}
\mcO_{\overline{\mbfS}}(-1)
\to\lambda^{*}\mcU
\to \lambda^{*}\bigotimes_{i=1}^{c}\tau_{1}^{(i)*}\mcO_{\mbKi}(-m^{(i)}),
\]
and the vanishing locus of their composite is the divisor $I$.
So we have
\begin{equation}\label{eq-I}
\mcO_{\overline{R}_{\mbK}}(I)
\simeq
(pr_{\overline{\mbfS}}\circ \xi \circ \eta)^{*}
\mcO_{\overline{\mbfS}}(1)\otimes
\lambda^{*}\bigotimes_{i=1}^{c}\tau_{1}^{(i)*}\mcO_{\mbKi}(-m^{(i)}).
\end{equation}
If we define a divisor $\overline{J}^{(i_{0})}$ of $\mbK^{(i_{0})}$ by
\[
\mbK^{(i_{0})}
\supset
\overline{J}^{(i_{0})}
:=\left\{
(l,p,E^{(i_{0})}) \mid p\in \Supp\,E^{(i_{0})}
\right\},
\]
then the divisor $J^{(i_{0})}\subset \overline{R}_{\mbK}$ is the pull-back of $\overline{J}^{(i_{0})}$
to $\overline{R}_{\mbK}$.
On $ \mbK^{(i_{0})} $,
there are natural maps
\[
\mcO_{\mbK^{(i_{0})}}(-1)
\to
(\pi_{1}\circ \tau_{2}^{(i_{0})})^{*}S^{ e^{(i_{0})} }\mcQ
\to
(\pi_{2}\circ \tau_{2}^{(i_{0})})^{*}
\mcO_{\Pn}(e^{(i_{0})}),
\]
where the right arrow is the pull-back by $\tau_{2}^{(i_{0})}$ of the
$e^{(i_{0})}$-symmetric tensor of the natural surjective map
$\pi_{1}^{*}\mcQ\to\pi_{2}^{*}\mcO_{\Pn}(1)$ on $\mbL$.
The vanishing locus of their composite is $\overline{J}^{(i_{0})}$,
so we have
\begin{equation}
\mcO_{ \mbK^{(i_{0})} }( \overline{J}^{(i_{0})})
\simeq
\mcO_{ \mbK^{(i_{0})} }(1)
\otimes(\pi_{2}\circ \tau_{2}^{(i_{0})})^{*}
\mcO_{\Pn}(e^{(i_{0})}).
\end{equation}
The support of the divisors $(\xi\circ \eta)^{*}\mbfD$ 
and $I+\sum_{i=1}^{c}J^{(i)}$ on $\overline{R}_{\mbK}$ are equal.
Taking multiplicity into account, we have
\begin{equation}\label{eq-pullD}
(\xi\circ \eta)^{*}\mbfD=
I+\sum_{i=1}^{c}m^{(i)}J^{(i)}.
\end{equation}
The pull-backs to $\overline{R}$ of the divisors $I$ and $J^{(i)}$ on $\overline{R}_{\mbK}$
are denoted by $I_{\overline{R}}$ and $J^{(i)}_{\overline{R}}$ respectively.
\begin{prop}\label{prop-IJi}
\begin{enumerate}
\item The divisor $I_{\barR}$ on $\barR$ is smooth.
\item The divisor $I_{\barR}+\left(
        \sum_{i=1}^{c}J_{\overline{R}}^{(i)}
        \right)_{\mathrm{red}}$ on $\barR$ is SNC.
\item We have an isomorphism of sheaves on $\barR_{\wtPd}$
\begin{equation}\label{eq-ORIR}
\rho^{*}\mcO_{\barR}\left(
\left(
\sum_{i=1}^{c}J_{\overline{R}}^{(i)}
\right)_{\mathrm{red}}
\right)
\simeq
\begin{cases}
\mcO_{\wtPd}\left(
\Delta
+\sum_{
\substack{1\leq i \leq c \\ 1\leq j \leq e^{(i)} }
}C_{j}^{(i)}
+(r-1)C_{1}^{(1)}
\right)
& \text{if $[0:1]\in\Supp\, E_{0}$} \\
\mcO_{\wtPd}\left(
\Delta
+\sum_{
\substack{1\leq i \leq c \\ 1\leq j \leq e^{(i)} }
}C_{j}^{(i)}
+rC_{0}
\right)
& \text{if $[0:1]\notin\Supp\, E_{0}$} 
\end{cases}
\end{equation}
where we omitted the notation of pull-back to $\overline{R}_{\wtPd}$ on the right-hand side.
\end{enumerate}
\end{prop}
\proof
(1)
Since $I_{\overline{R}}$ is a projective bundle over the smooth variety $\wtPd/\mathfrak{S}_{1}$, 
it is smooth.

(2)
In order to study $(J_{\overline{R}}^{(i)})_{\mathrm{red}}$, we first consider
the divisor $(\tau_{1}^{(i)}\circ\psi\circ\varphi)^{*}\overline{J}^{(i_{0})}$.
By the local description of the map $\widetilde{\mu_{0}}$,
we find that
\begin{equation}\label{eq-tpvarphi}
(\tau_{1}^{(i_{0})}\circ \psi\circ\varphi)^{*}\overline{J}^{(i_{0})}
=
e^{(i_{0})} \Delta+
\sum_{j=1}^{ e^{(i_{0})} }( e^{(i_{0})}-1)C_{j}^{(i_{0})}
+\sum_{
\substack{i\neq i_{0} \\ 1\leq j \leq e^{(i)} }
}e^{(i_{0})}C_{j}^{(i)}
\end{equation}
if $[0:1]\in \Supp\, E_{0}$,
and
\begin{equation}\label{eq-tpvarphi2}
(\tau_{1}^{(i_{0})}\circ \psi\circ\varphi)^{*}\overline{J}^{(i_{0})}
=
e^{(i_{0})} \Delta+
\sum_{j=1}^{ e^{(i_{0})} }( e^{(i_{0})}-1)C_{j}^{(i_{0})}
+\sum_{
\substack{i\neq i_{0} \\ 1\leq j \leq e^{(i)} }
}e^{(i_{0})}C_{j}^{(i)}
+e^{(i_{0})}C_{0}
\end{equation}
if $[0:1]\not\in \Supp\, E_{0}$,
where $\Delta$, $C_{j}^{(i_{0})}$ and $C_{0}$ are divisors on $\widetilde{\mcP}^{\dagger}$
(see the paragraph before and after Lemma \ref{lem-lbs2}).
It follows from this that 
\begin{equation}
\Supp\,
\varphi^{*}
\psi^{*}\sum_{i=1}^{c}\tau_{1}^{(i)*}\overline{J}^{(i)}
=
\begin{cases}
\Delta+\sum_{\substack{1\leq i\leq c \\ 1\leq j \leq e^{(i)} }}C_{j}^{(i)}
& \text{ if $[0:1]\in\Supp\, E_{0}$} \\
\Delta+\sum_{\substack{1\leq i\leq c \\ 1\leq j \leq e^{(i)} }}C_{j}^{(i)}+C_{0}
&\text{ if $[0:1]\notin\Supp\, E_{0}$}. 
\end{cases}
\end{equation}
This shows that the divisor $\left(\varphi^{*}
\psi^{*}\sum_{i=1}^{c}\tau_{1}^{(i)*}\overline{J}^{(i)}
\right)_{\mathrm{red}}$ on $\wtPd$ is SNC.
Since the morphism $\varphi:\wtPd\to\wtPd/\mathfrak{S}_{1}$ is ramified along the divisors $A$ and $C^{(1)}_{1}$
(resp. $A$ and $C_{0}$)
if $[0:1]\in\Supp\, E_{0}$ (resp. $[0:1]\not\in\Supp\, E_{0}$),
we see that
the divisor $\left(
\psi^{*}\sum_{i=1}^{c}\tau_{1}^{(i)*}\overline{J}^{(i)}
\right)_{\mathrm{red}}$ on $\wtPd/\mathfrak{S}_{1}$ is SNC,
and
\begin{equation}\label{eq-descri}
\varphi^{*}
\left(
\left(\psi^{*}\sum_{i=1}^{c}\tau_{1}^{(i)*}\overline{J}^{(i)}\right)_{\mathrm{red}}
\right)
=
\begin{cases}
\Delta+\sum_{i=1}^{c}\sum_{j=1}^{e^{(i)}}C_{j}^{(i)}+(r-1)C_{1}^{(1)} & \text{if $[0:1]\in\Supp\, E_{0}$} \\
\Delta+\sum_{i=1}^{c}\sum_{j=1}^{e^{(i)}}C_{j}^{(i)}+rC_{0} & \text{if $[0:1]\not\in\Supp\, E_{0}$}. 
\end{cases}
\end{equation}
Since $\left(
        \sum_{i=1}^{c}J_{\overline{R}}^{(i)}
        \right)_{\mathrm{red}}=\bar{\eta}_{3}^{*}\left(
\psi^{*}\sum_{i=1}^{c}\tau_{1}^{(i)*}\overline{J}^{(i)}
\right)_{\mathrm{red}}$
and 
$I_{\overline{R}}$ is a projective bundle over  $\wtPd/\mathfrak{S}_{1}$,
(2) follows.

(3) follows from (\ref{eq-descri}).
\endproof
We come to the main result of this section.
\begin{prop}\label{prop-KRIJ}
The line bundle $\rho^{*}K_{\overline{R}}\left(
I_{\overline{R}}+\left(
\sum_{i=1}^{c}J_{\overline{R}}^{(i)}
\right)_{\mathrm{red}}
\right)$
is isomorphic to
\begin{equation}\label{eq-33A}
\begin{split}
\mcO_{\barR}(2I_{\barR})\otimes
\mcO_{\mbG}\left(
\frac{(d-1)(d-2)}{2}-n
\right) 
\otimes
\mcO_{\Pn}\left(
2\sum_{i=1}^{c}e^{(i)}-4
\right) \\
\otimes 
\mcO_{\wtPd}\left(\Delta+
\left(
2\sum_{i=1}^{c}e^{(i)}-r-2
\right)A
\right)
\otimes
\bigotimes_{i=1}^{c}
\left\{
\mcO_{\mbK^{(i)}}(1)
\otimes\mcO_{\mbG}( e^{(i)} )
\right\}^{\otimes (2m^{(i)}-2)}
\\
\otimes
\mcO_{\overline{\mbfS}}(-\rank\;\mcU-1)
\end{split}
\end{equation}
if $[0:1]\in\Supp\, E_{0}$, and
to $(\ref{eq-33A})\otimes\eta^{*}_{3}\mcO_{\wtPd}(2C_{0})$
if $ [0:1]\not\in\Supp\, E_{0}$,
where we omitted the notation of pull-back to $\overline{R}_{\wtPd}$ in (\ref{eq-33A}).
\end{prop}
\proof
We give a proof only in the case $[0:1]\in\Supp\,E_{0}$.

In the proof,
we use the isomorphism
\begin{equation}\label{eq-DelA}
\mcO_{\mcP}(A)\simeq\mcO_{\mbG}(1)\otimes\mcO_{\Pn}(-1)
\otimes\mcO_{\mcP}(1),
\end{equation}
which is a  consequence of Lemma \ref{lem-lbs} and Lemma \ref{lem-lbs2}.

Combining (\ref{eq-rhoKR}) and (\ref{eq-ORIR}), we find
that
$\rho^{*}K_{\overline{R}}\left(
I_{\overline{R}}+\left(
\sum_{i=1}^{c}J_{\overline{R}}^{(i)}
\right)_{\mathrm{red}}
\right)$
is isomorphic to
\begin{equation}\label{eq-KRIR}
\mcO_{\barR}(I_{\barR})
\otimes
K_{\overline{R}_{\wtPd}}\otimes
\mcO_{\wtPd}\left(
\Delta+\sum_{\substack{1\leq i \leq c \\ 1\leq j \leq e^{(i)}}}
C^{(i)}_{j} -(r-1)A
\right).
\end{equation}
Using (\ref{eq-KPRdagger}), we see that the line bundle (\ref{eq-KRIR}) is isomorphic to
\begin{equation}\label{eq-KRIR2}
\begin{split}
\mcO_{\barR}(I_{\barR})\otimes
\mcO_{\mbG}\left(
\frac{d(d+1)}{2}-n-2
\right)\otimes\mcO_{\Pn}(-1)
\otimes\mcO_{\mcP}(-3)
\otimes
\bigotimes_{i=1}^{c}\mcO_{\mbK^{(i)}}(m^{(i)}) \\
\otimes
\mcO_{\barS}(-\rank\,\mcU)
\otimes
\mcO_{\wtPd}\left(
\Delta+2\sum_{\substack{1\leq i \leq c \\ 1\leq j \leq e^{(i)}}}
C^{(i)}_{j} -(r-1)A
\right),
\end{split}
\end{equation}
which is, by (\ref{eq-I}), isomorphic to
\begin{equation}\label{eq-KRIR3}
\begin{split}
\mcO_{\barR}(2I_{\barR})\otimes
\mcO_{\mbG}\left(
\frac{d(d+1)}{2}-n-2
\right)\otimes\mcO_{\Pn}(-1)
\otimes\mcO_{\mcP}(-3)
\otimes
\bigotimes_{i=1}^{c}\mcO_{\mbK^{(i)}}(2m^{(i)}) \\
\otimes
\mcO_{\barS}(-\rank\,\mcU-1)
\otimes
\mcO_{\wtPd}\left(
\Delta+2\sum_{\substack{1\leq i \leq c \\ 1\leq j \leq e^{(i)}}}
C^{(i)}_{j} -(r-1)A
\right).
\end{split}
\end{equation}
Noting (\ref{eq-K-1}), we see that the line bundle (\ref{eq-KRIR3})
is isomorphic to
\begin{align}\label{align-KRIR3}
\mcO_{\barR}(2I_{\barR})\otimes
\mcO_{\mbG}\left(
\frac{d(d+1)}{2}-n-2
\right)\otimes\mcO_{\Pn}(-1)
\otimes\mcO_{\mcP}\left(2\sum_{i=1}^{c}e^{(i)}-3\right) \\
\otimes
\bigotimes_{i=1}^{c}\mcO_{\mbK^{(i)}}(2m^{(i)}-2) 
\otimes
\mcO_{\barS}(-\rank\,\mcU-1)
\otimes
\mcO_{\wtPd}\left(
\Delta-(r-1)A
\right). \notag
\end{align}
By (\ref{eq-DelA}), also noting (\ref{align-etaAA}),
we see that the line bundle (\ref{align-KRIR3}) is isomorphic to
\begin{align}
\mcO_{\barR}(2I_{\barR})\otimes
\mcO_{\mbG}\left(
\frac{d(d+1)}{2}-n+1-2\sum_{i=1}^{c}e^{(i)}
\right)\otimes\mcO_{\Pn}\left(2\sum_{i=1}^{c}e^{(i)}-4\right)\notag  \\
\otimes\mcO_{\wtPd}\left(\Delta+\left(2\sum_{i=1}^{c}e^{(i)}-r-2\right)A\right) \notag \\
\otimes
\bigotimes_{i=1}^{c}\mcO_{\mbK^{(i)}}(2m^{(i)}-2) 
\otimes
\mcO_{\barS}(-\rank\,\mcU-1).
\notag
\end{align}
This is isomorphic to (\ref{eq-33A}).
\endproof

Denote by $\gamma$ the composition
$\overline{R}\to \overline{R}_{\overline{\mbO}}
\subset \overline{R}_{\overline{\mbK}}
\xrightarrow{\eta}\mbL\times\overline{\mbfS}
\xrightarrow{\xi}\Pn\times \overline{\mbfS}$
of morphisms in (\ref{eq-bigdiag}).

\begin{cor}\label{cor-ggen}
Let $Y$ be a smooth projective curve and 
let $b:Y\to \barR$ be a morphism.
Suppose that the image $(\gamma\circ b)(Y)\subset \Pn\times \barS$ of $Y$ by the morphism $\gamma\circ b$
is not contained in the divisor $\mbfD$,
and that the image $(\psi\circ \bar{\eta}_{3} \circ b)(Y)\subset \overline{\mbO}$ of $Y$ by the morphism $\psi\circ \bar{\eta}_{3} \circ b$
intersects the open subset $\mbO$ of $\overline{\mbO}$.
Then $\gamma^{-1}\mbfD(=(\gamma^{*}\mbfD)_{\mathrm{red}})$
is equal to the SNC divisor $I_{\overline{R}}+\left(
\sum_{i=1}^{c}J_{\overline{R}}^{(i)}
\right)_{\mathrm{red}}$
on $\barR$, and
the inequality
\begin{align}\label{align-bKineq}
\deg 
b^{*}K_{\overline{R}}\left(
\gamma^{-1}\mbfD
\right)
\geq
&
\deg \left. \mcO_{\mbG}\left(
\frac{(d-1)(d-2)}{2}-n
\right) \right|_{Y}
+\deg\left.
\mcO_{\Pn}\left(
2\sum_{i=1}^{c}e^{(i)}-4
\right) \right|_{Y} \\
&+\deg\left. \mcO_{\barS}(-\rank\, \mcU-1)\right|_{Y}\notag
\end{align}
holds,
where $?|_{Y}$ denotes the pull back of the line bundle $?$ to $Y$.
\end{cor}
\proof
We have $\gamma^{-1}\mbfD=I_{\overline{R}}+\left(
\sum_{i=1}^{c}J_{\overline{R}}^{(i)}
\right)_{\mathrm{red}}$ by (\ref{eq-pullD}), and it is a SNC divisor of $\barR$ by Proposition \ref{prop-IJi}.

Next we prove the inequality.
Replacing $Y$ by a finite cover of $Y$ if necessary,
we may assume that the morphism $b$ factors as $Y\to\barR_{\wtPd}\xrightarrow{\rho}\barR$.
We consider the degrees of line bundles appearing in (\ref{eq-33A}) pulled back to $Y$.
We have $\deg \left. \mcO_{\barR}(2I_{\barR})\right|_{Y}\geq 0$ because $(\gamma\circ g)(T) \not \subset \mbfD$.
By (\ref{eq-assumption}) and (\ref{eq-rdineq}),
we have $2\sum_{i=1}^{c}e^{(i)}-r-2\geq 0$.
So the pull-back to $Y$ of the line bundle
$\mcO_{\wtPd}\left(\Delta+
\left(
2\sum_{i=1}^{c}e^{(i)}-r-2
\right)A
\right)$ have non-negative degree (here we use the assumption $(\psi\circ \bar{\eta}_{3} \circ b)(Y)\not\subset \overline{\mbO}\setminus \mbO$).
By \cite[Proposition 2.2 (2)]{A24},
the line bundle $\mcO_{\mbKi}(1)\otimes (\pi_{1}\circ\tau_{2}^{(i)})^{*}\mcO_{\mbG}(e^{(i)})$ on $\mbKi$
is globally generated.
So we have $\deg \left. \mcO_{\mbK^{(i)}}(1)
\otimes\mcO_{\mbG}( e^{(i)} ) \right|_{Y} \geq 0$.
The inequality (\ref{align-bKineq}) follows from these inequalities and Proposition \ref{prop-KRIJ}.
\endproof

\section{Proof of Theorem \ref{thm-core}} \label{section-proofcore}

Suppose that the statement (ii) of Proposition \ref{prop-half} holds.
We can find a pointed line $(l_{0}, p_{0})$ and a degree $d$ effective divisor
 $E_{0}$ on $l_{0}$ with $p_{0} \notin \Supp \, E_{0}$
such that for a general point $y\in \mcY$, the divisor
$D_{\sigma'}|_{l_{y}}$ on the pointed line $(l_{y}, x)$ is equivalent to the divisor $E_{0}$
on $(l_{0}, p_{0})$,
where see Proposition \ref{prop-half} (ii) for the notation $D_{\sigma'}|_{l_{y}}$.

If $\sharp \Supp\, E_{0} \leq 2$, then
(ii) of Theorem \ref{thm-core} holds and we are done.
Now we suppose that $\sharp \Supp\, E_{0} \geq 3$,
and will derive the inequality (\ref{eq-ineqcore}).
For this divisor $E_{0}$ on $(l_{0}, p_{0})$,
we will use notation and results about moduli spaces of divisors
on pointed lines obtained in \S \ref{sect-lines}.

\subsection{Normal bundle $\mcN_{\tilde{f}}$}
In this subsection, we introduce a big diagram describing our situation,
and define a normal bundle, the lower bound of the degree of which
is a key to the inequality (\ref{eq-ineqcore}).

Recall the following part of the diagram (\ref{eq-bigdiag}):

\begin{equation}\label{eq-frombigdiag}
\xymatrix@R=15pt@C=40pt{
&\overline{R}_{ \widetilde{\mcP}^{\dagger} }  \ar[d]^-{\rho} 
&
 \\
&\overline{R}\ar[r]^-{\bar{\eta}_{3}} \ar[ddd]^-{\gamma}
& {\widetilde{\mcP}^{\dagger} /\mathfrak{S}_{1}  } \ar[d]^-{\psi}
 \\
&
& \overline{\mbO} \\
&
& \\
\mathbf{D} \ar@{}[r]|*{\subset}&\Pn\times \overline{\mbfS} \ar[r] \ar[d]^-{pr_{\overline{\mbfS}}} & \mbP^{n} &  \\
&\overline{\mbfS} & .
}
\end{equation}
Here recall that  $\gamma$ is the composition
$\overline{R}\to \overline{R}_{\overline{\mbO}}
\subset \overline{R}_{\overline{\mbK}}
\xrightarrow{\eta}\mbL\times\overline{\mbfS}
\xrightarrow{\xi}\Pn\times \overline{\mbfS}$
of morphisms in (\ref{eq-bigdiag}).
Recall also the morphism $\alpha:C\times\Sigma^{sm}\to C\times \overline{\mbfS}$
in (\ref{eq-alpha}).
We denote by $\alpha_{1}$ the composition
\begin{equation}\label{eq-alpha1}
C\times\Sigma^{sm}\xrightarrow{\alpha} C\times \overline{\mbfS}
\xrightarrow{pr_{\overline{\mbfS}}}\overline{\mbfS}.
\end{equation}
Taking the base change of the diagram (\ref{eq-frombigdiag})
by the morphisms
\[
C\times \Sigma' \xrightarrow{\id_{C}\times\beta}
C\times \Sigma^{sm}
\xrightarrow{\alpha_{1}}
\overline{\mbfS},
\]
we obtain the following diagram (except the morphism $\tilde{f}$).
\begin{equation}\label{eq-hugediag}
\xymatrix@R=20pt@C=30pt{
&\left( \overline{R}_{ \widetilde{\mcP}^{\dagger} } \right)_{C\times \Sigma'} \ar[r] \ar[d]
&\left( \overline{R}_{ \widetilde{\mcP}^{\dagger} } \right)_{C\times \Sigma^{sm}} \ar[r] \ar[d]
&\overline{R}_{ \widetilde{\mcP}^{\dagger} }  \ar[d]^-{\rho} 
&
\\
&\overline{R}_{C\times \Sigma'} \ar[r]^-{b_{2}} \ar[dd]^-{\gamma'}
&\overline{R}_{C\times \Sigma^{sm}} \ar[r]^-{b_{1}} \ar[dd]
&\overline{R}\ar[r]^-{\bar{\eta}_{3}} \ar[dd]^-{\gamma}
& {\widetilde{\mcP}^{\dagger} /\mathfrak{S}_{1}  } \ar[d]^-{\psi}
\\
&
&
&
& \overline{\mbO} 
\\
\mcY \ar@/^10pt/[ruu]^-{\tilde{f}} \ar[r]^-{f} \ar@/_10pt/[rd]^-{\tilde{p}_{1}} \ar@/_15pt/[rdd]_-{p_{1}}
&\Pn\times C\times\Sigma' \ar[r] \ar[d]
&\Pn\times C\times\Sigma^{sm} \ar[r] \ar[d]
&\Pn\times \overline{\mbfS} \ar[r] \ar[d]^-{pr_{\overline{\mbfS}}} & \mbP^{n} 
&  
\\
&C\times \Sigma' \ar[r]^-{\id_{C}\times \beta} \ar[d]
&C\times \Sigma^{sm} \ar[r]^-{\alpha_{1}} \ar[d]
&\overline{\mbfS} 
&  
\\
&\Sigma' \ar[r]^-{\beta}
&\Sigma^{sm}
&
&
}
\end{equation}
Here 
$\left( \overline{R}_{ \widetilde{\mcP}^{\dagger} } \right)_{?}$
and $\overline{R}_{?}$ denote
$\left( \overline{R}_{ \widetilde{\mcP}^{\dagger} } \right)\times_{\overline{\mbfS}} {?}$
and
$\overline{R}\times_{\overline{\mbfS}}{?}$
respectively for $?=C\times \Sigma', \,C\times\Sigma^{sm}$,
the morphisms $b_{1}$ and $b_{2}$
are base-changes of the morphisms $\alpha_{1}$ and $\id_{C}\times \beta$ respectively,
$\gamma'$ is the base-change of $\gamma$,
and $\tilde{p}_{1}$ is the composite of the morphism $f$ and the 
projection $\Pn\times C \times \Sigma' \to C \times \Sigma'$.

Recall that we defined a  divisor $I_{\overline{R}}\subset \overline{R}$ in \S \ref{subsect-calcano};
we denote by $I_{C\times \Sigma'}$ its pull-back to $\overline{R}_{C\times \Sigma'}$.
By construction, the open subscheme 
$(\psi\circ\bar{\eta}_{3}\circ b_{1}\circ b_{2})^{-1}(\mbO)
\setminus I_{C\times \Sigma'} \subset \overline{R}_{C\times \Sigma'}$
parametrized quadruples $(c, \sigma', l, z)$,
where $(c,\sigma')\in C\times \Sigma'$ and $l\subset \Pn$ is a line,
and $z\in l$ is a point
such that $l\nsubset D'_{\sigma'}|_{\Pn\times \{c\}}$ and the divisor $D'_{\sigma'}|_{l\times\{c\}}$ on the pointed line
$(l,z)(\simeq (l\times\{c\}, (z,c)))$ is equivalent to the divisor $E_{0}$ on $(l_{0}, p_{0})$.
Since we are in the situation where (ii) of Proposition \ref{prop-half} holds,
we obtain a rational map $\tilde{f}:\mcY\to \overline{R}_{C\times \Sigma'}$
such that for a general point $y\in \mcY$, 
$\tilde{f}(y)$ is a point of $(\psi\circ\bar{\eta}_{3}\circ b_{1}\circ b_{2})^{-1}(\mbO)
\setminus I_{C\times \Sigma'} \subset \overline{R}_{C\times \Sigma'}$
representing the quadruple $(c, \sigma', l_{y}, z)$,
where $c, \sigma', z$ are defined by $(z,c,\sigma')=f(y)$,
and $l_{y}$ is the line given in the statement (ii) of Proposition \ref{prop-half}.
By shrinking $\Sigma'$ if necessary, we may assume that $\tilde{f}$ is a morphism.
\begin{lem}
The divisor $\gamma'^{*}\mcS_{\Sigma'}$ is a smooth divisor on $\overline{R}_{C\times \Sigma'}$,
and the divisor
$\left(
\gamma'^{*}\mcD'
\right)_{\mathrm{red}}$
is equal to the SNC divisor
$(b_{1}\circ b_{2})^{*}\left(
I_{\overline{R}}+\left(
\sum_{i=1}^{c}J_{\overline{R}}^{(i)}
\right)_{\mathrm{red}}
\right).$
Moreover
the divisor 
$\left(
\gamma'^{*}\mcD'
\right)_{\mathrm{red}}
+\gamma'^{*}\mcS_{\Sigma'}$
on $\overline{R}_{C\times \Sigma'}$
is SNC.
Here see \S \ref{section-setting} for the definition of $\mcD'$
and $\mcS_{\Sigma'}$.
\end{lem}
\proof
The divisor $\mcD'$ is equal to the pull-back of the divisor
$\mbfD\subset \Pn\times\overline{\mbfS}$
to $\Pn\times C \times \Sigma'$.
So we have $\gamma'^{*}\mcD'=(\gamma\circ b_{1} \circ b_{2})^{*}(\mbfD)$.
By Corollary \ref{cor-ggen},
$(\gamma^{*}\mbfD)_{\mathrm{red}}$ is 
equal to the SNC divisor $I_{\overline{R}}+\left(
\sum_{i=1}^{c}J_{\overline{R}}^{(i)}
\right)_{\mathrm{red}}$
on $\barR$.
Since  $b_{1}\circ b_{2}$ is a smooth morphism,
$\left(
\gamma'^{*}\mcD'
\right)_{\mathrm{red}}$
is also SNC.

The divisor $\gamma'^{*}\mcS_{\Sigma'}$ is the pull-back of the divisor
$S\times \Sigma'\subset C\times \Sigma'$ to $\overline{R}_{C\times \Sigma'}$.
For each point $s\in S$, the morphism $\{s\}\times \Sigma' \to \overline{\mbfS}$
is smooth,
so 
$\barR_{\{s\}\times\Sigma'}:=\text{
(the pull-back of the divisor $\{s\}\times \Sigma' \subset C\times \Sigma'$ to $\overline{R}_{C\times \Sigma'}$)}$
is smooth.
Moreover the smoothness of the morphism $\{s\}\times \Sigma'\to \overline{\mbfS}$ implies that
the intersection of  $\barR_{\{s\}\times\Sigma'}$
and $\left(
\gamma'^{*}\mcD'
\right)_{\mathrm{red}}$
is a SNC divisor on $\barR_{\{s\}\times\Sigma'}$,
which shows that the divisor $\left(
\gamma'^{*}\mcD'
\right)_{\mathrm{red}}
+\gamma'^{*}\mcS_{\Sigma'}$
is SNC.
\endproof
By shrinking $\Sigma'$ if necessary, we may assume that the divisor
$\left(
\gamma'^{*}\mcD'
\right)_{\mathrm{red}}
+\gamma'^{*}\mcS_{\Sigma'}$
is relatively SNC over $\Sigma'$.
We define the sheaf $\mcN_{\tilde{f}}$ on $\mcY$ by the exact sequence
\begin{equation}\label{eq-Nmoto}
0
\to 
T_{\mcY}(-\log(\mcD_{\mcY}\cup \mcS_{\mcY}))
\to
\tilde{f}^{*}
T_{\overline{R}_{C\times \Sigma'}}
(-\log((\gamma'^{*}\mcD')_{\mathrm{red}}+\gamma'^{*}\mcS_{\Sigma'}))
\to
\mcN_{\tilde{f}} \to 0.
\end{equation}
By Proposition \ref{prop-reltan} (3),
we have an exact sequence
\begin{equation}\label{eq-Nftil}
0
\to 
T_{\mcY/\Sigma'}(-\log(\mcD_{\mcY}\cup \mcS_{\mcY}))
\to
\tilde{f}^{*}
T_{\overline{R}_{C\times \Sigma'}/\Sigma'}
(-\log((\gamma'^{*}\mcD')_{\mathrm{red}}+\gamma'^{*}\mcS_{\Sigma'}))
\to
\mcN_{\tilde{f}} \to 0.
\end{equation}
For a general point $\sigma'\in \Sigma'$,
restricting the above short exact sequence to the fiber $Y_{\sigma'}$,
we obtain an exact sequence
\begin{equation}
0
\to 
T_{Y_{\sigma'}}(-\log(D_{Y_{\sigma'}}\cup \mcS_{Y_{\sigma'}}))
\to
\left.
\tilde{f}^{*}
T_{\overline{R}_{C\times \Sigma'}/\Sigma'}
(-\log((\gamma'^{*}\mcD')_{\mathrm{red}}+\gamma'^{*}\mcS_{\Sigma'}))
\right|_{Y_{\sigma'}}
\to
\left.
\mcN_{\tilde{f}} 
\right|_{Y_{\sigma'}}
\to 0.
\end{equation}
From this, we obtain an isomorphism
\begin{align}\label{align-KfN}
K_{Y_{\sigma'}}\left(
D_{Y_{\sigma'}}\cup \mcS_{Y_{\sigma'}}
\right)
& \simeq
\left.
\tilde{f}^{*}
K_{\overline{R}_{C\times \Sigma'}/\Sigma'}
((\gamma'^{*}\mcD')_{\mathrm{red}}+\gamma'^{*}\mcS_{\Sigma'})
\right|_{Y_{\sigma'}}
\otimes
\det \left.
\mcN_{\tilde{f}} 
\right|_{Y_{\sigma'}} \\
& \simeq
\left.
\tilde{f}^{*}
K_{\overline{R}_{C\times \Sigma'}}
((\gamma'^{*}\mcD')_{\mathrm{red}}+\gamma'^{*}\mcS_{\Sigma'})
\right|_{Y_{\sigma'}}
\otimes
\det \left.
\mcN_{\tilde{f}} 
\right|_{Y_{\sigma'}}, \notag
\end{align}
where the second isomorphism follows because
the determinant line bundle of 
$T_{\overline{R}_{C\times \Sigma'}}
(-\log((\gamma'^{*}\mcD')_{\mathrm{red}}+\gamma'^{*}\mcS_{\Sigma'}))$
is isomorphic to that of 
$T_{\overline{R}_{C\times \Sigma'}/\Sigma'}
(-\log((\gamma'^{*}\mcD')_{\mathrm{red}}+\gamma'^{*}\mcS_{\Sigma'}))$
tensored by the pull-back of a line bundle on $\Sigma'$.
%
\subsection{Lower bound of $\deg \left.\mcN_{\tilde{f}}\right|_{Y_{\sigma'}} $}
In this subsection, we give a lower bound of  $\deg \left.\mcN_{\tilde{f}}\right|_{Y_{\sigma'}}$.

Note that the morphism
\[
\bar{\eta}_{3}\circ b_{1}\circ b_{2} : \overline{R}_{C\times \Sigma'}\to 
\widetilde{\mcP}^{\dagger}/\mathfrak{S}_{1}
\]
is smooth,
and that we have
\begin{align*}
(\gamma'^{*}\mcD')_{\mathrm{red}}
&=I_{C\times\Sigma'}
+(b_{1}\circ b_{2})^{*}\left(
\sum_{i=1}^{c}
J_{\overline{R}}^{(i)}
\right)_{\mathrm{red}} \\
&=I_{C\times\Sigma'}
+(\bar{\eta}_{3}\circ b_{1}\circ b_{2})^{*}
\left(
\left(
\psi^{*}
\sum_{i=1}^{c}\tau^{(i)*}_{1}\overline{J}^{(i)}
\right)_{\mathrm{red}}
\right).
\end{align*}
We have an exact sequence
\begin{align}\label{align-d1}
0
\to 
T_{\overline{R}_{C\times\Sigma'}/(\widetilde{\mcP}^{\dagger}/\mathfrak{S}_{1})}
\left(
-\log\left(
I_{C\times\Sigma'}+
\gamma'^{*}\mcS_{\Sigma'}
\right)
\right)
\to
T_{\overline{R}_{C\times\Sigma'}}
\left(
-\log\left(
\left(
\gamma'^{*}\mcD'
\right)_{\mathrm{red}}+
\gamma'^{*}\mcS_{\Sigma'}
\right)
\right)
\\
\xrightarrow{d_{1}}
(\bar{\eta}_{3}\circ b_{1} \circ b_{2})^{*}
T_{ (\widetilde{\mcP}^{\dagger}/\mathfrak{S}_{1}) }
\left(
-\log\left(\left(
\psi^{*}
\sum_{i=1}^{c}
\tau_{1}^{(i)*}\overline{J}^{(i)}
\right)_{\mathrm{red}}
\right)
\right)
\to
0 \notag
\end{align}
of $\mcO_{\overline{R}_{C\times\Sigma'}}$-modules.
The morphism
$\bar{\eta}_{3}\circ b_{1}\circ b_{2}\circ \tilde{f}
:\mcY\to \widetilde{\mcP}^{\dagger}/\mathfrak{S}_{1}$
is not necessarily smooth,
but since $GL(V)$ acts transitively on $\mbO$,
the restriction of $\bar{\eta}_{3}\circ b_{1}\circ b_{2}\circ \tilde{f}$
to the open subset $(\psi \circ \bar{\eta}_{3}\circ b_{1}\circ b_{2}\circ \tilde{f})^{-1}(\mbO)$
is smooth
(see \S \ref{subsect-const} for the varieties $\mbO$ and $\overline{\mbO}$,
and recall that the morphism $\psi^{-1}(\mbO)\to \mbO$ is an isomorphism).
So if we let $d_{2}$ be the composite of morphisms
\begin{align}
T_{\mcY}\left(
-\log\left(
\mcD_{\mcY}\cup\mcS_{\mcY}
\right)
\right)
\to
\tilde{f}^{*}
T_{\overline{R}_{C\times\Sigma'}}
\left(
-\log\left(
(\gamma'^{*}\mcD)_{\mathrm{red}}
+\gamma'^{*}\mcS_{\Sigma'}
\right)
\right)
\\
\xrightarrow{\tilde{f}^{*}(d_{1})}
(\bar{\eta}_{3}\circ b_{1} \circ b_{2}\circ \tilde{f})^{*}
T_{ \widetilde{\mcP}^{\dagger}/\mathfrak{S}_{1} }
\left(
-\log\left(\left(
\psi^{*}
\sum_{i=1}^{c}
\tau_{1}^{(i)*}\overline{J}^{(i)}
\right)_{\mathrm{red}}
\right)
\right),
\notag
\end{align}
then $d_{2}$ is surjective over the open subset $(\psi \circ \bar{\eta}_{3}\circ b_{1}\circ b_{2}\circ \tilde{f})^{-1}(\mbO)$.
Combining (\ref{eq-Nmoto}) and (\ref{align-d1}),
we obtain an exact sequence
\begin{equation}
0
\to
\Ker\,d_{2}
\to
\tilde{f}^{*}
T_{\overline{R}_{C\times\Sigma'}/(\widetilde{\mcP}^{\dagger}/\mathfrak{S}_{1})}
\left(
-\log\left(
I_{C\times\Sigma'}+
\gamma'^{*}\mcS_{\Sigma'}
\right)
\right)
\xrightarrow{d_{3}}
\mcN_{\tilde{f}},
\end{equation}
and the cokernel of $d_{3}$ has 
support in $\mcY\setminus
(\psi \circ \bar{\eta}_{3}\circ b_{1}\circ b_{2}\circ \tilde{f})^{-1}(\mbO)$.
Put $\mcN_{\tilde{f}}^{\flat}:=\mathrm{Im}\,d_{3}$.
We obtain a short exact sequence
\begin{equation}\label{eq-Kd2N}
0
\to
\Ker\,d_{2}
\to
\tilde{f}^{*}
T_{\overline{R}_{C\times\Sigma'}/(\widetilde{\mcP}^{\dagger}/\mathfrak{S}_{1})}
\left(
-\log\left(
I_{C\times\Sigma'}+
\gamma'^{*}\mcS_{\Sigma'}
\right)
\right)
\xrightarrow{d_{3}}
\mcN_{\tilde{f}}^{\flat}
\to
0,
\end{equation}
and an inequality
\begin{equation}\label{eq-NflatN}
\deg \left.\mcN_{\tilde{f}}^{\flat}\right|_{Y_{\sigma'}}
\leq
\deg \left.\mcN_{\tilde{f}}\right|_{Y_{\sigma'}}
\end{equation}
for general $\sigma'\in \Sigma'$.
We aim at getting a lower bound of $\deg \left.\mcN_{\tilde{f}}^{\flat}\right|_{Y_{\sigma'}}$.

From the short exact sequence (cf. Proposition \ref{prop-reltan} (2))
\begin{align}
0
\to
\tilde{f}^{*}
T_{\overline{R}_{C\times\Sigma'}/\overline{R}}
\left(
-\log
\gamma'^{*}\mcS_{\Sigma'}
\right)
\to
\tilde{f}^{*}
T_{\overline{R}_{C\times\Sigma'}/\left( \widetilde{\mcP}^{\dagger}/\mathfrak{S}_{1} \right) }
\left(
-\log\left(I_{C\times\Sigma'}+
\gamma'^{*}\mcS_{\Sigma'}\right)
\right) 
\\
\to
(b_{1}\circ b_{2} \circ \tilde{f})^{*}
T_{\overline{R}/\left( \widetilde{\mcP}^{\dagger}/\mathfrak{S}_{1} \right) }
(-\log I_{\overline{R}}) 
\to 
0
\notag
\end{align}
and the surjective morphism $d_{3}$ in (\ref{eq-Kd2N}),
we obtain an induced exact sequence
\begin{equation}
0\to
\tilde{\mcN}_{2}
\to
\mcN_{\tilde{f}}^{\flat}
\to
\tilde{\mcN}_{1}
\to 
0,
\end{equation}
together with surjective morphisms
\begin{equation}\label{eq-surjN2}
\tilde{f}^{*}
T_{\overline{R}_{C\times\Sigma'}/\overline{R}}
\left(
-\log
\gamma'^{*}\mcS_{\Sigma'}
\right)
\twoheadrightarrow\tilde{\mcN}_{2}
\end{equation}
and
\begin{equation}\label{eq-surjN1}
(b_{1}\circ b_{2} \circ \tilde{f})^{*}
T_{\overline{R}/\left( \widetilde{\mcP}^{\dagger}/\mathfrak{S}_{1} \right) }
(-\log I_{\overline{R}}) 
\twoheadrightarrow\tilde{\mcN}_{1}
\end{equation}
(see \S \ref{subsection-term} for what we mean by an induced exact sequence).

The relative log tangent bundle 
$T_{\barR_{C\times\Sigma'}/\barR}
\left(
-\log\gamma'^{*}\mcS_{\Sigma'}
\right)$
on $\barR_{C\times\Sigma'}$
is isomorphic to the pull-back 
to $\barR_{C\times\Sigma'}$
of the relative log tangent bundle
$T_{C\times\Sigma'/\overline{\mbfS}}\left(
-\log(S\times \Sigma')
\right)$
on $C\times\Sigma'$.
On $C\times\Sigma'$,
we have a short exact sequence (cf. Proposition \ref{prop-reltan} (2))
\begin{equation}
0\to T_{C\times\Sigma'/C\times\Sigma^{sm}}
\to
T_{C\times\Sigma'/\overline{\mbfS}}(-\log(S\times\Sigma'))
\to
(\id_{C}\times\beta)^{*}
T_{C\times\Sigma^{sm}/\overline{\mbfS}}
(-\log(S\times\Sigma^{sm}))
\to 0
\end{equation}
of tangent bundles.
From the pull-back to $\mcY$ of this short exact sequence and the surjection (\ref{eq-surjN2}),
we obtain an induced exact sequence
\begin{equation}
0\to \tilmcN_{4}
\to \tilmcN_{2} \to \tilmcN_{3} \to 0
\end{equation}
of $\mcO_{\mcY}$-modules,
together with surjective morphisms
\begin{equation}\label{eq-surjN3}
\left(
(\id_{C}\times\beta)
\circ p_{1}
\right)^{*}
T_{C\times\Sigma^{sm}/\overline{\mbfS}}\left(
-\log(S\times\Sigma^{sm})
\right)
\twoheadrightarrow\tilmcN_{3}
\end{equation}
and
\begin{equation}
\tilde{p}_{1}^{*}
T_{C\times\Sigma'/C\times\Sigma^{sm}}
\twoheadrightarrow\tilmcN_{4}.
\end{equation}
Since $\tilde{p}_{1}^{*}
T_{C\times\Sigma'/C\times\Sigma^{sm}}
\simeq p_{1}^{*}T_{\Sigma'/\Sigma^{sm}}$,
the restriction of $\tilde{p}_{1}^{*}
T_{C\times\Sigma'/C\times\Sigma^{sm}}$
to the fiber $Y_{\sigma'}$ ($\sigma'\in \Sigma'$)
is a trivial bundle.
So we have
\begin{equation}\label{eq-geqN4}
\deg \left.
\tilmcN_{4}
\right|_{Y_{\sigma'}}\geq 0.
\end{equation}
Since $\alpha_{1}$ is the composed morphism (\ref{eq-alpha1}),
we have a short exact sequence
\begin{equation}\label{eq-TCSsm}
0\to T_{C\times\Sigma^{sm}/C\times\overline{\mbfS}}
\to
T_{C\times\Sigma^{sm}/\overline{\mbfS}}
\left(
-\log(S\times\Sigma^{sm})
\right)
\to
\alpha^{*}
T_{C\times\underline{\mbfS}/\underline{\mbfS}}
\left(
-\log(S\times\overline{\mbfS})
\right)
\to 0.
\end{equation}
From the surjection (\ref{eq-surjN3}) and the pull-back to $\mcY$ of the short exact sequence (\ref{eq-TCSsm}),
we obtain an induced short exact sequence
\begin{equation}
0\to \tilmcN_{6}
\to \tilmcN_{3}
\to \tilmcN_{5}
\to 0
\end{equation}
of $\mcO_{\mcY}$-modules,
and surjective morphisms
\begin{equation}\label{eq-surjN6}
\left(
(\id_{C}\times\beta)\circ \tilde{p}_{1}
\right)^{*}
T_{C\times\Sigma^{sm}/C\times\overline{\mbfS}}
\twoheadrightarrow \tilmcN_{6}
\end{equation}
and
\begin{equation}
\left(
(\id_{C}\times\beta)\circ \tilde{p}_{1}
\right)^{*}\alpha^{*}
T_{C\times\underline{\mbfS}/\underline{\mbfS}}
\left(
-\log(S\times\overline{\mbfS})
\right)
\twoheadrightarrow \tilmcN_{5}.
\end{equation}
By (\ref{eq-TCSigsm}), for any $\sigma\in \Sigma^{sm}$, the vector bundle
$\left.T_{C\times\Sigma^{sm}/C\times\overline{\mbfS}}\right|_{C\times\{\sigma\}}$
on $C\times\{\sigma\}(\simeq C)$
is isomorphic to $\mcM_{C}^{(\mcL; W)}\otimes S^{d}V$,
so
$\left.T_{C\times\Sigma^{sm}/C\times\overline{\mbfS}}\right|_{C\times\{\sigma\}}
\otimes \mcL$
is generated by global sections by Proposition \ref{prop-posi}.
Thus, restricting the surjective morphism (\ref{eq-surjN6}) to the fiber $Y_{\sigma'}$ for general $\sigma'\in \Sigma'$,
we obtain
\begin{equation}\label{eq-geqN6}
\deg\left.
\tilmcN_{6}\right|_{Y_{\sigma'}}
\geq
-(\rank \tilmcN_{6})le.
\end{equation} 
Since the vector bundle $T_{ C\times\overline{\mbfS}/\barS }\left(-\log(S\times \overline{\mbfS})\right)$
is isomorphic to the pull-back to $C\times\overline{\mbfS}$
of the vector bundle $T_{C}(-S)$ on $C$, 
\begin{equation}\label{eq-geqN5}
\deg \left. \tilmcN_{5}\right|_{Y_{\sigma'}}
\geq
\begin{cases}
0 & \text{if $\rank\, \tilmcN_{5}=0$} \\
e(2-2g(C)-|S|) & \text{if $\rank\, \tilmcN_{5}=1.$} 
\end{cases}
\end{equation} 
From (\ref{eq-geqN6}) and (\ref{eq-geqN5}),
we have
\begin{equation}\label{eq-geqN3}
\deg \left. \tilmcN_{3}\right|_{Y_{\sigma'}}
\geq
e\cdot\min\left\{
-(\rank\,\tilmcN_{3})l,
-(\rank\,\tilmcN_{3}-1)l+2-2g(C)-|S|
\right\}.
\end{equation}
From (\ref{eq-geqN4}) and (\ref{eq-geqN3}),
we have
\begin{equation}\label{eq-geqN2}
\deg \left. \tilmcN_{2}\right|_{Y_{\sigma'}}
\geq
e\cdot\min\left\{
-(\rank\,\tilmcN_{2})l,
-(\rank\,\tilmcN_{2}-1)l+2-2g(C)-|S|
\right\}.
\end{equation}
By Proposition \ref{prop-TPU},
we have an isomorphism
\begin{equation}
T_{\overline{R}/\left(
\wtPd/\mathfrak{S}_{1}
\right)}
(-\log I_{\overline{R}})
\simeq
\bar{\eta}_{3}^{*}\psi^{*}(\pi_{1}\circ \tau)^{*}\mcM_{\mbG}^{d}
\otimes\gamma^{*}pr_{\overline{\mbfS}}^{*}\mcO_{\overline{\mbfS}}(1)
\end{equation}
of $\mcO_{\overline{R}}$-module.
From this and the surjective morphism (\ref{eq-surjN1}),
we see that
there is a surjective morphism of $\mcO_{\mcY}$-modules
\begin{equation}\label{eq-surjN1-1}
(b_{1}\circ b_{2} \circ \tilde{f})^{*}
\bar{\eta}_{3}^{*}
\psi^{*}(\pi_{1}\circ \tau)^{*}
\mcM_{\mbG}^{d}
\twoheadrightarrow
\tilmcN_{1}
\otimes
\tilde{p}_{1}^{*}
(\id_{C}\times\beta)^{*}
\alpha_{1}^{*}\mcO_{\overline{\mbfS}}(-1).
\end{equation}
Using the isomorphism (cf. (\ref{eq-alOS}))
\begin{equation}\label{eq-alpull}
\alpha_{1}^{*}\mcO_{\barS}(1)\simeq \mcL \boxtimes \mcO_{\Sigma^{sm}}(1)
\end{equation}
of $\mcO_{C\times\Sigma^{sm}}$-modules,
and the positivity of $\mcM_{\mbG}^{d}$ (Proposition \ref{prop-Mposi}),
we have an inequality
\begin{equation}\label{eq-geqN1}
\deg \left.
\tilmcN_{1}
\right|_{Y_{\sigma'}}
\geq \rank\,\tilmcN_{1}
\left(
el-\deg \left. \mcO_{\mbG}(1)\right|_{Y_{\sigma'}}
\right),
\end{equation}
where $\left.\mcO_{\mbG}(1)\right|_{Y_{\sigma'}}$
denotes the pull-back to $Y_{\sigma'}$ of $\mcO_{\mbG}(1)$,
that is, it is a shorthand for
$\left.
(\pi_{1}\circ\tau\circ\psi\circ\bar{\eta}_{3}\circ b_{1}\circ b_{2} \circ \tilf)^{*}
\mcO_{\mbG}(1)
\right|_{Y_{\sigma'}}.$
\begin{notation}\label{notation-Y}
As in the last sentence, we write $\mcE|_{Y_{\sigma'}}$
for the pull-back to $Y_{\sigma'}$ of a sheaf $\mcE$.
\end{notation}
From (\ref{eq-geqN2}) and (\ref{eq-geqN1}),
using the equality
$\rank\,\tilmcN_{1}+\rank\,\tilmcN_{2}=\rank\,\mcN_{\tilf}=2n-d+1$,
we have 
\begin{align}\label{align-N1N2}
\deg\left.
\mcN_{\tilf}^{\flat}
\right|_{Y_{\sigma'}}
& 
\geq 
\rank\,\tilmcN_{1}
\left(
el-\deg \left. \mcO_{\mbG}(1)\right|_{Y_{\sigma'}}
\right)
-
(\rank\,\tilmcN_{2})el
+e\cdot
\min\{
0, l+2-2g(C)-|S|
\}
\\
& 
\geq (2n-d+1)\cdot \min\left\{
el-\deg \left. \mcO_{\mbG}(1)\right|_{Y_{\sigma'}},
-el
\right\}
+e\cdot
\min\{
0, l+2-2g(C)-|S|
\} \notag
\\
&
\geq
(2n-d+1)\left(
-el-\deg \left. \mcO_{\mbG}(1)\right|_{Y_{\sigma'}}
\right)
+e\cdot
\min\{
0, l+2-2g(C)-|S|
\}. \notag
\end{align}
%
%
\subsection{Deriving the inequality (\ref{eq-ineqcore})}
%
%
In this subsection, we obtain a lower bound of
$\deg \left.
K_{\barR_{C\times\Sigma'}}
\left(
\left(
\gamma'^{*}\mcD'
\right)_{\mathrm{red}}
+\gamma'^{*}\mcS_{\Sigma'}
\right)
\right|_{Y_{\sigma'}}$.
Combining it with the lower bound of $\deg \left.\mcN_{\tilde{f}}\right|_{Y_{\sigma'}}$
obtained in the preceding subsection,
we derive the inequality (\ref{eq-ineqcore}).

We have 
\begin{equation}\label{eq-KKeS}
\deg \left.
K_{\barR_{C\times\Sigma'}}
\left(
\left(
\gamma'^{*}\mcD'
\right)_{\mathrm{red}}
+\gamma'^{*}\mcS_{\Sigma'}
\right)
\right|_{Y_{\sigma'}}
=
\deg \left.
K_{\barR_{C\times\Sigma'}}
\left(
\left(
\gamma'^{*}\mcD'
\right)_{\mathrm{red}}
\right)
\right|_{Y_{\sigma'}}
+e|S|.
\end{equation}
Since $(\gamma'^{*}\mcD')_{\mathrm{red}}=(b_{1}\circ b_{2})^{*}(\gamma^{*}\mathbf{D})_{\mathrm{red}}$,
we have an isomorphism of sheaves on $\barR_{C\times\Sigma'}$
\begin{equation}\label{eq-Kbb}
K_{\barR_{C\times\Sigma'}}
\left(
\left(
\gamma'^{*}\mcD'
\right)_{\mathrm{red}}
\right)
\simeq (b_{1}\circ b_{2})^{*}
K_{\barR}\left(
(\gamma^{*}\mbfD)_{\mathrm{red}}
\right)
\otimes
(pr_{C\times\Sigma'}\circ\gamma')^{*}K_{C\times\Sigma'/\barS},
\end{equation}
where $pr_{C\times\Sigma'}:\Pn\times C\times\Sigma'\to C\times\Sigma'$ is the projection.
We have an isomorphism
\begin{equation}\label{eq-KKK}
K_{C\times\Sigma'/\barS}\simeq
K_{C\times\Sigma'/C\times\Sigma^{sm}}
\otimes
(\id_{C}\times\beta)^{*}K_{C\times\Sigma^{sm}/\barS}.
\end{equation}
Noting that the morphism $\alpha_{1}$ factors as in (\ref{eq-alpha1}),
using (\ref{eq-TCSigsm}), we find that
\begin{equation}\label{eq-KLO}
K_{C\times\Sigma^{sm}/\barS}
\simeq
\left(
\mcL^{\otimes\dim S^{d}(V)}
\otimes K_{C}
\right)
\boxtimes
\mcO_{\Sigma^{sm}}(-\dim S^{d}(V)\cdot(\dim W-1)).
\end{equation}
From (\ref{eq-KKK}) and (\ref{eq-KLO}), we have
\begin{equation}\label{eq-degKCSS}
\deg \left.
K_{C\times\Sigma'/\barS}
\right|_{Y_{\sigma'}}
=e\left(
\deg K_{C}+l\cdot\dim S^{d}(V)
\right).
\end{equation}
Note that the isomorphism (\ref{eq-alpull}) implies that
\[
\deg \mcO_{\barS}(1)|_{Y_{\sigma'}} =el.
\]
Now we use Corollary \ref{cor-ggen} to obtain
\begin{align}\label{align-degKR1}
\deg \left.
K_{\barR}
\left(
(\gamma^{*}\mbfD)_{\mathrm{red}}
\right)
\right|_{Y_{\sigma'}}
 \geq &
\deg \left.
\mcO_{\mbG}\left(
\frac{(d-1)(d-2)}{2}-n
\right)\right|_{Y_{\sigma'}} 
\\
&
+\deg
\left. \mcO_{\Pn}\left(
2\sum_{i=1}^{c}e^{(i)}-4
\right)\right|_{Y_{\sigma'}}
-el(\rank\,\mcU+1). \notag
\end{align}
From (\ref{eq-KKeS}), (\ref{eq-Kbb}), (\ref{eq-degKCSS}) and
(\ref{align-degKR1}),
we see that
we have
\begin{align}\label{align-Kgas1}
\deg
\left.
K_{\barR_{C\times\Sigma'}}
\left(
\gamma'^{*}
\left(
\mcD'
\right)_{\mathrm{red}}
+\gamma'^{*}\mcS_{\Sigma'}
\right)\right|_{Y_{\sigma'}}
\geq 
e\left(
|S|+\deg K_{C}+l\left(
\dim S^{d}(V)-\rank\,\mcU-1
\right)
\right) \\
+\left(
\frac{(d-1)(d-2)}{2}-n
\right)\deg \left. \mcO_{\mbG}(1)\right|_{Y_{\sigma'}}
+\left(
2\sum_{i=1}^{c}e^{(i)}-4
\right)\deg \left. \mcO_{\Pn}(1)\right|_{Y_{\sigma'}}. \notag
\end{align}
From (\ref{align-KfN}), (\ref{eq-NflatN}), (\ref{align-N1N2}) and (\ref{align-Kgas1}),
using $\rank\, \mcU=\dim S^{d}(V)-d$,
we have
\begin{align}
& \deg K_{Y_{\sigma'}}\left(
D_{Y_{\sigma'}}\cup\mcS_{Y_{\sigma'}}
\right) \\
& \geq
\left(
\frac{d(d-1)}{2}-3n
\right)\deg \left. \mcO_{\mbG}(1)\right|_{Y_{\sigma'}}
+\left(
2\sum_{i=1}^{c}e^{(i)}-4
\right)\deg \left. \mcO_{\Pn}(1)\right|_{Y_{\sigma'}} \notag \\
&+e\cdot \min\left\{
0,l+2-2g(C)-|S|
\right\}
+e\left(
|S|+\deg K_{C}+l(2d-2n-2)
\right). \notag
\end{align}
Under our assumption $d>\frac{3n+3}{2}$,
the inequality $\frac{d(d-1)}{2}-3n\geq 0$ holds.
By our assumption $\sharp \Supp E_{0} \geq 3$,
we have
\begin{align}
\deg K_{Y_{\sigma'}}\left(
D_{Y_{\sigma'}}\cup\mcS_{Y_{\sigma'}}
\right) 
\geq &
2\deg \left. \mcO_{\Pn}(1)\right|_{Y_{\sigma'}}
+e\cdot \min\left\{
0,l+2-2g(C)-|S|
\right\} \\
&+e\left(
|S|+\deg K_{C}+2l(d-n-1)
\right). \notag 
\end{align}
Subtracting $e\cdot |S|$ from the both sides, and noting
\begin{equation*}
2g(Y_{\sigma'})-2+\left|
D_{Y_{\sigma'}}\setminus S_{Y_{\sigma'}}
\right|
\geq
\deg 
K_{Y_{\sigma'}}\left(
D_{Y_{\sigma'}}\cup \mcS_{Y_{\sigma'}}
\right)-e\cdot |S|,
\end{equation*}
we have
\begin{align*}
2g(Y_{\sigma'})-2+\left|
D_{Y_{\sigma'}}\setminus S_{Y_{\sigma'}}
\right|
\geq 
2\deg \left. \mcO_{\Pn}(1)\right|_{Y_{\sigma'}}
+e(2g(C)-2) 
+2el(d-n-1)
\\
+e\cdot\min\left\{
0, l+2-2g(C)-|S|
\right\}. \notag
\end{align*}
We can rewrite this as
\begin{align}
\frac{2}{e} \deg \left. \mcO_{\Pn}(1)\right|_{Y_{\sigma'}}
\leq
\frac{1}{e}\left(
2g(Y_{\sigma'})-2+\left|
D_{Y_{\sigma'}}\setminus S_{Y_{\sigma'}}
\right|
\right)-(2g(C)-2)
-2l(d-n-1) \\
+\max\left\{
0, 2g(C)-2+|S|-l
\right\}.\notag
\end{align}
The inequality (\ref{eq-ineqcore}) follows from this.
\hfill \qed
\section{Proof of Theorem \ref{thm-main} and Corollary \ref{cor-ofmain}}\label{section-proofofmain}
%
\subsection{Proof of Theorem \ref{thm-main}}
We show that Theorem \ref{thm-main} follows from Theorem \ref{thm-core}
by a standard argument using Hilbert schemes.

Let $\Hilb (X\times \Sigma^{sm}/\Sigma^{sm})$ be the Hilbert scheme
parametrizing integral curves $Y'$ on fibers $p^{-1}(\sigma)=X\times\{\sigma\}\simeq X$,
$\sigma\in \Sigma^{sm}$,
such that $Y'\not\subset D_{\sigma}$ and the composition
$Y'\hookrightarrow X=\Pn\times C\to C$ is a finite morphism.
Let $\pi:\Hilb (X\times \Sigma^{sm}/\Sigma^{sm})\to\Sigma^{sm}$
be the natural morphism.
In this proof, every subscheme of $\Hilb (X\times \Sigma^{sm}/\Sigma^{sm})$
is given its reduced scheme structure.
\begin{claim}\label{claim-hilb}
$\Hilb (X\times \Sigma^{sm}/\Sigma^{sm})$ can be expressed, set-theoretically,
as a countable disjoint union of $GL(V)$-invariant locally closed subset
\[
\Hilb (X\times \Sigma^{sm}/\Sigma^{sm})=\bigsqcup_{k=1}^{\infty}H_{k}
\]
such that for each $k$,
either (a) or (b) holds:
\begin{itemize}
 \item[(a)] $\pi|_{H_{k}}:H_{k}\to \Sigma^{sm}$ is not dominant.
 \item[(b)] $\pi|_{H_{k}}:H_{k}\to \Sigma^{sm}$ is a smooth morphism,
and for any $[Y\subset X\times \{\sigma\}] \in H_{k}$,
either the inequality in Theorem \ref{thm-main} (i) holds,
or $Y\subset \mathbf{E}_{\sigma}$.
\end{itemize}
\end{claim}
\proof[Proof of Claim \ref{claim-hilb}]
First note that $\Hilb (X\times \Sigma^{sm}/\Sigma^{sm})$
is a countable disjoint union of connected schemes of finite type over $\mathrm{Spec}\,\mbC$.
Let $H$ be a connected component of $\Hilb (X\times \Sigma^{sm}/\Sigma^{sm})$,
and we will show that $H$ can be expressed as a finite disjoint union
$H=\bigsqcup_{k}H_{k}$ of locally closed subsets with each $H_{k}$ satisfying (a) or (b) above.

If $\pi|_{H}:H\to\Sigma^{sm}$ is not dominant, then $H$ itself satisfies (a)
and we are done.
Suppose $\pi|_{H}:H\to\Sigma^{sm}$ is dominant.
Then we can find an irreducible $\GLV$-invariant open subset
$H_{1}\subset H$ of maximal dimension such that
$\pi|_{H_{1}}:H_{1}\to\Sigma^{sm}$ is smooth.
Let $\mcY'\subset X\times H_{1}$ be the universal family of curves,
and let $\mcY\to \mcY'$ be the normalization.
By shrinking $H_{1}$ if necessary,
we may assume that $\mcY$ is smooth over $H_{1}$.
Moreover by shrinking $H_{1}$ further if necessary,
we may assume that $\mcY\to H_{1}$
satisfies the conditions imposed on $\mcY\xrightarrow{p_{1}}\Sigma'$ in \S \ref{section-setting}
(such as $\mcD_{\mcY}$ is \'{e}tale over $\Sigma'$, and so on).
The applying Theorem \ref{thm-core}, by shrinking $H_{1}$ if necessary,
$H_{1}$ satisfies (b).
Let $H':=H\setminus H_{1}$,
and repeat the same procedure for $H'$.
Continuing this process, we we finally obtain the desired decomposition $H=\bigsqcup_{k}H_{k}$.
\endproof
Now we show that Claim \ref{claim-hilb} implies Theorem \ref{thm-main}.
Let $I\subset \mbN$ be the set of $k$ such that $\pi|_{H_{k}}:H_{k}\to \Sigma^{sm}$ is not dominant.
Then for every $\sigma \in \Sigma^{sm}\setminus \bigcup_{k\in I}\overline{\pi(H_{k})}$,
the conclusion of Theorem \ref{thm-main} holds.
\hfill \qed
%
\subsection{Proof of Corollary \ref{cor-ofmain}}
(1)
Let $Y\subset \Pn\times C$ be a  non-vertical irreducible curve not contained in $\mathbf{E}_{\sigma}\cup D_{\sigma}$,
let $\nu:\widetilde{Y}\to Y$ be the normalization,
and let $e$ be the degree of the finite morphism $Y\to C$.
By Theorem \ref{thm-main},
the inequality 
\[
N^{(1)}_{S}(D_{\sigma},Y)
+d_{C}(Y)
\geq
\frac{1}{e}
\deg(pr_{\Pn}\circ \nu)^{*}\mcO_{\Pn}(1)-c_{0}
\]
holds.
If the inequality (\ref{eq-corofmain}) does not hold,
then we have
\[
\frac{1}{e}
\deg(pr_{\Pn}\circ \nu)^{*}\mcO_{\Pn}(1)-c_{0}
<\frac{1}{d-n-1}h_{\mcK(D_{\sigma})}(Y)-\epsilon h_{\mcA}(Y)-c,
\]
which is equivalent to the inequality
\[
\deg\nu^{*}(\mcA|_{Y})
<
\frac{e}{\epsilon}
\left\{
\frac{1}{d-n-1}\left(
2g(C)-2+l
\right)+c_{0}-c
\right\}.
\]
It follows that if $e\leq r$, then $\deg\nu^{*}(\mcA|_{Y})$ is bounded above.

\noindent (2) If a section $\tau:C\to \Pn\times C$
satisfies $\tau^{-1}(D_{\sigma})\subset S$ and $\tau(C)\not\subset \mathbf{E}_{\sigma}$,
then by Theorem \ref{thm-main} we have
\[
\deg(pr_{\Pn}\circ \tau)\mcO_{\Pn}(1) \leq c_{0}.
\]
So if $\mcP$ is an ample line bundle on $C$,
the degree of $\tau(C)$ with respect to the ample line bundle $\mcO_{\Pn}(1)\boxtimes \mcP$ is
less than or equal to $c_{0}+\deg \mcP$.
\hfill \qed

\section{Some Remarks}\label{section-Remarks}
%
\subsection{Exceptional sets}\label{subsection-exceptional}
The following proposition gives some  reason why vertical lines intersecting the divisor $D_{\sigma}$
in at most two points
appear in the definition of the exceptional set $\mathbf{E}_{\sigma}$ in Theorem \ref{thm-main}.
\begin{prop}
Let $C$ be a smooth projective curve,
and let $\mcE$ be a rank $2$ vector bundle on $C$.
Consider the ruled surface $\pi:\mbP(\mcE) \to C$,
and let $D$ be a reduced divisor on $\mbP(\mcE)$
intersecting a general fiber $\pi^{-1}(t)$, $t\in C$, in two points.
Then for any $N>0$,
there exists a curve $C'\subset \mbP(\mcE)$ with normalization $\nu: \widetilde{C'}   \to C'$
such that the inequality
\[
\deg \nu^{*}\mcO_{\mbP(\mcE)}(1)>2g(\widetilde{C'} )-2 +\left|
\nu^{-1}(D)
\right|+N
\]
holds.
\end{prop}
\proof
\noindent \textbf{Case 1}.
We consider the case that $D=D_{1}+D_{2} + D_{3}$,
where $D_{1}, D_{2}\subset \mbP(\mcE)$ are irreducible curves
mapping isomorphically to $C$ by $\pi$,
and $D_{3}=\sum_{t\in T}\pi^{-1}(t)$
for a finite set $T\subset C$.
In this case, we show that we can find a section $\sigma:C\to \mbP(\mcE)$ of
$\pi$ such that
$\sigma^{*}\mcO_{\mbP(\mcE)}(1)-\left|\sigma^{-1}(D) \right|$ is arbitrarily large.
We consider two subcases.

\textbf{Subcase 1a}.
We consider the case that $\mcE=\mcO_{C}^{2}$, thus $\mbP(\mcO_{C}^{2})=\mbP^{1}\times C$,
and that  $D_{1}=\{0\}\times C$, $D_{2}=\{\infty\}\times C$.

Take a globally generated ample line bundle $\mcA$ on $C$,
and choose two global sections $s_{1}, s_{2}\in \mrH^{0}(C, \mcA)$
such that the map $\mcO_{C}^{2}\to \mcA$
given by $(f_{1},f_{2})\mapsto f_{1}s_{1}+f_{2}s_{2}$
is surjective.
For $l\geq 1$, let $\lambda_{l}:\mcO_{C}^{2}\to \mcA^{l}$
be the morphism given by $(f_{1},f_{2})\mapsto f_{1}s_{1}^{l}+f_{2}s_{2}^{l}$,
and let $\sigma_{l}:C\to \mbP(\mcO_{C}^{2})$ be the section of $\pi:\mbP(\mcO_{C}^{2})\to C$
corresponding $\lambda_{l}$.
Let $d_{i}$ be the number of points of the support of the divisor of zeros $(s_{i})_{0}$.
Then we have
\[
\deg \sigma_{l}^{*}\mcO_{\mbP(\mcO_{C}^{2})}(1) -\left| \sigma_{l}^{-1}(D)\right|
\geq l\cdot \deg\mcA -(d_{1}+d_{2}+|T|),
\]
which becomes arbitrarily large as $l\to \infty$.

\textbf{Subcase 1b}.
We consider the remaining case of Case 1.

Choose an open subset $U\subset C\setminus T$
and an isomorphism $\mcE|_{U}\simeq \mcO_{U}^{2}$
in such a way that by the isomorphism $g:\pi^{-1}(U)\to \mbP(\mcO_{U}^{2})=\mbP^{1}\times U$ 
induced by the isomorphisms of sheaves,
we have $g(D_{1}|_{\pi^{-1}(U)})=\{0\}\times U$
and $g(D_{2}|_{\pi^{-1}(U)})=\{\infty\}\times U$.
For our purpose, we may enlarge $T$, so we assume that $C\setminus U=T$.
Put $D'_{1}=\{0\}\times C$, $D'_{2}=\{\infty\}\times C$
and $D'_{3}=\sum_{t\in T}\pi'^{-1}(t)$,
where $\pi':\mbP(\mcO_{C}^{2})\to C$ is the projection.
We denote also by $g$ the rational map $ \mbP(\mcE) \cdots\!\to \mbP(\mcO_{C}^{2}) $
given by $g:\pi^{-1}(U)\to \mbP(\mcO_{U}^{2})$. 

By resolving the indeterminacy of the rational map $g$,
we obtain a smooth projective surface $\widetilde{R}$ and  birational projective morphisms $\tau:\widetilde{R}\to \mbP(\mcE)$
and $\tau':\widetilde{R}\to \mbP(\mcO_{C}^{2}) $ such that $g\circ \tau =\tau'$.
Put $\tilde{\pi}:=\pi\circ\tau=\pi'\circ\tau':\widetilde{R}\to C$.
For $i=1,2$, let $\widetilde{D}_{i}\subset \widetilde{R}$ be the Zariski closure of $(\tau|_{\tilde{\pi}^{-1}(U)})^{-1}(D_{i})$.
Since $\left.\tau^{*}\mcO_{\mbP(\mcE)}(1)\right|_{\tilde{\pi}^{-1}(U)}
\simeq
\left.\tau'^{*}\mcO_{\mbP(\mcO_{C}^{2})}(1)\right|_{\tilde{\pi}^{-1}(U)}$,
there exists a divisor $F$ on $\widetilde{R}$
such that $\mathrm{Supp}\,F\subset \tilde{\pi}^{-1}(T)$
and 
$\tau^{*}\mcO_{\mbP(\mcE)}(1)
\simeq
\tau'^{*}\mcO_{\mbP(\mcO_{C}^{2})}(1)\otimes \mcO_{\widetilde{R}}(F).$
Take a large integer $m$ such that $F+m\tilde{\pi}^{*}(T)$ is effective.
Let $\sigma':C\to \mbP(\mcO_{C}^{2})$ be a section of $\pi'$.
There exists a section $\tilde{\sigma}:C\to \widetilde{R}$ of $\tilde{\pi}$ such
that $\tau'\circ\tilde{\sigma}=\sigma'$.
Then $\sigma:=\tau\circ\tilde{\sigma}$ is a section of $\pi$.
We have
\begin{align*}
\deg\sigma^{*} \mcO_{\mbP(\mcE)}(1) -\left|
\sigma^{-1}(D)
\right| &=
\deg\tilde{\sigma}^{*}\tau^{*} \mcO_{\mbP(\mcE)}(1) 
-|T|-|\sigma^{-1}(D)\cap U| \\
&= \deg\tilde{\sigma}^{*}\tau'^{*} \mcO_{\mbP(\mcO^{2})}(1) 
 +\deg\tilde{\sigma}^{*}\mcO_{\widetilde{R}}(F) \\
& \qquad -|T|-|\sigma'^{-1}(D')\cap U| \\
& \geq \deg\sigma'^{*}\mcO_{\mbP(\mcO^{2})}(1) 
-|\sigma'^{-1}(D')|-m|T|.
\end{align*}
By Subcase 1a, we can choose $\sigma'$ so that 
$\deg\sigma'^{*}\mcO_{\mbP(\mcO^{2})}(1) 
-|\sigma'^{-1}(D')|$ is arbitrarily large,
hence
$\deg\sigma^{*} \mcO_{\mbP(\mcE)}(1) -\left|
\sigma^{-1}(D)
\right| $ also can be arbitrarily large.

\textbf{Case 2}.
Finally we need to consider the case $D=D_{0}+D_{3}$,
where $D_{0}$ is an irreducible curve mapping $2:1$ to $C$,
and $D_{3}=\sum_{t\in T}\pi^{-1}(t)$
for a finite set $T\subset C$.
Let $\mu:C_{0}\to D_{0}$ be the normalization of $D_{0}$,
and put $\lambda:=(\pi|_{D_{0}})\circ\mu:C_{0}\to C$.
Let $\mcE_{0}:=\lambda^{*}\mcE$,
and denote the projection $\mbP(\mcE_{0})\to C_{0}$ by $\pi_{0}$,
and the natural map $\mbP(\mcE_{0})\to \mbP(\mcE)$ by $\tilde{\lambda}$.
For our purpose, we may enlarge $T$.
So by enlarging $T$,
we may assume that $\lambda|_{C_{0}\setminus \lambda^{-1}(T)}: 
C_{0}\setminus \lambda^{-1}(T)\to C\setminus T$
is \'{e}tale.
We have $\tilde{\lambda}^{-1}(D_{0})=D'_{1}+D'_{2}$,
where $D'_{1}, D'_{2} \subset \mbP(\mcE_{0})$ map isomorphically to $C_{0}$
by $\pi_{0}$.
Put $D'_{3}:=\tilde{\lambda}^{-1}(D_{3})=\sum_{t'\in T'}\pi_{0}^{-1}(t')$,
where $T':=\lambda^{-1}(T)$
Then $D':=\tilde{\lambda}^{-1}(D)$ is $D'_{1}+D'_{2}+D'_{3}$.
By Case 1,
we can find a section $\sigma:C_{0}\to \mbP(\mcE_{0})$ of $\pi_{0}$
in such a way that 
$$\deg\sigma^{*}\mcO_{\mbP(\mcE_{0})}(1)
-|\sigma^{-1}(D')|$$
is arbitrarily large.
Put $C':=\tilde{\lambda}\left(
\sigma(C_{0})
\right)$
and let $\nu: \widetilde{C'} \to C'$ be the normalization.
The morphism
$\tilde{\lambda}|_{\sigma(C_{0})}:\sigma(C_{0})\to C'$
factors as
$$
\sigma(C_{0})\xrightarrow{\xi}
\widetilde{C'} \xrightarrow{\nu} C'.
$$
Since $C_{0}\setminus \lambda^{-1}(T) \to C\setminus T$
is \'{e}tale,
$\sigma(C_{0})\setminus \pi_{0}^{-1}(T')
\to  \widetilde{C'}\setminus \pi^{-1}(T)$
is also \'{e}tale. 
So we have
$$
\left|
(\sigma(C_{0})\cap D')\setminus D'_{3}
\right|
=(\deg\xi)\cdot\left|
\nu^{-1}(D\setminus D_{3})\right|.
$$
Using this,
we have
\begin{align*}
\left|
\nu^{-1}(D)
\right| 
&= \left| \nu^{-1}(D\setminus D_{3})\right|+\left| \nu^{-1}(D_{3})\right| \\
& \leq \frac{1}{\deg\xi}\left|
(\sigma(C_{0})\cap D')\setminus D'_{3}
\right|
+
\frac{\deg \lambda}{\deg \xi}
|T|
\\
&
=\frac{1}{\deg\xi}
\left\{
\left|\sigma^{-1}(D')\right|
-|T'|+\deg\lambda\cdot |T|
\right\}.
\end{align*}
Then we have
\begin{align*}
\deg\nu^{*}\mcO_{\mbP(\mcE)}(1)-|\nu^{-1}(D)|
-2g(\widetilde{C'})+2  \\
\geq 
\frac{1}{\deg\xi}
\left\{
\sigma^{*}\mcO_{\mbP(\mcE_{0})}(1)-|\sigma^{-1}(D')|
-\deg\lambda\cdot |T|+|T'|
\right\}-2g(C_{0})+2
\end{align*}
Noting $1\leq \deg \xi \leq \deg\lambda$,
we see that this can be arbitrarily large
because
$\deg\sigma^{*}\mcO_{\mbP(\mcE_{0})}(1)
-|\sigma^{-1}(D')|$ can be arbitrarily large.
\endproof
%
\subsection{Hypersurface case}\label{subsection-Hypersurface}
The main theorem of this paper concerns the Lang-Vojta conjecture
over function fields for very general log projective spaces.
Extending the argument in \cite{CR} to the relative situation over
a curve $C$, we can obtain a result for very general hypersurfaces.

Let $C$ be a smooth projective curve over $\mbC$.
Let $\mcL$ be a globally generated line bundle of degree $l$ on $C$.
Let $W\subset \mrH^{0}(C,\mcL)$ be a subspace
such that the natural map $W\otimes \mcO_{C}\to \mcL$ is surjective.
For $\sigma\in \Sigma=\mbP_{*}\left(
\mrH^{0}(\Pn, \mcO(d))\otimes W
\right)$,
let $D_{\sigma}\subset \Pn\times C$
be the divisor of zeros of $\sigma$.

\begin{thm}\label{thm-hypersuface}
Assume that $d>\frac{3n+2}{2}$.
Let $\sigma\in \Sigma$ be a very general point.
Then for each non-vertical curve $Y\subset D_{\sigma}$ with
its normalization $\nu:\widetilde{Y}\to Y\subset D_{\sigma}$,
either (i) or (ii) of the following holds:
\begin{enumerate}
 \item[(i)]
 The inequality
 \[
 \frac{1}{e} \deg (pr_{\Pn}\circ \nu)^{*}\mcO_{\Pn}(1)
 \leq \frac{1}{e} \left(2g(\widetilde{Y})-2
 \right)-(2g(C)-2)+c_{0}^{\sharp}
 \]
 holds, where $e$ is the degree of the finite map $Y\to C$,
 and $c_{0}^{\sharp}=l(n-2)+\max\{
 2g(C)-2-l,0
 \}$.
 \item[(ii)] $Y\subset \mathbf{E}_{\sigma}^{\sharp}$,
 there $\mathbf{E}_{\sigma}^{\sharp}\subset D_{\sigma}$ is the Zariski
 closure of $\displaystyle \bigcup_{D_{\sigma}\supset l:\text{vertical line}}l$.
\end{enumerate}
\end{thm}
The proof of this theorem is similar to that of Theorem \ref{thm-main}.


\end{document}